\documentclass{amsart}

\usepackage{amsmath,amsthm,amssymb}
\usepackage{enumitem}
\usepackage[draft=false]{hyperref}
\usepackage{xcolor}
\hypersetup{
	colorlinks,
	linkcolor={red!50!black},
	citecolor={blue!50!black},
	urlcolor={blue!80!black}
}
\usepackage{verbatim}

\newtheorem{thm}{Theorem}
\newtheorem{cor}[thm]{Corollary}

\newtheorem{theorem}{Theorem}[section]
\newtheorem{lemma}[theorem]{Lemma}
\newtheorem{proposition}[theorem]{Proposition}
\newtheorem{corollary}[theorem]{Corollary}
\newtheorem{fact}[theorem]{Fact}

\newtheorem*{claim*}{Claim}

\theoremstyle{definition}
\newtheorem{definition}[theorem]{Definition}
\newtheorem{question}[theorem]{Question}
\newtheorem{example}[theorem]{Example}

\theoremstyle{remark}
\newtheorem{remark}[theorem]{Remark}

\newenvironment{proofclaim}{
\begin{proof}[Proof of Claim.]
\renewcommand{\qed}{\hfill$\square_{\mbox{ Claim}}$}
}{\end{proof}}

\def\C{\mathfrak{C}}
\def\M{\mathcal{M}}
\def\qf{\mathrm{qf}}
\DeclareMathOperator{\tp}{tp}
\DeclareMathOperator{\EL}{EL}
\DeclareMathOperator{\NF}{NF}
\renewcommand{\Im}{\mathrm{Im}}
\DeclareMathOperator{\Inv}{Inv}
\DeclareMathOperator{\Aut}{Aut}
\DeclareMathOperator{\Emb}{Emb}
\DeclareMathOperator{\Age}{Age}
\DeclareMathOperator{\Sym}{Sym}
\DeclareMathOperator{\Gal}{Gal}
\DeclareMathOperator{\AutfSh}{Autf_{Sh}}
\DeclareMathOperator{\AutfKP}{Autf_{KP}}
\DeclareMathOperator{\AutfL}{Autf_L}
\DeclareMathOperator{\GalSh}{Gal_{Sh}}
\DeclareMathOperator{\GalKP}{Gal_{KP}}
\DeclareMathOperator{\GalL}{Gal_L}

\DeclareMathOperator{\cl}{cl}
\DeclareMathOperator{\J}{\mathcal{J}}
\DeclareMathOperator{\ct}{ct}
\DeclareMathOperator{\equivSh}{\equiv_{Sh}}
\DeclareMathOperator{\equivKP}{\equiv_{KP}}
\DeclareMathOperator{\equivL}{\equiv_{L}}
\DeclareMathOperator{\id}{id}

\DeclareMathOperator{\acl}{acl}
\DeclareMathOperator{\dcl}{dcl}
\DeclareMathOperator{\Th}{Th}
\DeclareMathOperator{\Stab}{Stab}

\def\invlim{\underleftarrow{\lim}}

\def\eerp{$EERP$}
\def\eecrp{$EECRP$}
\def\deerp{$DEERP$}
\def\deecrp{$DEECRP$}
\def\edeerp{$EDEERP$}
\def\eerdeg{$EERdeg$}

\def\edeerdeg{$EDEERdeg$}

\title{Ramsey theory and topological dynamics for first order theories}
\author[K. Krupi\'nski]{Krzysztof Krupi\'nski}
\address[K. Krupi\'nski]{Uniwersytet Wroc\l awski, Instytut Matematyczny, Plac Grunwaldzki 2/4, 50-384 Wroc\l aw, Poland}
\email{kkrup@math.uni.wroc.pl}
\thanks{All authors are supported by National Science Center, Poland, grant 2016/22/E/ST1/00450. The first author is also supported by National Science Center, Poland, grant 2018/31/B/ST1/00357. The third author is also supported by the Ministry of Education, Science and Technological Development of Serbia.}
\author[J. Lee]{Junguk Lee}
\address[J. Lee]{KAIST, Department of Mathematical Sciences, 291 Daehak-ro Yuseong-gu, 34141, Daejeon, South Korea}
\email{ljwhayo@kaist.ac.kr}
\author[S. Moconja]{Slavko Moconja}
\address[S. Moconja]{University of Belgrade, Faculty of Mathematics, Studentski trg 16, 11000 Belgrade, Serbia}
\email{slavko@matf.bg.ac.rs}



\begin{document}

\begin{abstract}
We investigate interactions between Ramsey theory, topological dynamics, and model theory. We introduce various Ramsey-like properties for first order theories and characterize them in terms of the appropriate dynamical properties of the theories in question (such as  [extreme] amenability of a theory or some properties of the associated Ellis semigroups). Then we relate them to profiniteness and triviality of the Ellis groups of first order theories. In particular, we find various criteria for [pro]finiteness and for triviality of the Ellis group of a given theory from which we obtain wide classes of examples of theories with [pro]finite or trivial Ellis groups. 
We also find several concrete examples illustrating the lack of implications between some fundamental properties. In the appendix, we give a full computation of the Ellis group of the theory of the random hypergraph with one binary and one 4-ary relation. This example shows that the assumption of NIP in the version of Newelski's conjecture for amenable theories (proved in \cite{KNS}) cannot be dropped.
\end{abstract}

\maketitle

\section{Introduction}

In their seminal paper \cite{KPT}, Kechris, Pestov and Todor\v cevi\'c discovered surprising interactions between dynamical properties of the group of automorphisms of a Fra\"iss\'e structure and Ramsey-theoretic properties of its age. For example, they proved that this group is extremely amenable iff the age has the structural Ramsey property and consists of rigid structures (equivalently, the age has the embedding Ramsey property in the terminology used by Zucker in \cite{Z}). This started a wide area of research of similar phenomena.  Recently, Pillay and the first author \cite{KP} gave a model-theoretic account for the fundamental results of Kechris-Pestov-Todor\v cevi\'c (shortly KPT)  theory, generalizing the context to arbitrary, possibly uncountable, structures. 
However, KPT theory (including such generalizations) is not really about model-theoretic properties of the underlying theory, because: on the dynamical side, it talks about the topological dynamics of the topological group of automorphisms of a given structure, which can be expressed in terms of the action of this group on the universal ambit rather than on type spaces of the underlying theory, and, on the Ramsey-theoretic side, it considers arbitrary colorings (without any definability properties) of the finite subtuples of a given model. 
Definitions of Ramsey properties for a  given structure stated in \cite{KP} suggest the corresponding definitions for first order theories just by applying them to a monster model. In this paper, we go much further and define various ``definable'' versions of Ramsey properties for first order theories by restricting the class of colorings to ``definable'' ones. And then we find the appropriate dynamical characterizations of our ``definable'' Ramsey properties in terms of the dynamics of the underlying theory (in place of the dynamics of the group of automorphisms of a given model) some of which are surprising and different comparing to classical KPT theory. 

The classes of amenable and extremely amenable theories introduced and studied in \cite{HKP} are defined in a different way than typical Shelah-style, combinatorially defined classes of theories (such as NIP, simple, $\textrm{NTP}_2$). In this paper, we give Ramsey-theoretic characterizations of [extremely] amenable theories; these characterizations are clearly combinatorial, but still of different flavor than Shelah's definitions. Also, the new classes of theories introduced in this paper via some Ramsey-theoretic properties or via their dynamical characterizations do not follow the usual  Shelah-style way of defining new classes of theories. This makes the whole topic rather novel in model theory. 

We find the interaction between ``definable'' Ramsey properties and the dynamics of first order theories natural and interesting in its own right. However, our original motivation to introduce the ``definable'' Ramsey properties has some specific origins in model theory and topological dynamics in model theory, which we explain in the next paragraph.

Some methods of topological dynamics were introduced to model theory by Newelski in \cite{Ne1}. Since then a wide research on this topic has been done by Chernikov, Hrushovski, Newelski, Pillay, Rzepecki, Simon, the first author, and others. For any given theory $T$, a particularly important place in this research is reserved for the investigation of the flow $(\Aut(\C),S_{\bar c}(\C))$, where $\bar c$ is an enumeration of a monster model $\C \models T$ and $S_{\bar c}(\C)$ is the space of global types extending $\tp(\bar c/\emptyset)$, as it turns out that topological properties of this flow carry important information about the underlying theory. In particular, in \cite{KPRz} it was proved that there exists a topological quotient epimorphism from the Ellis group of the flow $(\Aut(\C),S_{\bar c}(\C))$ (also called the Ellis group of $T$, as it does not depend on the choice of the monster model  by \cite{KNS}) to $\GalKP(T)$ (the Kim-Pillay Galois group of $T$), and even to the larger group $\GalL(T)$ (the Lascar Galois group of $T$); in particular, the Ellis group of $T$ captures more information about $T$ than the Galois groups of $T$.  This was the starting point for the research in this paper. Namely, from the aforementioned result from \cite{KPRz} one easily deduces that profiniteness of the Ellis group implies profiniteness of $\GalKP(T)$, which in turn is known to be equivalent to the equality of the Shelah and Kim-Pillay strong types. The question for which theories the Shelah and Kim-Pillay strong types coincide is fundamental in model theory 
(note that it can be viewed as ``elimination of bounded hyperimaginaries''; in the context of simple theories, it is equivalent to the independence theorem being true over the (imaginary rather than hyperimaginary) algebraically closed sets). This is known to be true in e.g.\ stable or supersimple theories, but remains a well-known open question in simple theories in general. This led us to the question for which theories the Ellis group is profinite, which is also interesting in its own rights (keeping in mind that the Ellis group of $T$ captures more information than any of the Galois groups of $T$). And among the main outcomes of this paper are results saying that various Ramsey-like properties of $T$ imply profiniteness of the Ellis group.



Let us briefly discuss the Ramsey properties which we investigate in this paper. They are given with respect to a monster model $\C$ of a first-order theory $T$, but we will show that they do not depend on the choice of $\C$, so they are really properties of $T$. We say that $T$ has {\em separately finite elementary embedding Ramsey degree (sep.\ fin.\ \eerdeg)} if for every finite $\bar a\subseteq\C$ there exists $l<\omega$ such that for every finite $\bar b\supseteq\bar a$, $r<\omega$, and coloring $c:{\C\choose\bar a}\to r$ there exists $\bar b'\in{\C\choose\bar b}$ such that $\#c[{\bar b'\choose\bar a}]\leqslant l$. Here, for a tuple $\bar a$ and a set $B$ (or a tuple which is treated as the set of coordinates), ${B\choose\bar a}$ denotes the set of all $\bar a'\subseteq B$ such that $\bar a'\equiv\bar a$. If $l$ above can be taken to be $1$ for every finite $\bar a$, we say that $T$ has the {\em elementary embedding Ramsey property (\eerp)}. If $l$  can be taken to be $1$ and we restrict ourselves to considering only [externally] definable colorings (see Section \ref{section ramsey} for definitions), we say that $T$ has the {\em [externally] definable elementary embedding Ramsey property ([E]DEERP)}. 
If for every finite set of formulae $\Delta$ and every finite $\bar a$ the above holds (for some $l$)  for the externally definable $\Delta$-colorings, then we say that $T$ has {\em separately finite externally definable elementary embedding Ramsey degree (sep.\ fin.\ \edeerdeg)}. 

Theories with \eerp\ and sep.\ fin.\ \eerdeg\ are generalizations of the classical notions of embedding Ramsey property and finite embedding Ramsey degree in the following sense: If $K$ is an $\aleph_0$-saturated Fra\"iss\'e structure, then its age has the embedding Ramsey property [sep.\ fin.\ embedding Ramsey degree] iff $\mathrm{Th}(K)$ has \eerp\ [sep.\ fin.\ \eerdeg].

We also consider the following convex Ramsey-like properties. We say that $T$ has the {\em elementary embedding convex Ramsey property (\eecrp)} if for every $\epsilon\geqslant 0$ and finite $\bar a\subseteq\bar b\subseteq\C$, $n<\omega$, and coloring $c:{\C\choose\bar a}\to 2^n$ there exist $k<\omega$, $\lambda_0,\dots,\lambda_{k-1}\in[0,1]$ with $\lambda_0+\dots+\lambda_{k-1}=1$, and $\sigma_0,\dots,\sigma_{k-1}\in\Aut(\C)$ such that for any two tuples $\bar a',\bar a''\in{\bar b\choose\bar a}$ the convex combinations $\sum_{j<k}\lambda_jc(\sigma_j(\bar a'))(i)$ and $\sum_{j<k}\lambda_jc(\sigma_j(\bar a''))(i)$ differ by at most $\epsilon$ for every $i<n$. If we restrict ourselves to definable colorings, we say that $T$ has the {\em definable elementary embedding convex Ramsey property (\deecrp)}.


To state our main results, we need to use a natural refinement of the usual space of $\Delta$-types, denoted by $S_{\bar c, \Delta}(\bar p)$ for a finite set of formulae $\Delta=\{\varphi_i(\bar x, \bar y)\}_{i<k}$ and a finite set (or sequence) of types $\bar p =\{p_j(\bar y)\}_{j<m} \subseteq S_{\bar y}(\emptyset)$. (It is defined after Lemma \ref{lemma ellis semi epi bigger lang invariant}.) For a flow $(G,X)$, by $EL(X)$ we denote the Ellis semigroup of this flow. By $\Inv_{\bar c}(\C)$, we denote the space of global invariant types extending $\tp(\bar c/\emptyset)$. (All these notations and definitions can be found in Section \ref{section prelim}.) Our main result yields dynamical characterizations of the introduced Ramsey properties.

\begin{thm}\label{THM} Let $T$ be a complete first-order theory and $\C$ its monster model. Then:
\begin{enumerate}[label=(\roman*),align=right,leftmargin=*]
\item $T$ has \deerp\ iff $T$ is extremely amenable (in the sense of \cite{HKP}).
\item $T$ has \edeerp\ iff there exists $\eta\in\EL(S_{\bar c}(\C))$ such that $\Im(\eta)\subseteq \Inv_{\bar c}(\C)$.
\item $T$ has sep.\ fin.\ \edeerdeg\ iff for every finite set of formulae $\Delta$ and finite sequence of types $\bar p$ there exists $\eta\in\EL(S_{\bar c,\Delta}(\bar p))$ such that $\Im(\eta)$ is finite.
\item $T$ has \deecrp\ iff $T$ is amenable (in the sense of \cite{HKP}).
\end{enumerate}
\end{thm}

A relation of this to the Ellis group of the theory is given by the next corollary.

\begin{cor}\phantomsection\label{the main corollary of the main theorem}
\begin{enumerate}[label=(\roman*), align=left, leftmargin=*, labelsep=-4pt]
\item A theory with \edeerp\ has trivial Ellis group (see Cor. \ref{corollary dpeerp implies trivial ellis}).
\item A theory with sep.\ fin.\ \edeerdeg\ has profinite Ellis group (see Cor. \ref{corollary sep fin dpeerdeg implies 0-dim ellis}).
\end{enumerate}
\end{cor}

Item (i) is an easy consequence of Theorem \ref{THM}(ii). Item (ii) follows from Theorem \ref{THM}(iii) and the implication (D) $\Longrightarrow$ (A) in Theorem \ref{THMsecond} below. 

In the next theorem, $\M$ denotes a minimal left ideal in $\EL(S_{\bar c}(\C))$ and $u$ an idempotent in this ideal, so $u\M$ is the Ellis group of $T$; $u\M/H(u\M)$ is the canonical Hausdorff quotient of $u\M$ (see Section \ref{section prelim}). Analogously, $u_{\Delta,\bar p}\M_{\Delta,\bar p}$ is the Ellis group of the flow $(\Aut(\C), S_{\bar c,\Delta}(\bar p))$. The main idea behind the next result is that a natural way to obtain that the Ellis group of $T$  is profinite is to present the flow $S_{\bar c}(\C)$ as the inverse limit of some flows each of which has finite Ellis group, and if it works, it should also work for the natural presentation of $S_{\bar c}(\C)$ as the inverse limit of the flows $S_{\bar c, \Delta}(\bar p)$ (where $\Delta$ and $\bar p$ vary).

\begin{thm}\label{THMsecond}
Consider the following conditions:
\begin{enumerate}[label=(\Alph*), align=right, leftmargin=*]
\item[(A")] $\GalKP(T)$ is profinite;

\item[(A')] $u\M/H(u\M)$ is profinite;

\item $u\M$ is profinite;

\item 
the $\Aut(\C)$-flow $S_{\bar c}(\C)$ is isomorphic to the inverse limit $\underleftarrow{\lim}_{i\in I}X_i$ of some $\Aut(\C)$-flows $X_i$ each of which has finite Ellis group;

\item for every finite sets of formulae $\Delta$ and types $\bar p\subseteq S(\emptyset)$, $u_{\Delta,\bar p}\M_{\Delta,\bar p}$ is finite;

\item for every finite sets of formulae $\Delta$ and types $\bar p\subseteq S(\emptyset)$, there exists $\eta\in\EL(S_{\bar c,\Delta}(\bar p))$ with $\Im(\eta)$ finite.
\end{enumerate}
Then (D) $\Longrightarrow$ (C)  $\Longleftrightarrow$ (B)  $\Longrightarrow$ (A)  $\Longrightarrow$ (A') $\Longrightarrow$ (A''). 
\end{thm}

We also find several other criteria for [pro]finiteness of the Ellis group. Applying Corollary \ref{the main corollary of the main theorem} or our other criteria together with some well-known theorems from structural Ramsey theory (saying that various Fra\"{i}ss\'{e} classes have the appropriate Ramsey properties), we get wide classes of examples of theories with [pro]finite or sometimes even trivial Ellis groups. But we also find some specific examples illustrating interesting phenomena, e.g.\ we give examples showing that in Theorem \ref{THMsecond}: (A'') does not imply (A'),  and (A') does not imply (B). The example showing that (A'') does not imply (A') is supersimple of SU-rank 1, so it shows that even for supersimple theories the Ellis group of the theory need not be profinite. We have not found examples showing that (C) does not imply (D), and (A) does not imply (B), which we leave as open problems. 

One of our most important examples is Example \ref{example R2 R4} which is analyzed in details in the appendix: we give there a precise computation of the Ellis group of the theory of the random hypergraph with one binary and one 4-ary relation. This group turns out to be $\mathbb Z/2\mathbb Z$.
This example is interesting for various reasons. Firstly, by classical KPT theory, we know that it has sep.\ finite \eerdeg, so the Ellis group is profinite by the above results (in fact, it satisfies the assumptions of some other criteria that we found, which implies that the Ellis group is finite), and the example shows that it may be non-trivial. A variation of this example (see Example \ref{example R2 R4 Ps}) yields an infinite Ellis group, which shows that in some of our criteria for profiniteness, we cannot expect to get finiteness of the Ellis group.  
Example \ref{example R2 R4} also shows that in general sep.\ fin.\ $EERdeg$ does not imply $EDEERP$ and that $DEERP$ does not imply $EDEERP$. Finally, this example is easily seen to be extremely amenable in the sense of \cite{HKP}, so its KP-Galois group is trivial. But the Ellis group is non-trivial. Hence, the epimorphism (found in \cite{KPRz}) from the Ellis group to the KP-Galois group is not an isomorphism. On the other hand, by \cite[Theorem 0.7]{KNS}, we know that under NIP, even amenability of the theory is sufficient for this epimorphism to be an isomorphism. So our example shows that one cannot drop the NIP assumption in \cite[Theorem 0.7]{KNS}, which was not known so far.

Using our observations that both properties \eerp\ and \eecrp\ do not depend on the choice of the monster model, or even an $\aleph_0$-saturated and strongly $\aleph_0$-homogeneous model $M\models T$, and the results from \cite{KP} saying that \eerp\ (defined in terms of $M$) is equivalent to extreme amenability of the topological group $\Aut(M)$, and \eecrp\ (defined in terms of $M$) is equivalent to amenability of $\Aut(M)$, we get the following corollary.


\begin{cor}\label{THM2}
Let $T$ be a complete first-order theory. The group $\Aut(M)$ is [extremely] amenable as a topological group for some $\aleph_0$-saturated and strongly $\aleph_0$-homogeneous model $M\models T$ iff it is [extremely] amenable as a topological group for all $\aleph_0$-saturated and strongly  $\aleph_0$-homogeneous models $M\models T$.
\end{cor}

This means that [extreme] amenability of the group of automorphisms of an $\aleph_0$-saturated and strongly $\aleph_0$-homogeneous structure is actually a property of its theory, which seems to be a new observation.

The paper is organized as follows. 
In Section \ref{section prelim}, we recall or introduce all the necessary notions from model theory, topological dynamics, and classical structural Ramsey theory. 
Furthermore, we prove several new fundamental and useful observations, which for example are helpful to compute Ellis groups  or show some dynamical properties of concrete theories. 
Section \ref{section ramsey} is the central part of the paper. We introduce and characterize all the aforementioned Ramsey properties for first order theories. We prove Theorem \ref{THM} and Corollary \ref{THM2}. In Section \ref{section around}, we prove Theorem \ref{THMsecond} and find some other conditions which imply [pro]finiteness of the Ellis group of the theory.
In Section \ref{section applications}, we give a long list of examples to which our results apply, and find several examples with some specific properties, e.g.\ the aforementioned examples showing the lack of two of the implications between the items of Theorem \ref{THMsecond}.
In the appendix, we give a complete computation of the Ellis group of the theory of the random hypergraph with one binary and one 4-ary symbol.\\

Some ``definable'' versions of Ramsey properties were also introduced and considered in a recent paper by Nguyen Van Th\'{e} \cite{The}. Also, Hrushovski \cite{H} has recently introduced some version of Ramsey properties in a first-order setting.
But all these notions seem to be different and they are introduced for different reasons. It would be interesting to see in the future if there are any relationships.

\section{Preliminaries and fundamental observations}\label{section prelim}

Most of this section consists of definitions, notations  and facts needed in this paper. But there are also new ingredients, especially in Subsection \ref{Subsection: Flows in model theory}, where we obtain some new reductions and introduce the type spaces $S_{\Delta, \bar c}(\bar p)$ playing a key role in this paper.

\subsection{Model theory}\label{subsection model theory}

We use standard model-theoretic concepts and terminology. By a theory we always mean a complete first-order theory $T$ in a language $L$. For simplicity, we will be assuming that $L$ is one-sorted, but the whole theory developed in this paper works almost the same for many-sorted languages. We usually work in a monster model $\C$ of $T$, i.e.\ a $\kappa$-saturated and strongly $\kappa$-homogeneous model of $T$ for a large enough cardinal $\kappa$ (called the degree of saturation of $\C$). Elements of $\C$ are denoted by $a,b,\dots$ and tuples (finite or infinite) of $\C$ are denoted by $\bar a,\bar b,\dots$. By a small set [model] we mean a subset [elementary submodel] of $\C$ of cardinality less than $\kappa$; small subsets of $\C$ are denoted by $A,B,\dots$, and small submodels by $M,N,\dots$. 

Global types are complete types over $\C$. 
By $S_{\bar x}(A)$ we denote the space of all complete types over $A$ in variables $\bar x$; if $A=\emptyset$, we also write $S_{\bar x}(T)$ for $S_{\bar x}(\emptyset)$. For a type $\pi(\bar x)$ over some $B \subseteq A$, $S_{\pi}(A)$ denotes  the subspace of $S_{\bar x}(A)$ consisting of all types extending $\pi(\bar x)$. For a tuple $\bar a$, $S_{\bar a}(A)$ denotes the space of all complete types over $A$ extending $\tp(\bar a)$; in other words, $S_{\bar a}(A) = S_{\tp(\bar a)}(A)$. These spaces are naturally compact, Hausdorff, $0$-dimensional topological spaces. For tuples $\bar a,\bar b$, $\bar a\equiv\bar b$ means that $\bar a$ and $\bar b$ have the same type over $\emptyset$. 

$\Aut(\C)$ is the group of all automorphisms of $\C$, and $\Aut(\C/A)$ is the pointwise stabilizer of $A$.
A subset of a power of $\C$ is invariant [$A$-invariant] if it is invariant under $\Aut(\C)$ [$\Aut(\C/A)$]. Having the same type over $\emptyset$ [small $A$] is the equivalence relation of lying in the same orbit of $\Aut(\C)$ [$\Aut(\C/A)$] on the appropriate power of $\C$. 
By $\equivSh$ we denote the intersection of all $\emptyset$-definable, finite equivalence relations (i.e.\ with finitely many classes) on a given power of $\C$;
the classes of $\equivSh$ are called Shelah strong types. By $\AutfSh(\C)$ we denote the group of all Shelah strong automorphisms of $\C$, i.e.\ all automorphisms of $\C$ fixing all Shelah strong types. An equivalence relation is bounded if it has less than $\kappa$ classes. $\equivKP$  and $\equivL$ are respectively the finest bounded $\emptyset$-type-definable equivalence relation and the finest bounded $\emptyset$-invariant equivalence relation (on a fixed power of $\C$); the classes of $\equivKP$ and $\equivL$ are called Kim-Pillay strong types and Lascar strong types, respectively. By $\AutfKP(\C)$ and $\AutfL(\C)$ we denote respectively the groups of all Kim-Pillay strong automorphism and all Lascar strong automorphisms, i.e.\ automorphisms of $\C$ fixing all $\equivKP$-classes and all $\equivL$-classes, respectively. It turns out that $\equivSh$, $\equivKP$, and $\equivL$ are the orbit equivalence relations of $\AutfSh(\C)$,  $\AutfKP(\C)$, and $\AutfL(\C)$, respectively.

$\AutfSh(\C)$, $\AutfKP(\C)$, $\AutfL(\C)$ are normal subgroups of $\Aut(\C)$, and the corresponding quotients do not depend on the choice of the monster $\C$ and are called respectively the Shelah Galois group, the Kim-Pillay Galois group, and the Lascar Galois group of $T$; we denote them by $\GalSh(T)$, $\GalKP(T)$, and $\GalL(T)$, respectively.
Since  $\AutfL(\C)\leqslant\AutfKP(\C)\leqslant\AutfSh(\C)$, we have natural epimorphisms $\GalL(T)\to\GalKP(T)\to\GalSh(T)$.

All the above Galois groups of $T$ are topological groups. 
The topology on $\GalL(T)$ is defined as follows. Let $M$ be a small model and let $\bar m$ be an enumeration of $M$. The natural projection $\Aut(\C)\to\GalL(T)$ factors through $S_{\bar m}(M)$: $\Aut(\C)\to S_{\bar m}(M)\to \GalL(T)$, and we equip $\GalL(T)$ with the quotient topology induced by $S_{\bar m}(M)\to \GalL(T)$. This does not depend on the choice of the model $M$. It turns out that $\GalL(T)$ is a compact (but not necessarily Hausdorff) topological group. 
The topologies on $\GalSh(T)$ and $\GalKP(T)$ are defined in  similar fashion. $\GalSh(T)$ is a 
profinite group,
and $\GalKP(T)$ is a compact, Hausdorff group. 
The epimorphisms  $\GalL(T)\to\GalKP(T)\to\GalSh(T)$ are  topological quotient maps.

Furthermore, $\GalSh(T)$ is the largest profinite quotient of $\GalKP(T)$. Thus, $\GalKP(T)$ is profinite iff it equals $\GalSh(T)$, and the last condition is clearly equivalent to saying that $\equivKP$ and $\equivSh$ are equal on all powers of $\C$, i.e.\ the Kim-Pillay and Shelah strong types coincide. 

For more details concerning strong types and Galois groups the reader is referred to \cite{CLPZ}, \cite{Zi}, or \cite[Chapter 2.5]{RzPhD}.

\subsection{Topological dynamics}\label{Subsection: Topological dynamics}

We quickly introduce and state some facts from topological dynamics.
As a general reference we can recommend \cite{Aus} and \cite{Gl}.

By a {\em $G$-flow} we mean a pair $(G,X)$ where $G$ is a topological group acting continuously on a (non-empty) compact, Hausdorff space $X$. 
The {\em Ellis semigroup} of a $G$-flow $(G,X)$, denoted by $\EL(X)$, is the closure of the set $\{\pi_g\mid g\in G\}$ in $X^X$ (equipped with the topology of pointwise convergence), where $\pi_g$ is the function given by $x\mapsto gx$, with composition as semigroup operation; this semigroup operation is continuous in the left coordinate. 
$\EL(X)$ itself is a $G$-flow, where the action is defined by $g\eta=\pi_g\circ\eta$ for $g\in G$ and $\eta\in\EL(X)$. 
By abusing notation, we denote $\pi_g$ simply by $g$, treat $G$ as a subset of $\EL(X)$ (although this ``inclusion'' is not necessarily 1-1), and then $\EL(X)=\cl(G)$.
The minimal $G$-subflows of $\EL(X)$ coincide with the minimal left ideals of $\EL(X)$. If $\mathcal M$ is any minimal left ideal of $\EL(X)$, then the set $\mathcal J(\mathcal M)$ of all idempotents in $\mathcal M$ is non-empty. Furthermore, $\mathcal M$ is a disjoint union of subsets $u\mathcal M$ for $u\in\mathcal J(\mathcal M)$. For each $u\in\mathcal J(\mathcal M)$, $u\mathcal M$ is a group with respect to the composition of functions (a subgroup of $\EL(X)$) with neutral $u$. Moreover, the isomorphism type of this group does not depend on the choice of $\mathcal M$ and $u\in\mathcal J(\mathcal M)$, and it is called the {\em Ellis group} of the flow $(G,X)$; abusing terminology, any $u\M$ is also called an (or the) Ellis group of $(G,X)$. 
Note that for every $\eta \in \mathcal M$, $\eta \mathcal M = u \mathcal M$ where $u$ is the unique idempotent with $\eta \in u \mathcal M$.

The existence of an element in the Ellis semigroup with finite image will be one of the key properties in this paper. 
The fundamental observation in this situation is given by the next fact, which follows from Lemmas 4.2 and 4.3 in \cite{KNS}.
The idea of the proof of this fact is that if $\eta \in EL(X)$ has finite image, then for any $u \in \mathcal J(\mathcal M)$ (where $\mathcal M$ is a minimal left ideal in  $\EL(X)$), the element $u \eta u \in u\mathcal M$ also has finite image, and so has $u$. Then one checks that the map $u\M\to \Sym(\Im(u))$ given by the restriction $\tau\mapsto \tau_{\upharpoonright\Im(u)}$ is a monomorphism, so $u \mathcal M$ is finite.

\begin{fact}\label{fact finite image} If there exists $\eta\in\EL(X)$ with $\Im(\eta)$ finite, then the Ellis group is also finite.\qed
\end{fact}



On an Ellis group $u\M$ we have a topology inherited from $\EL(X)$. Besides this topology, a coarser, so-called {\em $\tau$-topology} is defined. First, for $a\in\EL(X)$ and $B\subseteq \EL(X)$ let $a\circ B$ be the set of all limits of the nets $(g_ib_i)_i$ such that $g_i\in G$, $b_i\in B$ and $\lim_ig_i=a$. For $B\subseteq u\M$ we define $\cl_\tau(B)= u\M\cap(u\circ B)$.
$\cl_\tau$ is a closure operator on $u\M$; the {\em $\tau$-topology} is a topology on $u\M$ induced by $\cl_\tau$. $u\M$ with the $\tau$-topology is a compact,  $T_1$ semitopological group (i.e.\ group operation is separately continuous). The isomorphism types of the Ellis groups (for all $\M$ and $u \in \J(M)$) as semitopological groups do not depend on the choice of $u$ and $\M$.  Put $H(u\M)=\bigcap_U\cl_\tau(U)$, where the intersection is taken over all $\tau$-open neighbourhoods of $u$ in $u\M$. This is a $\tau$-closed normal subgroup of $u\M$, and the quotient $u\M/H(u\M)$ is a compact, Hausdorff topological group. Moreover, $H(u\M)$ is the smallest $\tau$-closed normal subgroup of $u\M$  such that $u\M/H(u\M)$ is Hausdorff. $u\M/H(u\M)$ will be called the {\em canonical Hausdorff quotient} of $u\M$.

\begin{remark}\phantomsection\label{remark uM and um/H zero dimensionality}
\begin{enumerate}[label=(\alph*), align=left, leftmargin=*, labelsep=0pt]
\item  If $u\M/H(u\M)$ is profinite, then $u\M$ is profinite iff it is Hausdorff. 
\item $u\M$ is $0$-dimensional iff it is profinite. 
\end{enumerate}
\end{remark}

\begin{proof}
(a) ($\Rightarrow$) is trivial. ($\Leftarrow$) holds, as $u\M$ is Hausdorff iff $H(u\M)$ is trivial.

(b) ($\Leftarrow$) is trivial. For ($\Rightarrow$), since $H(u\M)=\bigcap_{U}\cl_\tau(U)$ with the intersection  taken over all $\tau$-open neighbourhoods $U$ of $u$, we get $H(u\M)\subseteq\bigcap_{U}U$, where the intersection is taken over all $\tau$-clopen neighbourhoods of $u$. Since $u\M$ is $T_1$ and $0$-dimensional, the previous intersection is just $\{u\}$, so $H(u\M)$ is trivial. Hence, $u\M = u\M/H(u\M)$ is a compact, Hausdorff, $0$-dimensional topological group, so it is profinite.
\end{proof}

A mapping $\Phi:X\to Y$ between two $G$-flows $(G,X)$ and $(G,Y)$ is a {\em $G$-flow homomorphism} if it is continuous and for every $g\in G$ and $x\in X$ we have $\Phi(gx)= g\Phi(x)$. A surjective [bijective] $G$-flow homomorphism is a {\em $G$-flow epimorphism} [{\em $G$-flow isomorphism}]; note that a $G$-flow epimorphism is necessarily a topological quotient map, and an inverse of a $G$-flow isomorphism is necessarily a $G$-flow isomorphism. 


The proof of the next fact is basically the argument from the proof of \cite[Proposition 5.41]{RzPhD}, so we will not repeat it.
\begin{fact}\label{fact epi ellis} Let $(G,X)$ and $(G,Y)$ be two $G$-flows, and let $\Phi:\EL(X)\to \EL(Y)$ be a $G$-flow and semigroup epimorphism. Let $\M$ be a minimal left ideal of $\EL(X)$ and $u\in\J(\M)$. Then:
\begin{enumerate}[label=(\roman*), align=right, leftmargin=*]
\item $\M':=\Phi[\M]$ is a minimal left ideal of $\EL(Y)$ and $u':=\Phi(u)\in\J(\M')$;
\item $\Phi_{\upharpoonright u\M}:u\M\to u'\M'$ is a group epimorphism and a quotient map in the $\tau$-topologies.\qed
\end{enumerate}
\end{fact}

\begin{corollary}\label{corollary tau iso ellis} Let $(G,X)$ and $(G,Y)$ be two $G$-flows, and let $\Phi:\EL(X)\to \EL(Y)$ be a $G$-flow and semigroup isomorphism. Let $\M$ be a minimal left ideal of $\EL(X)$ and $u\in\J(\M)$. Then $\Phi_{\upharpoonright u\M}:u\M\to \Phi(u)\Phi[\M]$ is a group isomorphism and a homeomorphism (in the $\tau$-topologies).\qed
\end{corollary}

A natural way of obtaining a $G$-flow and semigroup epimorphism $\EL(X)\to \EL(Y)$ is to induce it from a $G$-flow epimorphism $X\to Y$ or a $G$-flow monomorphism $Y\to X$. This is explained in the next fact whose proof is left as an exercise.

\begin{fact}\label{fact induced epimorphism} 
Let $(G,X)$ and $(G,Y)$ be $G$-flows.
\begin{enumerate}[label=(\roman*), align=right, leftmargin=*]
\item  If $f:X\to Y$ is a $G$-flow epimorphism, then it induces a $G$-flow and semigroup epimorphism $\hat f:\EL(X)\to\EL(Y)$ given by: $$\hat f(\eta)(y)= f(\eta(x))$$ for $\eta\in\EL(X)$ and $y\in Y$, where $x\in X$ is any element such that $f(x)=y$. 
\item If $f:Y\to X$ is a $G$-flow monomorphism, then it induces a $G$-flow and semigroup epimorphism $\hat f:\EL(X)\to\EL(Y)$ given by:
$$\hat f(\eta)(y)= f^{-1}(\eta(f(y)))$$
for $\eta\in\EL(X)$ and $y\in Y$, where $f^{-1}:\Im(f)\to Y$ is the inverse of $f$.\qed
\end{enumerate}
\end{fact}


\begin{corollary}\label{corollary of fact 2.4}
Let $f:X\to Y$ be a $G$-flow epimorphism and let $\hat f:\EL(X)\to\EL(Y)$ be the induced epimorphism given by Fact \ref{fact induced epimorphism}(i). Then, for any $\eta \in \EL(X)$, $\Im(\hat{f}(\eta)) = f[\Im(\eta)]$. Thus, if $\Im(\eta)$ is finite, so is $\Im(\hat{f}(\eta))$. \qed
\end{corollary}

\begin{corollary}\label{corollary isomorphism of ellis groups having good eta}
If $(G,Y)$ is a subflow of $(G,X)$ and there is an element $\eta \in \EL(X)$ with $\Im(\eta) \subseteq Y$, then the Ellis groups of the flows $X$ and $Y$ are topologically isomorphic.
\end{corollary}

\begin{proof}
Let $f =\id: Y \to X$. Then the map $\hat{f}: \EL(X) \to \EL(Y)$ from Fact \ref{fact induced epimorphism}(ii) is the restriction to $Y$. Let $\M$ be a minimal left ideal of $\EL(X)$. Replacing $\eta$ by any element of $\eta\M$, we can assume that $\eta \in \M$. Let $u \in \J(\M)$ satisfy $\eta \in u\M$. Then $\Im(u) = \Im(\eta) \subseteq Y$. Let $\M' = \hat{f}[\M]$ and $u' = \hat{f}(u)$. By Fact \ref{fact epi ellis}, $u'\M'$ is the Ellis group of $(G,Y)$ and $\hat{f}_{\upharpoonright u\M}: u\M \to u'\M'$ is an epimorphism and a topological quotient map. 
Finally, $\hat{f}_{\upharpoonright u\M}$ is injective, as for $\tau_1,\tau_2 \in u\mathcal M$: $\tau_{1\upharpoonright Y}= \tau_{2\upharpoonright Y} \Longrightarrow \tau_{1\upharpoonright \Im(u)}= \tau_{2\upharpoonright \Im(u)} \iff \tau_1 u=\tau_2 u \iff \tau_1=\tau_2$.
\end{proof}

We further consider an inverse system of $G$-flows $((G,X_i))_{i\in I}$, where $I$ is a directed set, and for each $i<j$ in $I$ a $G$-flow epimorphism $\pi_{i,j}:X_j\to X_i$ is given. Then $X:=\invlim_{i\in I}X_i$ is a compact, Hausdorff space and $G$ acts naturally and continuously on $X$. Denote by $\pi_i:X\to X_i$ the natural projections; they are $G$-flow epimorphisms. By Fact \ref{fact induced epimorphism}, we have $G$-flow and semigroup epimorphisms $\hat\pi_{i,j}:\EL(X_j)\to \EL(X_i)$ for $i<j$ and $\hat \pi_i:\EL(X)\to \EL(X_i)$. 
It turns out that the previous inverse limit construction transfers to Ellis semigroups, and furthermore to Ellis groups. We state this in the next fact; for the proof see \cite[Lemma 6.42]{RzPhD}.

\begin{fact}\label{fact inverse limit and ellis} 
Fix the notation from the previous paragraph. 

The family $(\EL(X_i))_{i\in I}$, together with the mappings $\hat\pi_{i,j}$, is an inverse system of semigroups and of $G$-flows. There exists a natural $G$-flow and semigroup isomorphism $\EL(X)\cong\invlim_{i\in I}\EL(X_i)$, and after identifying $\EL(X)$ with  $\invlim_{i\in I}\EL(X_i)$, the natural projections of this inverse limit are just the $\hat \pi_i$'s.

For every minimal left ideal $\M$ of $\EL(X)$, each $\M_i:=\hat \pi_i[\M]$ is a minimal left ideal of $\EL(X_i)$ and $\M=\invlim_{i\in I}\M_i$. Also, if $u\in\J(\M)$, then $u_i:=\hat \pi_i(u)\in\J(\M_i)$.
Furthermore, $u\M= \invlim_{i\in I}u_i\M_i$ and the $\tau$-topology on $u\M$ coincides with the inverse limit topology induced from the $\tau$-topologies on the $u_i\M_i$'s.\qed
\end{fact}

The material in the rest of this subsection will be needed in the analysis of Examples \ref{example equivalences with two classes} and \ref{example equivalences n classes}. 

Recall that a $G$-ambit is a $G$-flow $(G,X,x_0)$ with a distinguished point $x_0$ with dense orbit.  It is well-known that for any discrete group $G$, $(G,\beta G, \mathcal{U}_e)$ is the universal $G$-ambit, where $\mathcal{U}_e$ is the principal ultrafilter concentrated on $e$; we will identify $e$ with $\mathcal{U}_e$. This easily yields a unique left continuous semigroup operation on $\beta G$ extending the action of $G$ on $\beta G$, and for any $G$-flow $(G,X)$ this also yields a unique left continuous action $*$ of the semigroup $\beta G$ on $(G,X)$ extending the action of $G$. Thus, there  is a unique $G$-flow and semigroup epimorphism $\beta G \to \EL(X)$ (mapping $e$ to the identity): it is given by $p \mapsto l_p$, where $l_p(x):=p * x$. Applying this to $X: = \beta G$, we get an isomorphism $\beta G \cong \EL(\beta G)$. For some details on this see \cite[Chapter 1]{Gl}.

 It is also well-known (see \cite[Exercise 1.25]{Gl2}) that for the Bernoulli shift $X:=2^G$ the above map $\beta G \to \EL(X)$ is an isomorphism of flows and of semigroups, i.e.\ $\beta G \cong \EL(2^G)$. 

\begin{proposition}\label{prop:ellisgroup_associated_bernoullishift}
Let $(G,X)$ be a $G$-flow for which there is a $G$-flow epimorphism $\varphi: X\rightarrow 2^G$. Then,
\begin{enumerate}[label=(\roman*), align=right, leftmargin=*]
	\item $\EL(X)\cong \beta G$ as $G$-flows and as semigroups.
	\item The Ellis group of $X$ is topologically isomorphic with the Ellis group of $\beta G$.
\end{enumerate}
\end{proposition}
\begin{proof}
(i) By Fact \ref{fact induced epimorphism}(i), we have the induced $G$-flow and semigroup epimorphism $\hat{\varphi}: \EL(X) \to \EL(2^G)$.
On the other hand, by the above discussion, there is a   $G$-flow and semigroup epimorphism $\psi : \beta G \to \EL(X)$. So the composition $\hat{\varphi} \circ \psi: \beta G \to \EL(2^G)$ is a $G$-flow and semigroup epimorphism. 
By the above discussion, it is an isomorphism.
Therefore, $\psi$ must be an isomorphism, too.\\
(ii) follows from (i) and Corollary \ref{corollary tau iso ellis}.
\end{proof}

Recall that the Bohr compactification of a topological group $G$ is a unique (up to isomorphism) universal object in the category of group compactifications of $G$, i.e.\ continuous homomorphisms $G \to H$ with dense image, where $H$ is a compact, Hausdorff group. We often identify the Bohr compactification with the target compact group, and denote it by $bG$. For a discrete group $G$, a minimal left ideal $\M$ in $\beta G$ and an idempotent $u \in \M$, the quotient $u\M/H(u\M)$ turns out to be the generalized Bohr compactification of $G$ in the terminology from \cite{Gl}. There is always a continuous epimorphism from the generalized Bohr compactification to the Bohr compactification.
We will need the following result, which is Corollary 4.3 of \cite{Gl}. 
(See also \cite{KP-1} for a more general, model-theoretic account.) We do not recall here the notion of strongly amenable group. We only need to know that abelian groups are strongly amenable. For more details see \cite{Gl}.

\begin{fact}\label{fact:bohr_compactification_tau_topology}
If a discrete group $G$ is strongly amenable (e.g.\ $G$ is abelian), then the generalized Bohr compactification of $G$ and the Bohr compactification of $G$ coincide.\qed
\end{fact}

We will also need the following consequence of the presentation of $bG$ (for an abelian discrete group $G$)  as the ``double  Pontryagin dual'', which can be found in \cite{St}; a short proof based on Pontryagin duality is given in \cite[Section 1]{GJK}.

\begin{fact}\label{fact:profinite_bohr_compactification}
Let $G$ be a discrete abelian group. Then, $bG$ is profinite iff $G$ is of finite exponent.\qed
\end{fact}

\subsection{Structural Ramsey theory}

Recall that a first order structure $K$ is called a {\em Fra\"iss\'e structure} if it is countable, locally finite (finitely generated substructures of $K$ are finite; although this property is not always taken as a part of the definition) and ultrahomogeneous (every isomorphism between finite substructures of $K$ lifts to an automorphism of $K$). Every Fra\"iss\'e structure $K$ is uniquely determined by its {\em age}, $\mathrm{Age}(K)$, i.e.\ the class of all finite structures which are embeddable in $K$. The age of a Fra\"iss\'e structure is a {\em Fra\"iss\'e class}, i.e.\ a class of finite structures which is closed under isomorphisms, countable up to isomorphism, and satisfies {\em hereditary}, {\em joint embedding} and {\em amalgamation property}. On the other hand, for every Fra\"iss\'e class $\mathcal C$ of first order structures there exists a unique Fra\"iss\'e structure, called the {\em Fra\"iss\'e limit} of $\mathcal C$, whose age is exactly $\mathcal C$.

Structural Ramsey theory, invented by Ne\v{s}et\v{r}il and  R\"odl in the 1970s, investigates combinatorial properties of Fra\"iss\'e classes (and, more generally, categories).
Originally, these structural combinatorial properties are given by colorings of isomorphic copies of $A$ in $B$, where $A$ and $B$ are members of the given class such that $A$ is embeddable in $B$. We follow the approach from \cite{Z}. Let $K$ be a Fra\"iss\'e structure. For $A,B\in\Age(K)$, $\Emb(A,B)$ stands for the set of all embeddings $A\to B$.  $K$ (or $\Age(K)$) has {\em separately finite embedding Ramsey degree} if for every $A \in \Age(K)$ there exists $l<\omega$ such that for every $B\in\Age(K)$ with $\Emb(A,B)\neq\emptyset$ and every $r<\omega$ there exists $C\in\Age(K)$ such that for every coloring $c:\Emb(A,C)\to r$ there exists $f\in\Emb(B,C)$ with $\#c[f\circ\Emb(A,B)]\leqslant l$.  
We added the word ``separately'' to emphasize that $l$ depends on $A$. If in the previous definition $l$ can be chosen to be $1$ for every $A$, we obtain the notion of the {\em embedding Ramsey property}.
By \cite[Proposition 4.4]{Z}, $K$ has separately finite embedding Ramsey degree iff it has separately finite structural Ramsey degree (defined by using colorings of isomorphic copies of $A$ in $B$ in place of embeddings). Moreover, by \cite[Corollary 4.5]{Z}, $K$ has the embedding Ramsey property iff it has the structural Ramsey property and all structures in $\Age(K)$ are rigid (have trivial automorphism groups).

Both the property of having separately finite embedding Ramsey degree and the embedding Ramsey property for a Fra\"iss\'e structure $K$ can be alternatively defined as follows. For finite $\bar a\subseteq K$ and $C\subseteq K$ denote by ${C\choose\bar a}^\qf$ the set of all $\bar a'\subseteq C$ such that $\bar a'\equiv^\qf\bar a$ ($\bar a'$ and $\bar a$ have the same quantifier-free type).

\begin{fact}\label{fact fraisse ramsey}
A Fra\"iss\'e structure $K$ has separately finite embedding Ramsey degree iff for every finite $\bar a \subseteq K$ there exists $l<\omega$ such that for any finite $\bar b\subseteq K$ containing $\bar a$ and $r<\omega$ there exists a finite $C\subseteq K$ such that for every coloring $c:{C\choose\bar a}^\qf\to r$ there exists $\bar b'\in{C\choose\bar b}^\qf$ with $\#c[{\bar b'\choose\bar a}^\qf]\leqslant l$.

The same holds for the embedding Ramsey property and $l=1$.\qed
\end{fact}

\subsection{Flows in model theory}\label{Subsection: Flows in model theory}

As mentioned in the introduction, the methods of topological dynamics were introduced to model theory by Newelski over ten years ago, and since then this approach has gained a lot of attention which resulted in various applications (e.g.\ to strong types in \cite{KPRz, KRz18}). Flows in model theory occur naturally in two ways. One way is to consider the action of a definable group on various spaces of types concentrated on this group. The other is to consider the action of the automorphism group of a model on various spaces of types. In this paper, we consider the second situation. We are mostly interested in the flow $(\Aut(\C), S_{\bar c}(\C))$, where $\C$ is a monster model of a theory $T$, $\bar c$ is an enumeration of $\C$, $S_{\bar c}(\C)$ is the space of all types over $\C$ extending $\tp(\bar c)$, and the action of $\Aut(\C)$ on $S_{\bar c}(\C)$ is the obvious one.


Recall that the question when the group $\GalKP(T)$ is profinite was an initial motivation behind this paper.
In \cite{KPRz} it is proved that there is a sequence of quotient topological maps and group epimorphisms:
$$u\M\to u\M/H(u\M)\to \GalL(T)\to\GalKP(T),$$
for  any minimal left ideal $\M$ of $\EL(S_{\bar c}(\C))$ and $u\in\J(\M)$,  
where $u\M$ is equipped with the $\tau$-topology, $u\M/H(u\M)$ with the corresponding quotient topology, and $\GalL(T)$ and $\GalKP(T)$ with the topologies described in Subsection \ref{subsection model theory}. 
The following easy consequence of this result gives a sufficient condition for profiniteness of $\GalKP(T)$.

\begin{proposition}\label{lemma galkp profinite}  $\GalKP(T)$ is profinite if $u\M/H(u\M)$ is profinite. Furthermore, $u\M/H(u\M)$ is profinite if $u\M$ is $0$-dimensional.
\end{proposition}
\begin{proof}
Both assertions are consequences of the following more general, easy fact: If $G$ is a compact semitopological group, $H$ is a Hausdorff topological group, and $f:G\to H$ is a quotient topological map and group epimorphism, then $H$ is profinite if $G$ is $0$-dimensional.
\end{proof}

Thus, by Remark \ref{remark uM and um/H zero dimensionality}, for profiniteness of $\GalKP(T)$ it is enough to prove profiniteness (equiv. $0$-dimensionality) of either $u\M$ or $u\M/H(u\M)$.


The investigation of the flow $(\Aut(\C), S_{\bar c}(\C))$ in concrete examples is not easy, thus we need reductions of it to flows which are easier to deal with. Some reductions of this kind appeared in Section 2 of \cite{KNS}, but here we need more specific ones.

Let $\bar d$ be a tuple of all elements of $\C$ in which each element of $\C$ is repeated infinitely many times. It is clear that $S_{\bar d}(\C)$ and $S_{\bar c}(\C)$ are isomorphic $\Aut(\C)$-flows. Hence,  their Ellis semigroups are isomorphic as semigroups and as $\Aut(\C)$-flows. So, in fact, one can work with $S_{\bar d}(\C)$ in place of $S_{\bar c}(\C)$ whenever it is convenient.

Let us fix some notation. 
Let $\bar x$ be a tuple of variables corresponding to $\bar d$. For $\bar x'\subseteq\bar x$,  $\bar z$ with $|\bar z|=|\bar x'|$, and for $p(\bar x')\in S_{\bar x'}(\C)$ we denote by $p[\bar x'/\bar z]$ the type from $S_{\bar z}(\C)$ obtained by replacing the variables $\bar x'$ by $\bar z$. Denote by $\bar z$ a sequence $(z_i)_{i<n}$ of variables, where $n\leqslant\omega$. We discuss connections between the $\Aut(\C)$-flows $S_{\bar d}(\C)$ and $S_{\bar z}(\C)$. 
The proof of the next lemma is left as an exercise.

\begin{lemma}\label{lemma: first reduction}
Let $\Phi:\EL(S_{\bar d}(\C))\to \EL(S_{\bar z}(\C))$ be defined by $\Phi(\eta):=\hat\eta$, where $\hat\eta$ is given by:
$$\hat\eta(p(\bar z))= \eta(q(\bar x))_{\upharpoonright\bar x'}[\bar x'/\bar z],$$
where $p(\bar z)\in S_{\bar z}(\C)$, and $\bar x'\subseteq\bar x$ and $q(\bar x)\in S_{\bar d}(\C)$ are such that $q(\bar x)_{\upharpoonright\bar x'}[\bar x'/\bar z]=p(\bar z)$. Then $\Phi$ is a well-defined epimorphism of $\Aut(\C)$-flows and of semigroups.\qed
\end{lemma}

Having in mind that the flows $S_{\bar d}(\C)$ and $S_{\bar c}(\C)$ are isomorphic, we get:

\begin{corollary}\label{lemma c to z} There exists an $\Aut(\C)$-flow and semigroup epimorphism from $\EL(S_{\bar c}(\C))$ to $\EL(S_{\bar z}(\C))$. In particular, such an epimorphism exists from $\EL(S_{\bar c}(\C))$ to $\EL(S_n(\C))$ for all $n<\omega$.\qed 
\end{corollary}

\begin{corollary}\label{corollary c to infinite z} 
If $\bar z=(z_i)_{i<\omega}$, then $\EL(S_{\bar c}(\C))$ and $\EL(S_{\bar z}(\C))$ are $\Aut(\C)$-flow and semigroup isomorphic. Thus, the Ellis groups of the flows $(\Aut(\C),S_{\bar c}(\C))$ and $(\Aut(\C),S_{\bar z}(\C))$ are topologically isomorphic.
\end{corollary}
\begin{proof}
It is enough to prove that $\Phi$ from Lemma \ref{lemma: first reduction} is injective. Let $\eta_1\neq\eta_2\in\EL(S_{\bar d}(\C))$, and take $q(\bar x)$ and $\varphi(\bar x,\bar a)$ such that $\varphi(\bar x,\bar a)\in\eta_1(q(\bar x))$ and $\lnot\varphi(\bar x,\bar a)\in\eta_2(q(\bar x))$. Take $\bar x'\subseteq\bar x$ such that $|\bar x'|=\omega$ and all variables occurring in $\varphi(\bar x,\bar a)$ are in $\bar x'$, so we may write $\varphi(\bar x,\bar a)$ as $\varphi'(\bar x',\bar a)$. Let $p(\bar z)= q(\bar x)_{\upharpoonright\bar x'}[\bar x'/\bar z]$. By the definition of $\hat\eta_1$ and $\hat\eta_2$, $\varphi'(\bar z,\bar a)\in\hat\eta_1(p(\bar z))$ and $\lnot\varphi'(\bar z,\bar a)\in\hat\eta_2(p(\bar z))$. Thus, $\hat\eta_1\neq\hat\eta_2$.
\end{proof}

If $|\bar z|=n<\omega$, then in general we do not have injectivity of $\Phi$, but we may distinguish a sufficient condition on $T$ for which $\Phi$ is injective for some $n$. We say that a theory $T$ is {\em $m$-ary} (for some $m<\omega$) if every $L$-formula is equivalent modulo $T$  to a Boolean combination of $L$-formulae with at most $m$ free variables.

\begin{corollary}\label{corollary c to finite z m-ary} 
If $T$ is $(m+1)$-ary, then $\EL(S_{\bar c}(\C))$ and $\EL(S_{m}(\C))$ are $\Aut(\C)$-flow and semigroup isomorphic. Thus, the Ellis groups of the flows $(\Aut(\C),S_{\bar c}(\C))$ and $(\Aut(\C),S_{m}(\C))$ are topologically isomorphic.
\end{corollary}
\begin{proof}
If $|\bar z|=m$, we prove that $\Phi$ is injective.
Let $\eta_1\neq\eta_2\in\EL(S_{\bar d}(\C))$ and let $q(\bar x)\in S_{\bar d}(\C)$ be such that $\eta_1(q(\bar x))\neq \eta_2(q(\bar x))$. By $(m+1)$-arity of the theory,  there is a formula $\varphi(\bar x'',\bar y)$ with $\bar x''\subseteq\bar x$ and $|\bar x''|+|\bar y|\leqslant m+1$, and $\bar a$ such that $\varphi(\bar x'',\bar a)\in\eta_1(q(\bar x))$ and $\lnot\varphi(\bar x'',\bar a)\in\eta_2(q(\bar x))$. Note that $|\bar y|\geqslant 1$, as otherwise $\varphi(\bar x'',\bar a)$ is a formula without parameters, so it belongs to $\eta_1(q(\bar x))$ iff it belongs to $\eta_2(q(\bar x))$. So $|\bar x''|\leqslant m$, and we can write $\varphi(\bar x'',\bar a)$ as $\varphi'(\bar x',\bar a)$ for some $\bar x'\subseteq\bar x$ such that $|\bar x'|=m$. Now, the final step of the proof of Corollary \ref{corollary c to infinite z} goes through.
\end{proof}

We will need the following general, easy observation.

\begin{remark}\label{remark invariance} 
Let $(G,X)$ and $(G,Y)$ be $G$-flows. Denote by $\Inv(X)$ and $\Inv(Y)$ the sets of all points fixed by $G$ in $X$ and $Y$, respectively. Assume that $F: \EL(X) \to \EL(Y)$ is a $G$-flow epimorphism. Then for any $\eta \in \EL(X)$ we have: $\Im(\eta)  \subseteq  \Inv(X)$ iff $\Im(\eta) = \Inv(X)$, and these equivalent conditions imply $\Im(F(\eta)) = \Inv(Y)$. So, if $F$ is an isomorphism, then  $\Im(\eta) = \Inv(X)$ iff $\Im(F(\eta)) = \Inv(Y)$.\qed
\end{remark}



By Corollary \ref{corollary c to finite z m-ary} and  Remark \ref{remark invariance}, we get:

\begin{corollary}\label{corollary: c to finite z for image in inv types}
If $T$ is $(m+1)$-ary, then the existence of $\eta\in\EL(S_{\bar c}(\C))$ with $\Im(\eta)= \Inv_{\bar c}(\C)$ is equivalent to the existence of $\eta'\in\EL(S_{m}(\C))$ with $\Im(\eta')= \Inv_{m}(\C)$ (where $\Inv_{m}(\C)$ is the set of all invariant types in $S_m(\C)$).\qed
\end{corollary}

We will need one more consequence of the above investigations. 
Let $L\subseteq L^*$ be two languages, let $T^*$ be a complete $L^*$-theory and $T:=T^*_{\upharpoonright L}$. Assume that $\C^*$ is a monster of $T^*$ such that $\C:=\C^*_{\upharpoonright L}$ is a monster of $T$. Let $\bar c$ be an enumeration of $\C$ (and $\C^*$). We can treat $S_{\bar c}(\C)$ as an $\Aut(\C^*)$-flow, and when we do that (in this section) we write  $S_{\bar c}^*(\C)$ in place of  $S_{\bar c}(\C)$; similarly for $S_{\bar z}^*(\C)$. Note that $\Aut(\C^*) \leqslant \Aut(\C)$, and so $\EL(S_{\bar c}^*(\C)) \subseteq \EL(S_{\bar c}(\C))$ and $\EL(S_{\bar z}^*(\C)) \subseteq \EL(S_{\bar z}(\C))$.

\begin{lemma}\label{lemma ellis semi epi bigger lang invariant} Take the previous notation.
\begin{enumerate}[label=(\roman*), align=right, leftmargin=*]
\item There is an $\Aut(\C^*)$-flow and semigroup epimorphism $\Psi:\EL(S_{\bar c}(\C^*))\to \EL(S_{\bar c}^*(\C))$.
\item If there is $\eta^*\in\EL(S_{\bar c}(\C^*))$ such that $\Im(\eta^*)\subseteq\Inv_{\bar c}(\C^*)$ (equiv. $=\Inv_{\bar c}(\C^*)$), then there is $\eta\in\EL(S_{\bar c}(\C))$ such that $\Im(\eta)= \Inv_{\bar c}^*(\C)$, where $\Inv_{\bar c}(\C^*)\subseteq S_{\bar c}(\C^*)$ and $\Inv_{\bar c}^*(\C)\subseteq S_{\bar c}^*(\C)$ are the subsets of all $\Aut(\C^*)$-invariant types in $S_{\bar c}(\C^*)$ and $S_{\bar c}^*(\C)$, respectively.
\item There is an $\Aut(\C^*)$-flow and semigroup epimorphism $\Psi_{\bar z}:\EL(S_{\bar c}(\C^*))\to \EL(S_{\bar z}^*(\C))$.
\item If there is $\eta^*\in\EL(S_{\bar c}(\C^*))$ such that $\Im(\eta^*)\subseteq\Inv_{\bar c}(\C^*)$ (equiv. $=\Inv_{\bar c}(\C^*)$), then there is $\eta\in\EL(S_{\bar z}(\C))$ such that $\Im(\eta)= \Inv_{\bar z}^*(\C)$, where $\Inv_{\bar z}^*(\C)\subseteq S_{\bar z}^*(\C)$ is the subset of all $\Aut(\C^*)$-invariant types in $S_{\bar z}^*(\C)$.
\end{enumerate} 
\end{lemma}
\begin{proof}
(i) Let $\bar z=(z_i)_{i<\omega}$.
By Corollary \ref{corollary c to infinite z}, we have an $\Aut(\C^*)$-flow and semigroup isomorphism $\Phi^*:\EL(S_{\bar c}(\C^*))\to \EL(S_{\bar z}(\C^*))$, and an $\Aut(\C)$-flow and semigroup isomorphism $\Phi:\EL(S_{\bar c}(\C))\to \EL(S_{\bar z}(\C))$. Since $\Aut(\C^*)\leqslant\Aut(\C)$, we get that $\Phi_{\upharpoonright \EL(S_{\bar c}^*(\C))}$ is an $\Aut(\C^*)$-flow and semigroup isomorphism $\EL(S_{\bar c}^*(\C))\to \EL(S_{\bar z}^*(\C))$. We also have an $\Aut(\C^*)$-flow epimorphism $S_{\bar z}(\C^*)\to S_{\bar z}^*(\C)$ given by the restriction of the language; hence,  by Fact \ref{fact induced epimorphism}, there exists an $\Aut(\C^*)$-flow and semigroup epimorphism $\Theta:\EL(S_{\bar z}(\C^*))\to \EL(S_{\bar z}^*(\C))$.  
Then $\Psi:=\Phi_{\upharpoonright \EL(S_{\bar c}^*(\C))}^{-1}\circ\Theta\circ\Phi^*$ is as required.

(ii) Since $\EL(S_{\bar c}^*(\C)) \subseteq \EL(S_{\bar c}(\C))$, (ii) follows from (i) and Remark \ref{remark invariance} applied to $X:= S_{\bar c}(\C^*)$, $Y:= S_{\bar c}^*(\C)$, and $G := \Aut(\C^*)$.

(iii) $\Psi_{\bar z}:=\Theta \circ \Phi^*$ from the proof of (i) does the job (but here $\Phi^*:\EL(S_{\bar c}(\C^*))\to \EL(S_{\bar z}(\C^*))$ is an epimorphism provided by Corollary \ref{lemma c to z}).

(iv) follows from (iii) and Remark \ref{remark invariance}.
\end{proof}

We now describe a natural way of presenting $S_{\bar c}(\C)$ as an inverse limit of $\Aut(\C)$-flows, 
which is one of the key tools in this paper.

For any $\bar a\subseteq\C^*\succ\C$, $\Delta=\{\varphi_i(\bar x,\bar y)\}_{i<k}$ where $|\bar x|=|\bar a|$, and $\bar p=\{p_j\}_{j<m}\subseteq S_{\bar y}(T)$, by $\tp_\Delta(\bar a/\bar p)$ we mean the $\Delta$-type of $\bar a$ over $\bigcup_{j<m}p_j(\C)$, i.e.\ the set of all formulae of the form $\varphi_i(\bar x,\bar b)^\epsilon$ such that $\epsilon\in 2$, $\bar b$ realizes one of $p_j$'s, and $\models\varphi_i(\bar a,\bar b)^\epsilon$. (Here, as usual, $\varphi^0$ denotes $\lnot\varphi$, and $\varphi^1$ denotes $\varphi$.)

For $\Delta=\{\varphi_i(\bar x,\bar y)\}_{i<k}$, where $\bar x$ is reserved for $\bar c$, and $\bar p=\{p_j\}_{j<m}\subseteq S_{\bar y}(T)$, by $S_{\bar c,\Delta}(\bar p)$ we denote the space of all complete $\Delta$-types over $\bigcup_{j<m}p_j(\C)$ consistent with $\tp(\bar c)$; equivalently:
$$S_{\bar c,\Delta}(\bar p)=\{\tp_\Delta(\bar c^*/\bar p)\mid \bar c^*\subseteq\C^*\mbox{ and }\bar c^*\equiv\bar c\}.$$
In the usual way, we endow $S_{\bar c,\Delta}(\bar p)$ with a topology, turning it into a $0$-dimensional, compact, Hausdorff space. Moreover, it is naturally an $\Aut(\C)$-flow.

Let $\mathcal F$ be the family of all pairs $(\Delta=\{\varphi_{i}\}_{i<k},\bar p=\{p_j\}_{j<m})$ as above. We order $\mathcal F$ naturally by:
$$(\Delta=\{\varphi_{i}(\bar x,\bar y)\}_{i<k},\bar p=\{p_j\}_{j<m})\ \leqslant\ (\Delta'=\{\varphi_{i}'(\bar x,\bar y')\}_{i<k'},\bar p'=\{p_j'\}_{j<m'})$$
if $\bar y\subseteq\bar y'$, $\{\varphi_i(\bar x,\bar y)\}_{i<k}\subseteq\{\varphi_i'(\bar x,\bar y')\}_{i<k'}$ by using dummy variables, and $\{p_j\}_{j<m}\subseteq\{p_{j'\upharpoonright\bar y}\}_{j<m'}$.
Then $\mathcal F$ is directed by $\leqslant$, and for pairs $t= (\Delta,\bar p)$ and $t'= (\Delta',\bar p')$ in $\mathcal F$, if $t\leqslant t'$, we have an $\Aut(\C)$-flow epimorphism given by the restriction: 
$$\pi_{t,t'}:S_{\bar c,\Delta'}(\bar p')\to S_{\bar c,\Delta}(\bar p).$$ 
Therefore, we have an inverse system of $\Aut(\C)$-flows $((\Aut(\C),S_{\bar c,\Delta}(\bar p)))_{(\Delta,\bar p)\in\mathcal F}$, and we clearly have:

\begin{lemma}\label{lemma inverse limit of St's} $S_{\bar c}(\C)\cong \invlim_{(\Delta,\bar p)\in\mathcal F}S_{\bar c,\Delta}(\bar p)$ as $\Aut(\C)$-flows.\qed
\end{lemma}

The usual $\Aut(\C)$-flow $S_{\bar c, \Delta}(\C)$ of complete $\Delta$-types over the whole $\C$ consistent with $\tp(\bar c)$  projects onto $S_{\bar c,\Delta}(\bar p)$. We also have $S_{\bar c}(\C)\cong \invlim_{\Delta\in\mathcal F}S_{\bar c,\Delta}(\C)$, but this presentation is not sufficient to be used in our main results 
(see Example \ref{example R_2 Ps}).

\subsection{Contents and strong heirs}

The following definition was introduced in \cite{KNS}.

\begin{definition}[{\cite[Definition 3.1]{KNS}}]\label{definition content} Fix $A\subseteq B$.
\begin{enumerate}[label=(\alph*), align=right, leftmargin=*]
\item For $p(\bar x)\in S(B)$, the {\em content} of $p$ over $A$ is defined as:
$$\ct_A(p)=\{(\varphi(\bar x,\bar y),q(\bar y))\in L(A)\times S(A)\mid \varphi(\bar x,\bar b)\in p(\bar x)\mbox{ for some }\bar b\models q\}.$$
\item The content of a sequence $p_0(\bar x),\dots,p_{n-1}(\bar x)\in S(B)$ over $A$, $\ct_A(p_0,\dots,p_{n-1})$, is defined as the set of all $(\varphi_0(\bar x,\bar y),\dots,\varphi_{n-1}(\bar x,\bar y),q(\bar y))\in L(A)^n\times S(A)$ such that 
for some $\bar b\models q$ and for every $i < n$ we have $\varphi_i(\bar x,\bar b)\in p_i$.
\end{enumerate}
If $A=\emptyset$, we write just $\ct(p)$ and $\ct(p_0,\dots,p_{n-1})$.
\end{definition}

A fundamental connection between contents and the Ellis semigroup is given by:

\begin{fact}[{\cite[Proposition~3.5]{KNS}}]\label{fact contents ellis} Let $\pi(\bar x)$ be a type over $\emptyset$, and $(p_0,\dots,p_{n-1})$ and $(q_0,\dots,q_{n-1})$ sequences from $S_{\pi}(\C)$. Then $\ct(q_0,\dots,q_{n-1})\subseteq \ct(p_0,\dots,p_{n-1})$ iff there exists $\eta\in\EL(S_{\pi}(\C))$ such that $\eta(p_i)=q_i$ for every $i<n$.\qed
\end{fact}

As was explained in \cite{KNS}, contents expand the concept of fundamental order of Lascar and Poizat, hence an analogous notion of heir can be defined.

\begin{definition}[{\cite[Definition 3.2]{KNS}}]\label{definition strong heir} Let $M\subseteq A$ and $p(\bar x)\in S(A)$. $p(\bar x)$ is a {\em strong heir} over $M$ if for every finite $\bar m\subseteq M$ and $\varphi(\bar x,\bar a)\in p(\bar x)$, where $\bar a\subseteq A$ is finite and $\varphi(\bar x,\bar y)\in L(M)$, there is $\bar a'\subseteq M$ such that $\varphi(\bar x,\bar a')\in  p(\bar x)$ and $\tp(\bar a'/\bar m)=\tp(\bar a/\bar m)$.
\end{definition}


\begin{fact}[{\cite[Lemma 3.3]{KNS}}]\label{fact str heir exists} Let $M\subseteq A$, where $M$ is $\aleph_0$-saturated. Then every $p(\bar x)\in S(M)$ has an extension $p'(\bar x)\in S(A)$ which is a strong heir over $M$.\qed
\end{fact}

\subsection{Amenability of a theory}\label{subsection: amenability}

Amenable and extremely amenable theories were introduced and studied in \cite{HKP}. We will not give the original definitions but rather the characterizations which we will use in this paper. For the details see \cite{HKP}.

\begin{definition}\label{definition of amenability} 
\begin{enumerate}[label=(\alph*), align=right, leftmargin=*]
\item A theory $T$ is {\em amenable} if every finitary type $p\in S(T)$ is {\em amenable}, i.e.\ there exists an invariant, (regular) Borel probability measure on $S_p(\C)$.

\item A theory $T$ is {\em extremely amenably} if every finitary type $p\in S(T)$ is {\em extremely amenable}, i.e.\ there exists an invariant type in $S_p(\C)$. 
\end{enumerate}
\end{definition}

In fact, in the above definitions we can remove the adjective ``finitary''. We also get the same notions if we use only $p:=\tp(\bar c)$ (where $\bar c$ is an enumeration of $\C$). 

These definitions do not depend on the choice of the monster model $\C$, i.e.\ they are indeed properties of the theory $T$. In fact, it is enough to assume only that $\C$ is $\aleph_0$-saturated and strongly  $\aleph_0$-homogeneous. 

One should also recall that a regular, Borel probability measure on the space $S_p(\C)$ (or on any $0$-dimensional compact space) is the same thing as a Keisler measure, i.e.\ finitely additive probability measure on the Boolean algebra of all clopen sets. All such measures form a compact subspace $\frak M_p$ of $[0,1]^{\mathrm{clopens}}$ equipped with the product topology.

\section{Definable structural Ramsey theory and topological dynamics}\label{section ramsey}

As mentioned in the introduction, in order to find interactions between Ramsey-like properties of a given theory $T$ and dynamical properties of $T$, one has to impose the appropriate definability conditions on colorings.

For a finite tuple $\bar a$ and a subset $C\subseteq\C$, by $C\choose\bar a$ we denote the set of all realizations of $\tp(\bar a)$ in $C$:
${C\choose\bar a}:= \{\bar a'\in C^{|\bar a|}\mid \bar a'\equiv\bar a\}.$
If instead of $C$ we have a tuple $\bar d$, the meaning of ${\bar d\choose\bar a}$ is the same, i.e.\ it is $D\choose\bar a$, where $D$ is the set of all coordinates of $\bar d$.

For $r<\omega$, a {\em coloring} of the realizations of $\tp(\bar a)$ in $C$ into $r$ colors is any mapping $c:{C\choose\bar a}\to r$. A subset $S\subseteq{C\choose\bar a}$ is {\em monochromatic} with respect to $c$ if $c[S]$ is a singleton.

\begin{definition}\phantomsection\label{definition of definable colorings}
\begin{enumerate}[label=(\alph*), align=left, leftmargin=*,labelsep=0pt]
\item A coloring $c:{C\choose\bar a}\to 2^n$ is {\em definable} if there are formulae with parameters $\varphi_i(\bar x)$, $i<n$, such that:
$$c(\bar a')(i)= \left\{
\begin{array}{cl}
1, & \models \varphi_i(\bar a')\\
0, & \models \lnot\varphi_i(\bar a')
\end{array}
\right.$$
for any $\bar a'\in{C\choose \bar a}$ and $i<n$. 
\item A coloring $c:{C\choose\bar a}\to 2^n$ is {\em externally definable} if there are formulae without parameters $\varphi_i(\bar x,\bar y)$ and types $p_i(\bar y)\in S_{\bar y}(\C)$, $i<n$, such that:
$$c(\bar a')(i)= \left\{
\begin{array}{cl}
1, &  \varphi_i(\bar a',\bar y)\in p_i(\bar y)\\
0, & \lnot\varphi_i(\bar a',\bar y)\in p_i(\bar y)
\end{array}
\right.$$
for any $\bar a'\in{C\choose\bar a}$ and $i<n$. 
\item If $\Delta$ is a set of formulae in variables $\bar x$ and $\bar y$ and $q \in S_{\bar y}(\emptyset)$, then an externally definable coloring $c$ is called {\em externally definable $(\Delta,q)$-coloring} if all the formulae $\varphi_i(\bar x,\bar y)$'s defining $c$ are taken from $\Delta$ and $p_i(\bar y) \in S_{q}(\C)$ for $i<n$.
\end{enumerate}
\end{definition}

\begin{remark}\label{remark: definable is given by realized types}
A coloring $c:{C\choose\bar a}\to 2^n$ is definable iff it is externally definable via realized (in $\C$)  types $p_0(\bar y),\dots,p_{n-1}(\bar y)\in S_{\bar y}(\C)$.\qed
\end{remark}

\begin{remark}\phantomsection\label{remark ext def coloring} 
\begin{enumerate}[label=(\roman*), align=left, leftmargin=*,labelsep=-3pt]
\item An externally definable coloring $c:{C\choose\bar a}\to 2^n$ given by $\varphi_i(\bar x,\bar y)$ and $p_i(\bar y)\in S_{\bar y}(\C)$, $i<n$, can be defined by using $n$ formulae $\psi_i(\bar x,\bar z)$, $i<n$, and only one type $p(\bar z)\in S_{\bar z}(\C)$. Moreover, we can require that $|\bar z|=|\bar c|$ and $p(\bar z) \in S_{\bar c}(\C)$. 
\item In the definition of definable coloring, we can assume that all formulae $\varphi_i(\bar x)$, $i<n$, have the same parameters $\bar e$, and then the coloring is externally definable witnessed by the single realized type $p(\bar y) := \tp(\bar e/\C)$.
\item In the definition of externally definable $(\Delta,q)$-coloring, without loss of generality we can assume that $|\bar y|=|\bar c|$ and $q(\bar y)=\tp(\bar c/\emptyset)$. Whenever $|\bar y|=|\bar c|$ and $q(\bar y) = \tp(\bar c/\emptyset)$, instead of ``externally definable $(\Delta,q)$-coloring'' we will just say {\em externally definable $\Delta$-coloring}.
\end{enumerate}
\end{remark}


\begin{proof}
(i)
Let $\bar z=(\bar y_0,\dots,\bar y_{n-1})$, where each $\bar y_i$ is of length $|\bar y|$ (where wlog $\bar y$ is finite) and the $\bar y_i$'s are pairwise disjoint. Let $p(\bar z)\in S_{\bar z}(\emptyset)$ be any completion of $\bigcup_{i<n}p_i(\bar y_i)$ and let $\psi_i(\bar x,\bar z):= \varphi_i(\bar x,\bar y_i)$. Then $\psi_i(\bar a',\bar z)\in p(\bar z)$ iff $\varphi_i(\bar a',\bar y)\in p_i(\bar y)$. For the ``moreover part'', by adding dummy variables, we can assume that $\bar z$ corresponds to $\bar d$ and $p(\bar z) \in S_{\bar d}(\C)$, where $\bar d$ is a tuple of all elements of $\C$ in which each element is repeated infinitely many times. Finally, we can identify the variables in the tuple $\bar z$ which correspond to the same elements in $\bar d$, and we obtain a shorter tuple, still denoted by $\bar z$, corresponding to $\bar c$ and a type, still denoted by $p(\bar z)$, in $S_{\bar c}(\C)$ which satisfy the requirements.\\
(ii) follows by adding dummy parameters, and
(iii) by adding dummy variables.
\end{proof}

The next remark explains that Definition \ref{definition of definable colorings} coincides with the usual definition of [externally] definable map from a type-definable set to a compact, Hausdorff space (in our case this space is finite). Recall that a function $f: X \to C$, where $X$ is a type-definable subset of a sufficiently saturated model and $C$ is a compact, Hausdorff space, is said to be {\em [externally] definable} if the preimages of any two disjoint closed subsets of $C$ can be separated by a relatively  [externally] definable subset of $X$. In particular, if $C$ is finite, then this is equivalent to saying that all fibers of $f$ are relatively  [externally] definable subsets of $X$.

\begin{remark} 
\begin{enumerate}[label=(\roman*), align=left, leftmargin=*,labelsep=-4pt]
\item For a coloring $c \colon {\C \choose \bar a} \to 2^n$,  $c$ is [externally] definable in the sense of Definition \ref{definition of definable colorings} iff it is [externally] definable in the above sense  (i.e.\ the fibers are relatively [externally] definable subsets of the type-definable set  ${\C \choose \bar a}$).
\item If $c: {\C \choose \bar a} \to r$ is [externally] definable in the above sense, then it remains so as a coloring with the target space $2^n$, where $2^n \geqslant r$. So we can always assume that $r=2^n$.\qed
\end{enumerate}
\end{remark}


\subsection{Ramsey properties}

Motivated by the embedding Ramsey property for\linebreak Fra\"{i}ss\'{e} structures, we introduce the following natural notion.

\begin{definition}\label{definition ramsey property}
A theory $T$ has \eerp\ ({\em the elementary embedding Ramsey property}) if for any two finite tuples $\bar a\subseteq\bar b\subseteq\C$ and any $r<\omega$ there exists a finite subset $C\subseteq\C$ such that for any coloring $c:{C\choose\bar a}\to r$ there exists $\bar b'\in{C\choose\bar b}$ such that $\bar b'\choose\bar a$ is monochromatic with respect to $c$.
\end{definition}

\begin{remark}\label{remark rp ind in sat and str hom}
The definition of \eerp\ does not depend on the choice of the monster (or just an $\aleph_0$-saturated) model $\C$ of $T$, i.e.\ \eerp\ is indeed a property of $T$.\qed
\end{remark}

\begin{remark}\label{remark fraisse eerp}
Definition \ref{definition ramsey property} generalizes the definition of the embedding Ramsey property for Fra\"iss\'e structures in the following sense: If $K$ is an $\aleph_0$-saturated Fra\"iss\'e structure, then $K$ has the embedding Ramsey property iff $\mathrm{Th}(K)$ has EERP. 
\end{remark}
\begin{proof}
Recall that if $K$ is an $\aleph_0$-saturated Fra\"iss\'e structure, then $\mathrm{Th}(K)$ has q.e., so ${C\choose\bar a}^{\qf}= {C\choose\bar a}$ for any finite $\bar a, C \subseteq K$. Therefore, we conclude using  Fact \ref{fact fraisse ramsey} and Remark \ref{remark rp ind in sat and str hom}.
\end{proof}

The next remark is standard and follows from \cite[Remark 3.1]{KP}.

\begin{remark}\label{lemma equivalents of eerps} A theory $T$ has \eerp\ iff for any two finite tuples $\bar a\subseteq\bar b\subseteq\C$, any $r<\omega$, and any coloring $c:{\C\choose\bar a}\to r$ there exists $\bar b'\in{\C\choose\bar b}$ such that $\bar b'\choose\bar a$ is monochromatic with respect to $c$. \qed
\end{remark}

In {\cite{KP}}, the class of first order structures with {\em ERP} ({\em the embedding Ramsey property}) is introduced. A first order structure $M$ has {\em ERP} if
for any finite $\bar a\subseteq\bar b\subseteq M$ and any $r<\omega$ there exists a finite subset $C\subseteq M$ such that for any coloring $c:{C\choose\bar a}^{\Aut}\to r$ there exists $\bar b'\in{C\choose\bar b}^{\Aut}$ such that ${\bar b'\choose\bar a}^{\Aut}$ is monochromatic with respect to $c$.
Here, for finite $\bar a\subseteq M$ and $C\subseteq M$:
$${C\choose\bar a}^{\Aut}: =\{\bar a'\subseteq C\mid \bar a'=f(\bar a)\mbox{ for some }f\in\Aut(M)\}.$$

For Fra\"{i}ss\'{e} structures the next fact is one of the main results from \cite{KPT}, which was later generalized to arbitrary locally finite ultrahomogeneous structures in \cite{TSS}. The formulation below comes from \cite{KP}, but it can be checked (by passing to canonical ultrahomogeneous expansions and using an argument as in Fact \ref{fact fraisse ramsey}) that it is equivalent to the one from \cite{TSS}.

\begin{fact}[{\cite[Theorem 3.2]{KP}}]\label{fact characterization of eerp} A first order structure $M$ has ERP iff $\Aut(M)$ is extremely amenable as a topological group.\qed
\end{fact}

Note that if $M$ is strongly $\aleph_0$-homogeneous, then ${C\choose\bar a}^{\Aut}={C\choose\bar a}$, so if we in addition assume that $M$ is $\aleph_0$-saturated, then by Remark \ref{remark rp ind in sat and str hom}, $M$ has {\em ERP} iff $\mathrm{Th}(M)$ has \eerp. Hence, by Fact \ref{fact characterization of eerp}, we obtain the following.

\begin{corollary}\label{corollary extr amen prop of theory} For a theory $T$, $\Aut(M)$ is extremely amenable (as a topological group) for some $\aleph_0$-saturated and strongly $\aleph_0$-homogeneous model $M\models T$ iff it is so for all $\aleph_0$-saturated and strongly $\aleph_0$-homogeneous models of $T$.\qed
\end{corollary}

We now introduce and study two new classes of theories by restricting our considerations to definable and externally definable colorings, which makes the whole subject more general and more model-theoretic.

\begin{definition}\label{definition def ramsey property}
A theory $T$ has $[E]DEERP$ ({\em the [externally] definable elementary embedding Ramsey property}) if for any finite $\bar a\subseteq\bar b\subseteq\C$, any $n<\omega$ and any [externally] definable coloring $c:{\C\choose \bar a}\to 2^n$ there exists $\bar b'\in{\C\choose\bar b}$ such that $\bar b'\choose\bar a$ is monochromatic with respect to $c$.
\end{definition}

\begin{proposition}\label{lemma def rp independent on monster} The previous definitions do not depend on the choice of the monster (or just an $\aleph_0$-saturated) model, i.e.\ \deerp\ and \edeerp\ are indeed properties of $T$. 
\end{proposition}
\begin{proof}
We fix the following notation. For any finite $\bar a\subseteq\bar b$ in a model of $T$, let $\bar x'$ be some variables corresponding to $\bar b$ and denote by $V_{\bar a,\bar b}$ the set of all $\bar x\subseteq\bar x'$ corresponding to the elementary copies of $\bar a$ within $\bar b$. Note that $V_{\bar a,\bar b}$ depends only on $\tp(\bar a)$, $\tp(\bar b)$ and the choice of $\bar x'$. 

In order to see the \deerp \ case, it is enough to note that (even assuming only that $\C$ is $\aleph_0$-saturated) the statement defining \deerp \ inside $\C$ is equivalent to the following condition: for any finite $\bar a \subseteq \bar b$ (in any model of $T$), any formula $\varphi(\bar x')\in \tp(\bar b)$, and any formulae $\varphi_0(\bar x,\bar y),\dots,\varphi_{n-1}(\bar x,\bar y)$ without parameters with $\bar x$ corresponding to $\bar a$:  
$$T \vdash (\forall \bar y)(\exists \bar x')\left(\varphi(\bar x') \wedge \bigwedge_{\bar x_1,\bar x_2 \in V_{\bar a,\bar b}}\bigwedge_{i<n}(\varphi_i(\bar x_1,\bar y) \leftrightarrow \varphi_i(\bar x_2,\bar y)) \right).$$

We now turn to \edeerp. Suppose that $\C$ is a monster (or $\aleph_0$-saturated) model which satisfies the property given in the externally definable case of Definition \ref{definition def ramsey property}. It suffices to prove that any monster (or $\aleph_0$-saturated) model $\C^*$  such that $\C^*\prec\C$ or $\C^*\succ\C$ satisfies it as well.

Let $\C^*\prec\C$. Consider any finite $\bar a\subseteq \bar b\subseteq\C^*$, $n<\omega$, and an externally definable coloring $c^*:{\C^*\choose\bar a}\to 2^n$ given by formulae $\varphi_0(\bar x,\bar y),\dots,\varphi_{n-1}(\bar x,\bar y)$ and a type $p^*(\bar y)\in S_{\bar y}(\C^*)$ (see Remark \ref{remark ext def coloring}). 
By Fact \ref{fact str heir exists}, let $p(\bar y)\in S_{\bar y}(\C)$ be a strong heir extension of $p^*(\bar y)$, and let $c:{\C\choose\bar a}\to 2^n$ be the externally definable extension of $c^*$ given by $\varphi_0(\bar x,\bar y),\dots,\varphi_{n-1}(\bar x,\bar y)$ and  $p(\bar y)$. For a color $\varepsilon\in 2^n$ consider the formula $\theta_\varepsilon(\bar x',\bar y):= \bigwedge_{\bar x\in V_{\bar a,\bar b}}\bigwedge_{i<n}\varphi_i(\bar x,\bar y)^{\varepsilon(i)}$. Note that for $\bar b'\in{\C\choose\bar b}$, $c[{\bar b'\choose\bar a}]=\{\varepsilon\}$ iff $\theta_\varepsilon(\bar b',\bar y)\in p(\bar y)$. By assumption, there is a color $\varepsilon\in 2^n$ and $\bar b'\in{\C\choose\bar b}$ such that $c[{\bar b'\choose\bar a}]=\{\varepsilon\}$, hence $\theta_\varepsilon(\bar b',\bar y)\in p(\bar y)$.
Since $p^*(\bar y)\subseteq p(\bar y)$ is a strong heir extension, there is $\bar b''\subseteq\C^*$ such that $\bar b''\equiv\bar b'$ and $\theta_\varepsilon(\bar b'',\bar y)\in p^*(\bar y)$. But this means that $\bar b''\in{\C^*\choose\bar b}$ and $c[{\bar b''\choose\bar a}]=\{\varepsilon\}$, so $\bar b''\choose\bar a$ is monochromatic with respect to $c$, and hence with respect to $c^*$. 
The case $\C^* \succ \C$ is very easy.
\end{proof}

The following remark describes the connections between the introduced notions. 
Both implications below are strict: the lack of the converse of the first implication is witnessed by the theory of the random graph (see Example \ref{example random graph}) or by $T:=ACF_0$ with named constants from the algebraic closure of $\mathbb{Q}$ (see example \ref{example ACF0}), and the second one by the theory of a certain random hypergraph (see Example \ref{example R2 R4}).

\begin{remark}\label{remark relations between eerps} For every theory $T$, \eerp$\implies$ \edeerp$\implies$ \deerp.
\end{remark}

We now turn to dynamical characterizations of theories with \deerp\ and \edeerp. 
Let $T$ be a theory.
In order to unify some proofs below, it is convenient to work with local versions of $[E]DEERP$. We say that a finitary type $q \in S(\emptyset)$ has $[E]DEERP$ if for some (equivalently, any) $\bar a\models q$, any finite $\bar b\subseteq\C$ containing $\bar a$, any $n<\omega$ and any [externally] definable coloring $c:{\C\choose \bar a}\to 2^n$ there exists $\bar b'\in{\C\choose\bar b}$ such that $\bar b'\choose\bar a$ is monochromatic with respect to $c$. Clearly, a theory $T$ has $[E]DEERP$ iff every finitary type has it.
For $q\in S_{\bar x}(T)$, a type $p\in S_{\bar c}(\C)$ is {\em $q$-invariant} if for every formula $\varphi(\bar x,\bar y)$ and $\bar a,\bar a'\models q$: $\varphi(\bar a,\bar y)\in p$ iff $\varphi(\bar a',\bar y)\in p$ holds. The set of all $q$-invariant types from $S_{\bar c}(\C)$ is denoted by $\Inv_{q,\bar c}(\C)$. Clearly, $\Inv_{q,\bar c}(\C)$ is a subflow of $S_{\bar c}(\C)$, so for every $\eta\in\EL(S_{\bar c}(\C))$:
\begin{equation}\tag{$\dagger$}\label{equation inv q}
\eta[\Inv_{q,\bar c}(\C)]\subseteq\Inv_{q,\bar c}(\C).
\end{equation}

\begin{theorem}\phantomsection\label{theorem characterization of dpeerp} 
\begin{enumerate}[label=(\roman*), align=left, leftmargin=*, labelsep=-4pt]
\item A finitary type $q$ has \edeerp\ iff there is $\eta\in\EL(S_{\bar c}(\C))$ such that $\Im(\eta)\subseteq\Inv_{q,\bar c}(\C)$.
\item $T$ has \edeerp\ iff there exists $\eta\in \EL(S_{\bar c}(\C))$ such that  $\Im(\eta)\subseteq\Inv_{\bar c}(\C)$.
\end{enumerate}
\end{theorem}

\begin{proof}
(i) ($\Rightarrow$) Suppose that $q$ has \edeerp\ and take $\bar a\models q$. 
Fix a finite $\bar b\subseteq\C$ containing $\bar a$, $L(\bar a)$-formulae
$\varphi_i(\bar a,\bar y)$ and types
$p_i\in S_{\bar c}(\C)$, $i<n$.
Consider the externally definable coloring $c:{\C\choose\bar a}\to 2^n$ given
via $\varphi_i(\bar x,\bar y)$'s and $p_i$'s.
By \edeerp\ there is $\bar b'\in{\C\choose\bar b}$ with $\bar b'\choose\bar a$
being monochromatic; take $\sigma\in\Aut(\C)$ such that $\sigma(\bar b')=\bar
b$. Then for $i<n$ and $\bar a'\in{\bar b\choose \bar a}$ we have
$\varphi_i(\bar a,\bar y)\in \sigma(p_i)$ iff $\varphi_i(\bar a',\bar y)\in
\sigma(p_i)$.

Consider the family $\mathcal F$ of pairs $(\bar b,\{(\varphi_i(\bar
y),p_i)\}_{i<n})$, for finite $\bar b\supseteq\bar a$, $n<\omega$,
$p_i\in S_{\bar c}(\C)$ and $\varphi_i(\bar y)\in L(\bar a)$ for $i<n$.
We order $\mathcal F$
naturally by:
$$(\bar b,\{(\varphi_i(\bar y),p_i)\}_{i<n})\ \leqslant\ (\bar
b',\{(\varphi_i'(\bar y),p_i')\}_{i<n'})$$
iff $\bar b\subseteq\bar b'$, $n\leqslant n'$ and $\{(\varphi_i(\bar
y),p_i)\}_{i<n}\subseteq \{(\varphi_i'(\bar y),p_i')\}_{i<n'}$; clearly,
$\mathcal F$ is directed by $\leqslant$. 
Consider a net $(\sigma_f)_{f\in\mathcal F}$ of automorphisms, where
each $\sigma_f$ is chosen by the previous paragraph for $f\in\mathcal F$. It is easy to check that any accumulation point $\eta\in
EL(S_{\bar c}(\C))$ of this net satisfies our requirements.

($\Leftarrow$) Let $\eta\in\EL(S_{\bar c}(\C))$ be such that $\Im(\eta)\subseteq\Inv_{q,\bar c}(\C)$. 
Take $\bar a\models q$ and finite $\bar b\subseteq\C$ containing $\bar a$, $n<\omega$, and an externally definable coloring $c:{\C\choose\bar a}\to 2^n$ given by formulae $\varphi_i(\bar x,\bar y)$, $i<n$, and type $p\in S_{\bar c}(\C)$ (see Remark \ref{remark ext def coloring}). Since $\eta(p)$ is $q$-invariant, 
$(\varphi_i(\bar a,\bar y)\leftrightarrow\varphi_i(\bar a',\bar y))\in\eta(p)$ holds for all ${\bar a'\in{\bar b\choose\bar a}}$ and $i<n$. This is an open condition on $\eta$, so there is $\sigma\in\Aut(\C)$ such that
$(\varphi_i(\bar a,\bar y)\leftrightarrow\varphi_i(\bar a',\bar y))\in\sigma(p)$ for all ${\bar a'\in{\bar b\choose\bar a}}$ and $i<n$. For $\bar b':=\sigma^{-1}(\bar b)$ we get that $\bar b'\choose\bar a$ is monochromatic with respect to $c$.

(ii) ($\Leftarrow$) is clear by (i), so we prove ($\Rightarrow$). Suppose that $T$ has \edeerp. By (i), for every finitary type $q$, choose $\eta_q\in\EL(S_{\bar c}(\C))$ with $\Im(\eta_q)\subseteq\Inv_{q,\bar c}(\C)$. For a finite set $S=\{q_1,\dots,q_n\}$ of finitary types from $S(\emptyset)$ put $\eta_S:=\eta_{q_1}\circ\dots\circ\eta_{q_n}$ (the order of the composition is irrelevant).
By the choice of $\eta_{q_i}$ and (\ref{equation inv q}), we obtain $\Im(\eta_S)\subseteq\bigcap_{i=1}^n\Inv_{q_i,\bar c}(\C)$. Let $S$ be the family of all finite sets of finitary types in $S(\emptyset)$. Direct $S$ by inclusion, and let $\eta$ be an accumulation point of the obtained net of the $\eta_S$'s. One easily concludes that $\Im(\eta)\subseteq\Inv_{\bar c}(\C)$.
\end{proof}

\begin{corollary}\label{corollary dpeerp implies trivial ellis}
If $T$ has \edeerp,\ then any minimal left ideal $\M\lhd\EL(S_{\bar c}(\C))$ is trivial, hence $u\M$ (i.e.\ the Ellis group of $T$) and $\GalL(T) =\GalKP(T)$ are trivial as well.
\end{corollary}
\begin{proof}
Assume that $T$ has \edeerp. Let $\M\lhd\EL(S_{\bar c}(\C))$ be a minimal left ideal and let $\eta_0\in\M$ be arbitrary. By Theorem \ref{theorem characterization of dpeerp}, there exists $\eta\in\EL(S_{\bar c}(\C))$ with $\Im(\eta)\subseteq \Inv_{\bar c}(\C)$. Set $\eta_1=\eta\eta_0$; clearly $\eta_1\in\M$ and $\Im(\eta_1)\subseteq\Inv_{\bar c}(\C)$. Then $\Aut(\C)\eta_1=\{\eta_1\}$, so $\{\eta_1\}$ is a minimal subflow. Therefore, $\M=\{\eta_1\}$, and so $u\M$ is trivial. Triviality of $\Gal_L(T)$ follows from the existence of an epimoprhism $u\M \to \Gal_L(T)$ found in \cite{KPRz}.
\end{proof}


\begin{theorem}\phantomsection\label{theorem characterization of deerp} 
\begin{enumerate}[label=(\roman*), align=left, leftmargin=*, labelsep=-4pt]
\item For a finitary type $q$, TFAE:
\begin{enumerate}[label=(\arabic*), align=right, leftmargin=*, labelsep=4pt]
\item $q$ has \deerp.
\item There is an element $\eta \in \EL(S_{\bar c}(\C))$ mapping all realized types in $S_{\bar c}(\C)$ to $\Inv_{q,\bar c}(\C)$.
\item $\Inv_{q,\bar c}(\C) \ne \emptyset$
\end{enumerate}

\item For a theory $T$, TFAE:
\begin{enumerate}[label=(\arabic*), align=right, leftmargin=*,labelsep=4pt]
\item $T$ has \deerp,
\item There is an element $\eta \in \EL(S_{\bar c}(\C))$ mapping all realized types in $S_{\bar c}(\C)$ to $\Inv_{\bar c}(\C)$,
\item $T$ is extremely amenable, that is $\Inv_{\bar c}(\C) \ne \emptyset$.
\end{enumerate}
\end{enumerate}
\end{theorem}

\begin{proof}
(i) For $(1)\Rightarrow(2)$ it suffices to note that by working with the realized types in $S_{\bar c}(\C)$ in place of arbitrary types $p_i\in S_{\bar c}(\C)$, the construction in the proof of Theorem \ref{theorem characterization of dpeerp}(i,$\Rightarrow$) yields $\eta$ mapping all realized types in $S_{\bar c}(\C)$ to $\Inv_{q,\bar c}(\C)$. The implication $(2)\Rightarrow(3)$ is obvious, so it remains to prove $(3)\Rightarrow(1)$.

Assume that $\Inv_{q,\bar c}(\C)\neq\emptyset$ and take any $p\in \Inv_{q,\bar c}(\C)$. Fix $\bar a\models q$. Consider any finite $\bar b\subseteq\C$ containing $\bar a$, $n<\omega$, and a definable coloring $c:{\C\choose\bar a}\to 2^n$ given by 
$\varphi_i(\bar x,\bar d)$, $i<n$, with $\bar d$ finite. 
Let $p_0(\bar y)\in S_{\bar y}(\bar b)$ be the restriction of $p$, where $\bar y$ are the variables corresponding to $\bar d$ within $\bar c$, and note that, by $q$-invariance of $p$,  $p_0$ contains $\varphi_i(\bar a,\bar y)\leftrightarrow\varphi_i(\bar a_0,\bar y)$ for all $\bar a_0\in{\bar b\choose\bar a}$ and $i<n$. Take $\bar d_0\models p_0$ and $f\in\Aut(\C)$ such that $f(\bar d_0)=\bar d$; let $\bar a'=f(\bar a)$ and $\bar b'=f(\bar b)$. Then $\tp(\bar d/\bar b')$ contains $\varphi_i(\bar a',\bar y)\leftrightarrow\varphi_i(\bar a_0,\bar y)$ for all $\bar a_0\in{\bar b'\choose\bar a'}$ and $i<n$. This means that ${\bar b'\choose\bar a}$ is monochromatic with respect to $c$, as ${\bar b'\choose\bar a}={\bar b'\choose\bar a'}$.

(ii) (2)$\Rightarrow$(3) is obvious, and (3)$\Rightarrow$(1) follows by (i). To prove (1)$\Rightarrow$(2), for every finite $\bar a\subseteq\C$ choose, by (i), $\eta_{\bar a}\in\EL(S_{\bar c}(\C))$ mapping all realized types in $S_{\bar c}(\C)$ to $\Inv_{\tp(\bar a),\bar c}(\C)$. Direct all finite tuples $\bar a$ from $\C$ by inclusion and take any accumulation point $\eta$ of the net of the $\eta_{\bar a}$'s. It is easy to see that $\eta$ maps all realized types in $S_{\bar c}(\C)$ to $\Inv_{\bar c}(\C)$.
\end{proof}

 Proposition \ref{lemma def rp independent on monster} together with the observation that in the proofs of Theorems \ref{theorem characterization of dpeerp} and \ref{theorem characterization of deerp} it is enough to assume only $\aleph_0$-saturation and strong $\aleph_0$-homogeneity of $\C$ and consider \edeerp\ [resp. \deerp] only in the chosen model $\C$ yield that both the existence of $\eta\in \EL(S_{\bar c}(\C))$ with $\Im(\eta)\subseteq\Inv_{\bar c}(\C)$ as well as extreme amenability of $T$ are independent of the choice of the $\aleph_0$-saturated and strongly $\aleph_0$-homogeneous model $\C$. On the other hand, the fact that extreme amenability of $T$ is absolute was easily observed directly in \cite{KPT}, so, using the above observation on the proof of Theorem \ref{theorem characterization of deerp}, we get the first part of Proposition \ref{lemma def rp independent on monster}, i.e.\ absoluteness of \deerp\ (at least for $\aleph_0$-saturated and strongly $\aleph_0$-homogeneous models).

Note that \deerp\ implies that $\GalL(T)$ is trivial, because, by Theorem \ref{theorem characterization of deerp}, $T$ is extremely amenable and so $\GalL(T)$ is trivial by 
\cite[Proposition 4.2]{HKP}. However, in contrast with \edeerp,\ Examples \ref{example R2 R4} and \ref{example R2 R4 Ps} show that a theory with \deerp\ need not have trivial or even finite Ellis group.

In Corollaries \ref{corollary c to finite z m-ary} and \ref{corollary: c to finite z for image in inv types}, we saw that in an $(m+1)$-ary theory, in order to compute the Ellis group or to test the existence of an element in the Ellis semigroup with image contained in invariant types, we can restrict ourselves to the $\Aut(\C)$-flow $S_m(\C)$. The next proposition shows  a similar behavior of \edeerp.

\begin{proposition}\label{proposition C} If $T$ is $(m+1)$-ary and each type from $S_m(T)$ has \edeerp, then $T$ has \edeerp.
\end{proposition}
\begin{proof}
The construction from Theorem \ref{theorem characterization of dpeerp}(ii,$\Rightarrow$), but working with types $q_i\in S_{m}(T)$ instead of arbitrary finitary types, yields $\eta\in\EL(S_{\bar c}(\C))$ such that $\Im(\eta)\subseteq\bigcap_{q\in S_m(T)}\Inv_{q,\bar c}(\C)$. By $(m+1)$-arity of $T$, this intersection equals $\Inv_{\bar c}(\C)$, so $T$ has \edeerp\ by Theorem \ref{theorem characterization of dpeerp}(ii).
\end{proof}

\subsection{Separately finite externally definable elementary embedding Ramsey degree}

By a direct analogy with the introduced notions of elementary embedding Ramsey properties, one may define the notions of finite elementary embedding Ramsey degrees. We introduce here theories with separately finite \eerdeg\ and a weak form of separately finite \edeerdeg, as they play an essential role in this paper. 
In order to state this weak version, we need to use  externally definable $\Delta$-colorings (see Definition \ref{definition of definable colorings} and Remark \ref{remark ext def coloring}). 

\begin{definition}\phantomsection\label{definition strong ramsey degree}
\begin{enumerate}[label=(\alph*), align=right, leftmargin=*]
\item A theory $T$ has {\em separately finite \eerdeg} ({\em separately finite elementary embedding Ramsey degree}) if for any finite tuple $\bar a$ there exists $l<\omega$ such that for any finite tuple $\bar b\subseteq\C$ containing $\bar a$, $r<\omega$, and coloring $c:{\C\choose\bar a}\to r$ there exists $\bar b'\in{\C\choose\bar b}$ such that $\#c[{\bar b'\choose\bar a}]\leqslant l$.

\item
 A theory $T$ has {\em separately finite \edeerdeg} ({\em  separately finite externally definable elementary embedding Ramsey degree}) if for any finite tuple $\bar a$, finite set of formulae $\Delta$ in variables $\bar x$ (where $|\bar x|=|\bar a|$) and $\bar y$, and $q \in S_{\bar y}(\emptyset)$, there exists $l<\omega$ such that for any finite tuple $\bar b\subseteq\C$ containing $\bar a$, $n<\omega$, and externally definable $(\Delta,q)$-coloring $c:{\C\choose\bar a}\to 2^n$ there exists $\bar b'\in{\C\choose\bar b}$ such that $\#c[{\bar b'\choose\bar a}]\leqslant l$. By Remark \ref{remark ext def coloring}(iii), wlog we can assume here that $|\bar y|=|\bar c|$ and $q := \tp(\bar c/\emptyset)$, so $c$ is then an externally definable $\Delta$-coloring.
\end{enumerate}
\noindent
The word ``separately'' is used here to stress that $l$ depends on $\bar a$ (in (b), $l$ also depends on $(\Delta,q)$, or only on $\Delta$ if $q = \tp(\bar c/\emptyset)$). The least such number will be called the {\em [externally definable] elementary embedding Ramsey degree} of $\bar a$ [with respect to $(\Delta,q)$, or $\Delta$ if $q= \tp(\bar c/\emptyset)$].
\end{definition}

As in Remark \ref{lemma equivalents of eerps}, 
a theory $T$ has sep.\ fin.\ \eerdeg\ iff the following holds: For any finite tuple $\bar a$ there exists $l<\omega$ such that for any finite tuple $\bar b\subseteq \C$ and any $r<\omega$ there exists a finite subset $C\subseteq\C$ such that for any coloring $c:{C\choose\bar a}\to r$ there exists $\bar b'\in{C\choose\bar b}$ such that $\#c[{\bar b'\choose\bar a}]\leqslant l$. 
By using this characterization, as in Remark \ref{remark rp ind in sat and str hom} one shows that the property of having sep.\ fin.\ \eerdeg\ does not depend on the choice of the monster (or just an $\aleph_0$-saturated) model $\C$, 
i.e.\ it is indeed a property of the theory. The counterpart of Remark \ref{remark fraisse eerp} also holds:

\begin{remark}\label{remark fraisse fin sep eerdeg}
Definition \ref{definition strong ramsey degree}(a) generalizes the definition of sep.\ fin.\ embedding Ramsey degree for Fra\"iss\'e structures in the following sense: If $K$ is an $\aleph_0$-saturated Fra\"iss\'e structure, then it has sep.\ fin.\ embedding Ramsey degree iff $\mathrm{Th}(K)$ has sep.\ fin.\ \eerdeg.\qed
\end{remark}

The property of having sep.\ fin.\ \edeerdeg\ is absolute, too, i.e.\ does not depend on the choice of the monster model. To see this, one should follow the lines of the proof of absoluteness of \edeerp\ in Proposition \ref{lemma def rp independent on monster} with the following modifications. 
Fix $\bar a$, $\Delta$ and $q$ as in the definition of \edeerdeg. For any finite $\bar b\supseteq\bar a$ consider the same $V_{\bar a,\bar b}$ as in the proof of Proposition \ref{lemma def rp independent on monster}. For any $n<\omega$ and formulae $\varphi_0(\bar x,\bar y),\dots,\varphi_{n-1}(\bar x,\bar y)\in\Delta$ consider the formula:
$$\theta(\bar x',\bar y_0,\dots,\bar y_{n-1}):= \bigvee_{\substack{V\subseteq V_{\bar a,\bar b}\\|V|=l}} 
\bigwedge_{\bar x_0\in V_{\bar a,\bar b}}\bigvee_{\bar x\in V}\bigwedge_{i<n} \varphi_i(\bar x_0,\bar y_i)\leftrightarrow\varphi_i(\bar x,\bar y_i),$$
where $l$ is the \edeerdeg\ of $\bar a$ with respect to $(\Delta,q)$ computed in $\C$, and the $\bar y_i$'s are pairwise disjoint copies of $\bar y$. 
Note that for an externally definable $(\Delta,q)$-coloring $c:{\C\choose\bar a}\to 2^n$ given by 
$\varphi_i(\bar x,\bar y)$ and $p_i(\bar y) \in S_q(\C)$, $i<n$, for $\bar b'\in{\C\choose\bar b}$ we have: $\#c[{\bar b'\choose\bar a}]\leqslant l$ iff 
$\theta(\bar b',\bar y_0,\dots,\bar y_{n-1})$ is consistent with $\bigcup_{i<n}p_i(\bar y_i)$, iff 
$\bigcup_{i<n} p_i(\bar y_i)\vdash \theta(\bar b',\bar y_0,\dots,\bar y_{n-1})$. Using this, one can argue as in the proof of Proposition \ref{lemma def rp independent on monster}, starting from 
$p_i^*(\bar y) \in S_q(\C^*)$, $i<n$, choosing a completion $p^*(\bar y_0,\dots,\bar y_{n-1}) \in S(\C^*)$ of $\bigcup_{i<n}p_i^*(\bar y_i)$,  and using a strong heir extension $p \in S(\C)$ of $p^*$.

Note that the diagram in Remark \ref{remark relations between eerps} expands to:
\begin{center}\begin{tabular}{ccccc}
\eerp & $\Rightarrow$ & \edeerp & $\Rightarrow$ & \deerp \\
$\Downarrow$ && $\Downarrow$\\
sep.\ fin.\ \eerdeg & $\Rightarrow$ & sep.\ fin.\ \edeerdeg
\end{tabular}\end{center}

All the implications written above are strict. The converse of the first vertical implication fails for the random graph; the converse of the second one fails  for the hypergraph from Example \ref{example R2 R4}; the converse of the lower horizontal implication fails e.g.\ in $ACF_0$ (see Example \ref{example ACF0}).

We now prove the counterpart of Theorem \ref{theorem characterization of dpeerp} for sep.\ finite \edeerdeg,\ which is one of the main results of this paper. For a formula $\varphi(\bar x, \bar y)$ put $\varphi^{\textrm{opp}}(\bar y, \bar x):= \varphi(\bar x, \bar y)$. For $\Delta= \{\varphi_i(\bar x,\bar y)\}_{i<k}$ put $\Delta^{\textrm{opp}}: = \{\varphi_i^{\textrm{opp}}(\bar y,\bar x)\}_{i<k}$ .

\begin{theorem}\label{theorem characterization of sep fin dpeerdeg} $T$ has sep.\ fin.\ \edeerdeg\ iff for every  $\Delta= \{\varphi_i(\bar x,\bar y)\}_{i<k}$ and $\bar p=\{p_j\}_{j<m}\subseteq S_{\bar y}(T)$ there exists $\eta\in\EL(S_{\bar c,\Delta}(\bar p))$ such that $\Im(\eta)$ is finite.
\end{theorem}

\begin{proof}
$(\Rightarrow)$ 
Suppose that $T$ has sep.\ fin.\ \edeerdeg. Fix any $\Delta= \{\varphi_i(\bar x,\bar y)\}_{i<k}$ and $\bar p=\{p_j\}_{j<m}\subseteq  S_{\bar y}(T)$, where $\bar x$ corresponds to $\bar c$ and $\bar y$ is finite. Fix $\bar a_j\models p_j$ for $j<m$, and put $\bar a=(\bar a_0,\dots,\bar a_{m-1})$. 
Let $\Delta':= \{\phi_{i,j}(\bar z,\bar x) \mid i<k, j<m\}$, where $\phi_{i,j}(\bar z,\bar x):= \varphi_i(\bar x, \bar y_j)$ and $\bar z = (\bar y_0,\dots,\bar y_{m-1})$ with $\bar y_i$ corresponding to $\bar a_i$. Let $l<\omega$ be the \edeerdeg\ of $\bar a$ with respect to $\Delta'$.

For any $n<\omega$, $q_0,\dots,q_{n-1}\in S_{\bar c}(\C)$, and finite $\bar b\supseteq\bar a$ consider the externally definable $\Delta'$-coloring $c:{\C\choose\bar a}\to 2^{kmn}$ given by:
$$c(\bar a')(i,j,t)=\left\{
\begin{array}{cl}
1, &  \varphi_i(\bar x,\bar a_j')\in q_t\\
0, &  \lnot\varphi_i(\bar x,\bar a_j')\in q_t
\end{array}
\right.,$$
where $\bar a_j'\subseteq\bar a'$ is the subtuple corresponding to $\bar y_j$. 
(Note that $c$ is indeed an externally definable $\Delta'$-coloring via the formulae $\phi_{i,j,t}(\bar z,\bar x):= \varphi_i(\bar x,\bar y_j)$ and types $r_{i,j,t}(\bar x):= q_t(\bar x)$ for $i<k$, $j<m$, and $t<n$.) By the choice of $l$, we can find $\bar b'\in{\C\choose\bar b}$ such that $\#c[{\bar b'\choose\bar a}]\leqslant l$. Let $\sigma_{\bar b,\bar q}\in\Aut(\C)$ be such that $\sigma_{\bar b,\bar q}(\bar b')=\bar b$; here $\bar q$ denotes $(q_0,\dots,q_{n-1})$. Consider the naturally directed family $\mathcal F$ of all pairs $(\bar b,\bar q)$, and let $\eta$ be an accumulation point of the net $(\sigma_{\bar b,\bar q})_{(\bar b,\bar q)\in\mathcal F}$ in $\EL(S_{\bar c,\Delta}(\bar p))$. 

We claim that $\Im(\eta)$ is finite. Suppose not, i.e.\ we can find $q_0,q_1,\ldots\in S_{\bar c}(\C)$ such that the $\eta(\hat q_t)$'s are pairwise distinct, where $\hat q_t \in S_{\bar c,\Delta}(\bar p)$ denotes the restriction of $q_t$. For each $t \ne t'$ the fact that $\eta(\hat q_{t}) \ne \eta(\hat q_{t'})$ is witnessed  by one of the formulae $\varphi_0(\bar x,\bar y),\dots,\varphi_{k-1}(\bar x,\bar y)$ and a realization of one of the types $p_0,\dots,p_{m-1}$. By Ramsey theorem, passing to a subsequence, we can assume that there are $i_0<k$ and $j_0<m$ such that for each $t \ne t'$ the fact that $\eta(\hat q_{t}) \ne \eta(\hat q_{t'})$ is witnessed  by $\varphi_{i_0}(\bar x,\bar y)$ and a realization of $p_{j_0}$.
Let $n>2^l$, and for $t<t'<n$ denote by $\bar a_{j_0}^{t,t'}$ a realization of $p_{j_0}$ such that $\varphi_{i_0}(\bar x, \bar a_{j_0}^{t,t'})\in\eta(\hat q_t)$ iff $\lnot\varphi_{i_0}(\bar x, \bar a_{j_0}^{t,t'})\in\eta(\hat q_{t'})$. Note that this is an open condition on $\eta$, so if we let $\bar a^{t,t'}$ to be a copy of $\bar a$ extending $\bar a_{j_0}^{t,t'}$, then we can find a tuple $(\bar b,\bar q)$ greater than $(\bar a^\frown(\bar a^{t,t'})_{t<t'<n},(q_0,\dots,q_{n-1}))$ in $\mathcal F$ such that $\varphi_{i_0}(\bar x, \bar a_{j_0}^{t,t'})\in\sigma_{\bar b,\bar q}(\hat q_t)$ iff $\lnot\varphi_{i_0}(\bar x, \bar a_{j_0}^{t,t'})\in\sigma_{\bar b,\bar q}(\hat q_{t'})$. By the choice of $\sigma_{\bar b,\bar q}$, $\bar b'=\sigma_{\bar b,\bar q}^{-1}(\bar b)$ satisfies $\#c[{\bar b'\choose\bar a}]\leqslant l$.

For $\bar a'\in{\bar b'\choose\bar a}$ let $S(\bar a'):=\{t<n\mid \varphi_{i_0}(\bar x,\bar a_{j_0}')\in q_t\}$. Since $\#c[{\bar b'\choose\bar a}]\leqslant l$, we have $\#\{S(\bar a')\mid \bar a'\in{\bar b'\choose\bar a}\}\leqslant l$. 
Put $l'=\#\{S(\bar a')\mid\bar a'\in{\bar b'\choose\bar a}\}$, so $l'\leqslant l$, and let $\{S(\bar a')\mid\bar a'\in{\bar b'\choose\bar a}\}=\{S_0,\dots,S_{l'-1}\}$. Note that the mapping $f:n\to\mathcal P(\{S_0,\dots,S_{l'-1}\})$ given by $f(t)=\{S_u\mid u<l', t\in S_u\}$ is injective. Indeed, for $t<t'<n$, by the choice of $\bar b$, we have $\bar a^{t,t'}\subseteq\bar b$, so $\bar a':=\sigma_{\bar b,\bar q}^{-1}(\bar a^{t,t'})\subseteq\bar b'$. 
But since $\varphi_{i_0}(\bar x,\bar a_{j_0}^{t,t'})\in\sigma_{\bar b,\bar q}(\hat q_t)$ iff $\lnot\varphi_{i_0}(\bar x, \bar a_{j_0}^{t,t'})\in\sigma_{\bar b,\bar q}(\hat q_{t'})$, we have $\varphi_{i_0}(\bar x,\bar a_{j_0}')\in\hat q_t$ iff $\lnot\varphi_{i_0}(\bar x, \bar a_{j_0}')\in\hat q_{t'}$, i.e.\ $\varphi_{i_0}(\bar x, \bar a_{j_0}')\in q_t$ iff $\lnot\varphi_{i_0}(\bar x, \bar a_{j_0}')\in q_{t'}$, hence $t\in S(\bar a')$ iff $t'\notin S(\bar a')$. Thus, $S(\bar a')$ distinguishes $f(t)$ and $f(t')$, and $f$ is injective. Therefore, $n\leqslant 2^{l'}\leqslant 2^l$, which contradicts the choice of $n$. 

$(\Leftarrow)$ Suppose now that the right hand side of the theorem holds. Take a finite tuple $\bar a\subseteq\C$ and a finite set of formulae $\Delta$ in variables $\bar x, \bar y$, 
where $\bar y$ corresponds to $\bar c$; 
let $\#\Delta= k$ and set $p=\tp(\bar a)$. By assumption, we may find $\eta\in\EL(S_{\bar c,\Delta^{\textrm{opp}}}(p))$ with $\Im(\eta)$ finite, say $\#\Im(\eta)= t$. We claim that $l:=2^{kt}$ satisfies our requirements for $\bar a$ and $\Delta$.

Fix a finite $\bar b$ containing $\bar a$, $n<\omega$, and an externally definable $\Delta$-coloring $c:{\C\choose\bar a}\to 2^n$ given by $\varphi_i(\bar x,\bar y)\in\Delta$ and $q_i(\bar y)\in S_{\bar c}(\C)$, $i<n$. 
Actually, we may assume that the types $q_i(\bar y)$'s are from $S_{\bar c,\Delta^{\textrm{opp}}}(\C)$ by taking their appropriate restrictions, as the coloring $c$ depends only on these restrictions. 
For each $\bar a'\in{\C\choose\bar a}$ find $\varepsilon_{\bar a'}\in 2^n$ such that $\varphi_i(\bar a',\bar y)^{\varepsilon_{\bar a'}(i)}\in\eta(q_i(\bar y))$ for all $i<n$. Note that we have only $2^{kt}$ possibilities for $\varepsilon_{\bar a'}$, as we only have $k$-many formulae in $\Delta$ and $t$-many types in $\Im(\eta)$ (namely, if $\varphi_i(\bar x,\bar y)=\varphi_{i'}(\bar x,\bar y)$ and $\eta(q_i(\bar y))=\eta(q_{i'}(\bar y))$, then, by consistency, we must have $\varepsilon_{\bar a'}(i)=\varepsilon_{\bar a'}(i')$). Consider the open condition on $\eta$ given by:
$\varphi_i(\bar a',\bar y)^{\varepsilon_{\bar a'}(i)}\in\eta(q_i(\bar y))$ for all $\bar a'\in{\bar b\choose\bar a}$ and $i<n$.
Let $\sigma\in\Aut(\C)$ be such that:
$\varphi_i(\bar a',\bar y)^{\varepsilon_{\bar a'}(i)}\in\sigma(q_i(\bar y))$, i.e.\ $\varphi_i(\sigma^{-1}(\bar a'),\bar y)^{\varepsilon_{\bar a'}(i)}\in q_i(\bar y)$, for all $\bar a'\in{\bar b\choose\bar a}$ and $i<n$.
Put $\bar b'=\sigma^{-1}(\bar b)$; we have:
$\varphi_i(\bar a',\bar y)^{\varepsilon_{\bar a'}(i)}\in q_i(\bar y)$ for all $\bar a'\in{\bar b'\choose\bar a}$ and $i<n$.
But since we only had $2^{kt}$ possibilities for $\varepsilon$'s, this means that $\# c[{b'\choose\bar a}]\leqslant 2^{kt}$.
\end{proof}

In fact, a slight elaboration on our proof yields concrete bounds.

\begin{corollary}\phantomsection\label{corollary: bounds}
\begin{enumerate}[label=(\roman*), align=right, leftmargin=*]
\item Let  $\Delta= \{\varphi_i(\bar x,\bar y)\}_{i<k}$ and $\bar p=\{p_j\}_{j<m}\subseteq S_{\bar y}(T)$, where $|\bar x|=|\bar c|$ and $\bar y$ is finite. Fix $\bar a_j\models p_j$ for $j<m$, and put $\bar a=(\bar a_0,\dots,\bar a_{m-1})$.  Let $\Delta':= \{\phi_{i,j}(\bar z,\bar x) \mid i<k, j<m\}$, where $\phi_{i,j}(\bar z,\bar x):= \varphi_i(\bar x, \bar y_j)$ and $\bar z = (\bar y_0,\dots,\bar y_{m-1})$. Assume $l<\omega$ is the \edeerdeg\ of $\bar a$ with respect to $\Delta'$. Then there exists $\eta\in\EL(S_{\bar c,\Delta}(\bar p))$ such that $\#\Im(\eta) \leqslant 2^{mkl}$.
\item 
Let $\bar a$ be a finite tuple and $\Delta= \{\varphi_i(\bar x,\bar y)\}_{i<k}$, where $|\bar y|=|\bar c|$. Put $p= \tp(\bar a)$. Assume that there is $\eta\in\EL(S_{\bar c,\Delta^{\textrm{opp}}}(p))$ with $\#\Im(\eta)= t<\omega$. Then the \edeerdeg\ of $\bar a$ with respect to $\Delta$ is bounded by $2^{kt}$.
\end{enumerate}
\end{corollary}

\begin{proof}
(ii) was obtained in the proof of ($\Leftarrow$) in Theorem  \ref{theorem characterization of sep fin dpeerdeg}. To get (i), one needs to elaborate on the proof of ($\Rightarrow$) as follows. We do not use Ramsey theorem. Instead, suppose we have $n$ types $q_0,\dots q_{n-1} \in S_{\bar c}(\C)$ with the $\eta(\hat q_t)$'s pairwise distinct. Then for any $t<t'<n$ we can find $i(t,t')<k$, $j(t,t')<m$, and a realization $\bar a_{j(t,t')}^{t,t'}$ of $p_{j(t,t')}$ such that $\varphi_{i(t,t')}(\bar x, \bar a_{j(t,t')}^{t,t'})\in\eta(\hat q_t)$ iff $\lnot\varphi_{i(t,t')}(\bar x, \bar a_{j(t,t')}^{t,t'})\in\eta(\hat q_{t'})$. Then we continue with obvious adjustments until the definition of $S(\bar a')$. Next, for any $\bar a'\in{\bar b'\choose\bar a}$, $i<k$, and $j<m$ put $S(\bar a')^{i,j}:=\{t<n\mid \varphi_{i}(\bar x, \bar a_{j}')\in q_t\}$; define $S(\bar a') := (S(\bar a')^{i,j})_{i<k,\; j<m}$. Then still $\#\{S(\bar a')\mid \bar a'\in{\bar b'\choose\bar a}\}\leqslant l$, so  $\{S(\bar a')\mid a'\in{\bar b'\choose\bar a}\}=\{S_0,\dots,S_{l'-1}\}$ for some $l'\leqslant l$. 
Let $f:n\to\mathcal (2^{k \times m})^{l'}$ be defined by saying that $f(t)$ is the unique function $\delta: l' \to 2^{k \times m}$ such that for all $u<l'$, $i<k$, $j<m$ one has  $\delta(u)(i,j)=1 \iff t \in S_u^{i,j}$, where $S_u^{i,j}$ is the $(i,j)$-th coordinate of $S_u$. As before, one shows that $f$ is injective, which clearly implies that $n \leqslant 2^{kml}$.
\end{proof}

The fact that the property of having sep.\ fin.\ \edeerdeg\ does not depend on the choice of $\C$ together with the fact that in the proof of Theroem \ref{theorem characterization of sep fin dpeerdeg} we work in the given $\C$ imply that the right hand side of Theorem \ref{theorem characterization of sep fin dpeerdeg} does not depend on the choice of $\C$, too. 
By examining the above proofs, one can also see that it is enough to assume here only that $\C$ is $\aleph_0$-saturated and strongly $\aleph_0$-homogeneous.

The main consequence of Theorem \ref{theorem characterization of sep fin dpeerdeg} will be Corollary \ref{corollary sep fin dpeerdeg implies 0-dim ellis} which leads to many examples of theories with profinite Ellis group.

\subsection{Elementary embedding convex Ramsey property}\label{subsection on eecrp}

We briefly explain here the counterpart of Corollary \ref{corollary extr amen prop of theory} for amenability in place of extreme amenability, leaving some details to the reader.


In \cite{Mo}, Moore characterized amenability of the automorphism group of a Fra\"{i}ss\'{e} structure via the so-called convex Ramsey property. This was generalized to arbitrary structures in \cite{KP}, where the notion of {\em ECRP} (the {\em embedding convex Ramsey property}) for a first order structure is introduced. 
A structure $M$ has {\em ECRP} if for any $\epsilon>0$, any two finite tuples $\bar a\subseteq \bar b\subseteq M$, and any $n<\omega$ there exists a finite subset $C\subseteq M$ such that for any coloring $c:{C\choose\bar a}^{\Aut}\to 2^n$ there exist $k<\omega$, $\lambda_0,\dots,\lambda_{k-1}\in[0,1]$ with $\lambda_0+\dots+\lambda_{k-1}=1$, and $\sigma_0,\dots,\sigma_{k-1}\in\Aut(M)$ such that $\sigma_0(\bar b),\dots,\sigma_{k-1}(\bar b)\subseteq C$ and for any two tuples $\bar a',\bar a''\in{\bar b\choose\bar a}^{\Aut}$ we have:
$$\max_{i<n}\vert\sum_{j<k}\lambda_jc(\sigma_j(\bar a'))(i)- \sum_{j<k}\lambda_jc(\sigma_j(\bar a''))(i)\vert\leqslant\epsilon.$$
Recall that for finite $\bar a\subseteq M$ and $C\subseteq M$,
 ${C\choose \bar a}^{\Aut}$ denotes the set of all $f(\bar a)$ for $f\in\Aut(M)$ with $f(\bar a)\subseteq C$.
Also, recall that if $M$ is strongly $\aleph_0$-homogeneous, then ${C\choose\bar a}^{\Aut}= {C\choose\bar a}$. By \cite[Remark 4.2]{KP}, in the above definition we can equivalently remove the part ``there is a finite subset $C \subseteq M$ such that" and everywhere replace $C$ by $M$.
The following fact was proved in \cite{KP}.

\begin{fact}{\cite[Theorem 4.3]{KP}}\label{fact KP amenable} Let $M$ be a first-order structure. TFAE:
\begin{enumerate}[label=(\arabic*), align=right, leftmargin=*]
\item $\Aut(M)$ is amenable as a topological group;
\item $M$ has ECRP;
\item $M$ has strong ECRP.\qed
\end{enumerate}
\end{fact}

By analogy with \eerp, we introduce the following definition.

\begin{definition}\label{definition eecrp} A theory $T$ has {\em EECRP} (the {\em elementary embedding convex Ramsey property}) if it has a monster model $\C$ with {\em ECRP}.
\end{definition}

A proof of the next lemma is left as an exercise.

\begin{lemma}\label{lemma eecrp independent}
If $M$ and $M^*$ are two elementarily equivalent $\aleph_0$-saturated and strongly $\aleph_0$-homogeneous models, then $M$ has $ECRP$ iff $M^*$ does. Thus, a theory $T$ has $EECRP$ iff some (equiv.\ every) $\aleph_0$-saturated and strongly $\aleph_0$-homogeneous model of $T$ has $ECRP$.\qed
\end{lemma}

By Fact \ref{fact KP amenable} and Lemma \ref{lemma eecrp independent}, we get the main conclusion of Subsection \ref{subsection on eecrp}.

\begin{corollary}\label{corollary amen prop of theory} For a theory $T$, $\Aut(M)$ is amenable (as a topological group) for some $\aleph_0$-saturated and strongly $\aleph_0$-homogeneous model $M\models T$ iff it is amenable for all $\aleph_0$-saturated and strongly $\aleph_0$-homogeneous models of $T$.\qed
\end{corollary}

\subsection{Definable elementary embedding convex Ramsey property}

By Theorem \ref{theorem characterization of deerp}, a theory $T$ is extremely amenably iff it has \deerp. In this subsection, we will get an analogous result for amenable theories (see Definition \ref{definition of amenability}). The corresponding Ramsey property-like condition which describes amenability of a theory is the {\em definable elementary embedding convex Ramsey property}. 

\begin{definition}\label{definition deecrp} A theory $T$ has \deecrp\ (the {\em definable elementary embedding convex Ramsey property}) iff for any $\epsilon>0$, any two finite tuples $\bar a\subseteq\bar b\subseteq\C$,  any $n<\omega$, and any definable coloring $c:{\C\choose\bar a}\to 2^n$ there exist $k<\omega$, $\lambda_0,\dots,\lambda_{k-1}\in[0,1]$ with $\lambda_0+\dots+\lambda_{k-1}=1$, and $\sigma_0,\dots,\sigma_{k-1}\in\Aut(\C)$ such that for any two tuples $\bar a',\bar a''\in{\bar b\choose\bar a}$ we have:
$$\max_{i<n}\vert\sum_{j<k}\lambda_jc(\sigma_j(\bar a'))(i)- \sum_{j<k}\lambda_jc(\sigma_j(\bar a''))(i)\vert\leqslant\epsilon.$$

A theory $T$ has the {\em strong \deecrp} if the previous holds for $\epsilon=0$.
\end{definition}

\begin{remark}\label{lemma deecrp absolutness} 
By a direct argument, one may see that the previous definition does not depend on the choice of the monster model $\C$ of $T$. Moreover, it is enough to ask that $\C$ is $\aleph_0$-saturated and strongly $\aleph_0$-homogeneous.
\end{remark}

\begin{theorem}\label{theorem deecrp} For a theory $T$ the following conditions are equivalent:
\begin{enumerate}
\item $T$ is amenable;
\item $T$ has \deecrp;
\item $T$ has strong \deecrp.
\end{enumerate} 
\end{theorem}
\begin{proof}
(3)$\Rightarrow$(2) is obvious.

(2)$\Rightarrow$(1) Assume that $T$ has \deecrp. 
Toward a contradiction, suppose that for a finite $\bar d$ there is no an invariant, finitely additive probability measure on the Boolean algebra of clopens in $S_{\bar d}(\C)$. By compactness, for some $\epsilon>0$, formulae $\varphi_i(\bar x_i,\bar y)$ and tuples $\bar a_i\equiv\bar a_i'$, $i<n$, there is no measure $\mu\in\frak M_{\bar d}:=\frak M_{\tp(\bar d)}$ (see Subsection \ref{subsection: amenability}) satisfying: 
\begin{equation}\tag{$\dagger$}\label{equation in 3.4}
\vert\mu([\varphi_i(\bar a_i,\bar y)])-\mu([\varphi_i(\bar a_i',\bar y)])\vert\leqslant\epsilon\mbox{ for all $i<n$.}
\end{equation} 
By using dummy variables we may assume that all $\bar x_i$'s are mutually equal, as well as all $\bar a_i$'s; denote them by $\bar x$ and $\bar a$, respectively. 

Set $\bar b=\bar a^\frown\bar a_0'^\frown\dots^\frown\bar a_{n-1}'$. Consider the definable coloring $c$ on ${\C\choose\bar a}$ given by the $\varphi_i(\bar x,\bar d)$'s. By \deecrp,\ there are $k<\omega$, $\lambda_0,\dots,\lambda_{k-1}\in[0,1]$ with $\sum\lambda_i=1$, and $\sigma_0,\dots,\sigma_{k-1}\in\Aut(\C)$ such that for any two tuples $\bar a',\bar a''\in{\bar b\choose\bar a}$ and for any $i<n$ we have
$\vert\sum_{j<k}\lambda_jc(\sigma_j(\bar a'))(i)- \sum_{j<k}\lambda_jc(\sigma_j(\bar a''))(i)\vert\leqslant\epsilon$. 

Consider $\mu= \lambda_0\tp(\sigma_0^{-1}(\bar d)/\C)+\dots+\lambda_{k-1}\tp(\sigma_{k-1}^{-1}(\bar d)/\C)$, defined by:
$$\mu([\phi(\bar y)]):= \sum_{j<k}\{\lambda_j\mid \phi(\bar y)\in\tp(\sigma_j^{-1}(\bar d)/\C)\},$$
for $\phi(\bar y)\in L(\C)$; it is easy to see that $\mu\in\frak M_{\bar d}$, as well as that for $i<n$, $j<k$, and $\bar a'\in{\bar b\choose\bar a}$ we have $c(\sigma_j(\bar a'))(i)=1$ iff $\varphi_i(\bar a',\bar y)\in\tp(\sigma_j^{-1}(\bar d)/\C)$, and thus, $\sum_{j<k}\lambda_j c(\sigma_j(\bar a'))(i)= \mu([\varphi_i(\bar a',\bar y)])$. Therefore, $\mu$ satisfies (\ref{equation in 3.4}). A contradiction.

(1)$\Rightarrow$(3) Assume that $T$ is amenable. Fix $\bar a\subseteq\bar b\subseteq\C$, $n<\omega$, and a definable coloring $c:{\C\choose\bar a}\to 2^n$ given by $\varphi_0(\bar x,\bar d),\dots,\varphi_{n-1}(\bar x,\bar d)$. By amenability, we can find an invariant, finitely additive probability measure $\mu$ on the clopens of $S_{\bar d}(\C)$. Let ${\bar b\choose\bar a}= \{\bar a_s\}_{s<t}$ and let us consider clopens $[\varphi_i(\bar a_s,\bar y)]$ in $S_{\bar d}(\C)$ for $i<n$, $s<t$. Let $\mathcal F$ be the family of all $\varepsilon\in 2^{n\times t}$ such that $\psi_\varepsilon(\bar y):= \bigwedge_{\substack{i<n\\s<t}}\varphi_i(\bar a_s,\bar y)^{\varepsilon(i,s)}$ is consistent with $\tp(\bar d)$. (Note that the $[\psi_\varepsilon(\bar y)]$'s are the atoms of the Boolean algebra generated by the $[\varphi_i(\bar a_s,\bar y)]$'s in $S_{\bar d}(\C)$.) Set $k=\#\mathcal F$ and $\lambda_\varepsilon= \mu([\psi_\varepsilon(\bar y)])$ for $\varepsilon\in\mathcal F$; clearly $\sum_{\varepsilon\in\mathcal F}\lambda_\varepsilon=1$. Further, for each $\varepsilon\in\mathcal{F}$ choose $\sigma_\varepsilon\in\Aut(\C)$ such that $\sigma_\varepsilon^{-1}(\bar d)\models\psi_\varepsilon(\bar y)$. We claim that the chosen $\lambda_\varepsilon$'s and $\sigma_\varepsilon$'s satisfy our requirements.

First, note the following: for $s<t$ and $\varepsilon\in\mathcal F$, $c(\sigma_{\varepsilon}(\bar a_s))(i)=1$ iff $\models\varphi_i(\sigma_\varepsilon(\bar a_s),\bar d)$ iff $\models\varphi_i(\bar a_s,\sigma_\varepsilon^{-1}(\bar d))$ iff $\varepsilon(i,s)=1$. Hence,
$$\sum_{\varepsilon\in\mathcal F}\lambda_\varepsilon c(\sigma_{\varepsilon}(\bar a_s))(i)= \sum_{\substack{\varepsilon\in\mathcal F\\\varepsilon(i,s)=1}}\mu([\psi_\varepsilon(\bar y)])= \mu(\bigsqcup_{\substack{\varepsilon\in\mathcal F\\\varepsilon(i,s)=1}}[\psi_\varepsilon(\bar y)])= \mu([\varphi_i(\bar a_s,\bar y)]),$$
so for any $s<s'<t$ we have:
$$\sum_{\varepsilon\in\mathcal F}\lambda_\varepsilon c(\sigma_{\varepsilon}(\bar a_s))(i)- \sum_{\varepsilon\in\mathcal F}\lambda_\varepsilon c(\sigma_{\varepsilon}(\bar a_{s'}))(i)= \mu([\varphi_i(\bar a_s,\bar y)])-\mu([\varphi_i(\bar a_{s'},\bar y)])=0$$
by the invariance of $\mu$. Thus, (3) holds.
\end{proof}

As usual, Remark \ref{lemma deecrp absolutness}  together with the fact that in the proof of Theorem \ref{theorem deecrp} we work in the given $\C$ imply that amenability of $T$ is absolute (i.e.\ does not depend on the choice of $\C$). But this was proved directly in \cite{HKP}, which together with the observation that in the proof of Theorem \ref{theorem deecrp} we work in the given $\C$ implies Remark \ref{lemma deecrp absolutness}. By examining the above proofs, one can also see that we can assume here only that $\C$ is $\aleph_0$-saturated and strongly $\aleph_0$-homogeneous.

\section{Around profiniteness of the Ellis group}\label{section around}

In this section, we will prove Theorem \ref{THMsecond} as well as several other criteria for [pro]finiteness of the Ellis group of the theory in question. 

\begin{proof}[Proof of Theorem \ref{THMsecond}]
The implications (D)$\implies$(C)$\implies$(B)$\implies$(A)$\implies$(A')$\implies$(A") follow from various facts or observations made so far: the first implication follows by Fact \ref{fact finite image}, the second one is clear by Lemma \ref{lemma inverse limit of St's}, the third one by Fact \ref{fact inverse limit and ellis}, and the last two by Proposition \ref{lemma galkp profinite}.

It remains to show (B)$\implies$(C).
Denote by $\mathcal F$ the set of all pairs $(\Delta,\bar p)$, where $\Delta$ is a finite set of formulae and $\bar p$ a finite set of types from $S(T)$ with the same variables as the parameter variables of the formulae in $\Delta$. For any $t=(\Delta,\bar p)\in\mathcal F$ put $S_t:= S_{\bar c,\Delta}(\bar p)$.

Let $F:S_{\bar c}(\C)\to\invlim_{t\in\mathcal F}S_t$ be the flow isomorphism given by restrictions, and let $G:S_{\bar c}(\C)\to\invlim_{i\in I}X_i$ be a flow isomorphism. Also, let $f_{t_0}: \invlim_{t\in\mathcal F}S_t\to S_{t_0}$ and $g_{i_0}:\invlim_{i\in I}X_i\to X_{i_0}$ be projections. Consider:
$F_{t}:=\{(p,q)\in S_{\bar c}(\C)^2\mid f_{t}\circ F(p)= f_{t}\circ F(q)\}$ and $G_{i}:=\{(p,q)\in S_{\bar c}(\C)^2\mid g_{i}\circ G(p)= g_{i}\circ G(q)\}$ for $t\in \mathcal F$ and $i\in I$. Both the $F_{t}$'s and the $G_{i}$'s are obviously equivalence relations on $S_{\bar c}(\C)$. They are closed: $F_{t}$ is the inverse image of the diagonal on $S_{t}$ under $f_{t}\circ F$, and similarly for the $G_{i}$'s. Moreover, they are clearly $\Aut(\C)$-invariant and also directed in the sense that $t\leqslant t'\in\mathcal F$ implies $F_{t'}\subseteq F_t$, and $i\leqslant i'\in I$ implies $G_{i'}\subseteq G_i$.

Notice that $\bigcap_{t\in\mathcal F}F_t=D$ and $\bigcap_{i\in I}G_i=D$, where $D=\{(p,p)\mid p\in S_{\bar c}(\C)\}$.

\begin{claim*} For any $t\in\mathcal F$ there is $i\in I$ such that $G_i\subseteq F_t$.
\end{claim*}
\begin{proofclaim}
Observe that for any clopen $C\subseteq S_{\bar c}(\C)$ we can find $i_0\in I$ such that $C$ is a union of $G_{i_0}$-classes. To see this, note that $\bigcap_{i\in I}G_i=D$ implies that $\bigcup_{i\in I}G_i^c\cup (C\times C)\cup (C^c\times C^c)=S_{\bar c}(\C)^2$, where each member is open. By compactness, $\bigcup_{i\in I_0}G_i^c\cup (C\times C)\cup (C^c\times C^c)=S_{\bar c}(\C)^2$ for some finite $I_0$, so $\bigcap_{i\in I_0}G_i\subseteq (C\times C)\cup (C^c\times C^c)$. By choosing $i_0$ to be greater than all elements of $I_0$, we get $G_{i_0}\subseteq (C\times C)\cup (C^c\times C^c)$, and the conclusion follows. 

Let $t=(\Delta,\bar p)$, where $\Delta=\{\varphi_l(\bar x,\bar y)\}_{l<k}$ and $\bar p=\{p_j(\bar y)\}_{j<m}\subseteq S_{\bar y}(T)$. For $j<m$ choose $\bar a_j\models p_j$. Then $[\varphi_l(\bar x,\bar a_j)]$ is a clopen subset of $S_t$ for $l<k$ and $j<m$, so $(f_t\circ F)^{-1}[[\varphi_l(\bar x,\bar a_j)]]$ is clopen in $S_{\bar c}(\C)$. By the the previous paragraph, we can find $i_{l,j}\in I$ such that $(f_t\circ F)^{-1}[[\varphi_l(\bar x,\bar a_j)]]$ is a union of $G_{i_{l,j}}$-classes. Since $G_{i_{l,j}}$ is $\Aut(\C)$-invariant and $f_t\circ F$ is a homomorphism of flows, we get that for every  $\bar a\models p_j$, $(f_t\circ F)^{-1}[[\varphi_l(\bar x,\bar a)]]$ is a union of $G_{i_{l,j}}$-classes. Let $i\in I$ be greater than or equal to all $i_{l,j}$'s, for $l<k$ and $j<m$. It follows that for every clopen $C$ in $S_t$ we have that $(f_t\circ F)^{-1}[C]$ is a union of $G_i$-classes. 

Therefore, for any $q\in S_t$, since
$(f_t\circ F)^{-1}[\{q\}]= \bigcap_C (f_t\circ F)^{-1}[C],$ with the intersection taken over all clopens $C$ containing $q$,
we get that $(f_t\circ F)^{-1}[\{q\}]$ is a union of $G_i$-classes. This means that any $F_t$-class is a union of $G_i$-classes, hence $G_i\subseteq F_t$.
\end{proofclaim}

By the claim, for any $t\in\mathcal F$ we can find $i\in I$ such that the identity map $S_{\bar c}(\C) \to S_{\bar c}(\C)$ induces a flow epimorphism $S_{\bar c}(\C)/G_i\to S_{\bar c}(\C)/F_t$. 
On the other hand, $S_{\bar c}(\C)/G_i \cong X_i$ and $S_{\bar c}(\C)/F_t \cong S_t$ as flows. Therefore, there exists a flow epimorphism $X_i\to S_t$, and we are done by Fact \ref{fact induced epimorphism}(i) and Fact \ref{fact epi ellis}.
\end{proof}

By Theorem \ref{theorem characterization of sep fin dpeerdeg}, for a theory with sep.\ finite \edeerdeg\ Condition (D) holds. Hence, since (D)$\implies$(A), we obtain:

\begin{corollary}\label{corollary sep fin dpeerdeg implies 0-dim ellis}
If $T$ has sep.\ fin.\ \edeerdeg, then the Ellis group of $T$ is profinite.\qed
\end{corollary}

As mentioned in the introduction, we will give examples showing that (A'') does not imply (A') (see Example \ref{example equivalences n classes}) and (A') does not imply (B) (see Example \ref{example equivalences with two classes}). We do not know however whether Example \ref{example equivalences with two classes} satisfies (A), so we do not know whether it shows that (A') does not imply (A) or that (A) does not imply (B). We have also not found an example showing that (C) does not imply (D). 

We do not expect that (A)$\implies$(C) is true. However, we can easily see that (A) is equivalent to a weaker version of (C):
\begin{enumerate}
\item[(C')] for every finite set of formulae $\Delta$ and types $\bar p\subseteq S(T)$, the Ellis group of the flow $(\Aut(\C),S_{\bar c,\Delta}(\bar p))$ is profinite.
\end{enumerate}

By Lemma \ref{lemma inverse limit of St's} and Fact \ref{fact inverse limit and ellis}, we see that (C')$\implies$(B')$\implies$(A), where (B') is the weaker version of (B) in which we require only profinitenes of the Ellis groups of the flows $X_i$.

\begin{proposition}\label{lemma A equiv C'} Conditions (A), (B'), and (C') are equivalent. 
\end{proposition}
\begin{proof}
It remains to prove (A)$\implies$(C'). Consider an $\Aut(\C)$-flow epimorphism $S_{\bar c}(\C)\to S_{\bar c,\Delta}(\bar p)$ given by the restriction. By Fact \ref{fact induced epimorphism}, it induces an $\Aut(\C)$-flow and semigroup epimorphism $f:\EL(S_{\bar c}(\C))\to \EL(S_{\bar c,\Delta}(\bar p))$. Let $\M$ be a minimal left ideal of $\EL(S_{\bar c}(\C))$ and $u\in\J(\M)$. By Fact \ref{fact epi ellis}, $\M':= f[\M]$ is a minimal left ideal of $\EL(S_{\bar c,\Delta}(\bar p))$, $u':=f(u)\in\J(\M')$, and $f_{\upharpoonright u\M}:u\M\to u'\M'$ is a group epimorphism and quotient map (in the $\tau$-topology). By Remark \ref{remark uM and um/H zero dimensionality}(b), it is enough to prove that $u'\M'$ is $0$-dimensional. For this, it suffices to prove that $f_{\upharpoonright u\M}[U]$ is clopen for any clopen $U\subseteq u\M$ (as $u\M$ is $0$-dimensional), which boils down to showing that $f_{\upharpoonright u\M}^{-1}[f_{\upharpoonright u\M}[U]]$ is clopen. Since $f_{\upharpoonright u\M}^{-1}[f_{\upharpoonright u\M}[U]]=\ker(f_{\upharpoonright u\M})\cdot U$, this set is clopen by the fact that $u'\M'$ is $T_1$, compactness of $u\M$, and continuity of multiplication in $u\M$.
\end{proof}

We now state two criteria which in some quite common situations guarantee profiniteness (or even finiteness) of the Ellis group.

\begin{proposition}\label{proposition A} Suppose that $L\subseteq L^*$ are finite relational languages, $T^*$ is a complete $L^*$-theory with q.e.\ such that $T:= T^*_{\upharpoonright L}$ also has q.e. Let $\C^*$ be a monster model of $T^*$ such that $\C:=\C^*_{\upharpoonright L}$ is a monster model of $T$. If there is $\eta^*\in \EL(S_{\bar c}(\C^*))$ such that $\Im(\eta^*)\subseteq\Inv_{\bar c}(\C^*)$, then for any finite $\bar z$ there exists $\eta\in\EL(S_{\bar z}(\C))$ such that $\Im(\eta)$ is finite.
In particular, the Ellis group of the flow $(\Aut(\C),S_{\bar z}(\C))$ is finite. 
\end{proposition}

We should stress that here we consider two different flows: $(\Aut(\C^*),S_{\bar c}(\C^*))$ and $(\Aut(\C),S_{\bar c}(\C))$, defined with respect to two different languages. In order to avoid confusions, we denote by asterisk notions defined with respect to the language $L^*$, and without asterisk notions defined with respect to the language $L$.

\begin{proof} 
By Lemma \ref{lemma ellis semi epi bigger lang invariant}(iv), there is $\eta \in \EL(S_{\bar z}(\C))$ with $\Im(\eta) = \Inv_{\bar z}^*(\C)$, where $\Inv_{\bar z}^*(\C)\subseteq S_{\bar z}(\C)$ is the set of all $\Aut(\C^*)$-invariant types.  Hence,
since $T$ has q.e., each type $p\in \Im(\eta)$ is determined by saying whether $R(\bar z,\bar b)\in p$ or not for some (every) $\bar b\models q$, for all $R(\bar z,\bar y)\in L$ and $q\in S_{\bar y}(T^*)$. On the other hand, $S_{\bar y}(T^*)$ is finite by q.e.\ of $T^*$ and finiteness of $L^*$. Since $L$ is also finite, we see that we have only finitely many possibilities for $p\in\Im(\eta)$, and thus $\Im(\eta)$ is finite. (Implicitly we also used here that the languages are relational.)

The ``In particular" part follows by Fact \ref{fact finite image}.
\end{proof}

\begin{corollary}\label{corollary A} Under the assumptions of Proposition \ref{proposition A}, the Ellis group of the flow $(\Aut(\C),S_{\bar c}(\C))$ is finite.
\end{corollary}
\begin{proof}
Since $L$ is a finite relational language and $T$ has q.e., $T$ is $m$-ary, where $m$ is the maximal arity of the symbols in $L$. By Corollary \ref{corollary c to finite z m-ary}, the Ellis groups of the flows $(\Aut(\C),S_{\bar c}(\C))$ and $(\Aut(\C),S_{m-1}(\C))$ are isomorphic. The latter is finite by Proposition \ref{proposition A}.
\end{proof}


\begin{proposition}\label{proposition B}
Suppose that $L\subseteq L^*$, $T^*$ is a complete $L^*$-theory, and $T:=T^*_{\upharpoonright L}$. Let $\C^*$ be a monster of $T^*$ such that $\C:=\C^*_{\upharpoonright L}$ is a monster of $T$. Assume that for any finitely many variables $\bar y$ there are only finitely many extensions in $S_{\bar y}(T^*)$ of each type from $S_{\bar y}(T)$. If there is $\eta^*\in\EL(S_{\bar c}(\C^*))$ such that $\Im(\eta^*)\subseteq\Inv_{\bar c}(\C^*)$, then (D) holds. 
In particular, the Ellis group of $T$ is profinite.
\end{proposition}
\begin{proof}
The existence of $\eta^*\in\EL(S_{\bar c}(\C^*))$ with $\Im(\eta^*)\subseteq\Inv_{\bar c}(\C^*)$ implies, by Lemma \ref{lemma ellis semi epi bigger lang invariant}(ii), that there is $\eta \in\EL(S_{\bar c}(\C))$ with $\Im(\eta)\subseteq\Inv_{\bar c}^*(\C)$, where $\Inv_{\bar c}^*(\C)\subseteq S_{\bar c}(\C)$ is the set of all $\Aut(\C^*)$-invariant types. Then, for each finite $\Delta, \bar p$, the induced $\eta_{\Delta,\bar p} \in \EL(S_{\bar c,\Delta}(\bar p))$ satisfies $\Im(\eta_{\Delta, \bar p})\subseteq \Inv^*_{\bar c,\Delta}(\bar p)$, where $\Inv^*_{\bar c,\Delta}(\bar p)$ is the set of all $\Aut(\C^*)$-invariant types from $S_{\bar c,\Delta}(\bar p)$.

Consider any finite $\Delta$ and $\bar p$. For each  $q_{\Delta,\bar p}\in \Inv^*_{\bar c,\Delta}(\bar p)$ and $\varphi(\bar x,\bar y)\in\Delta$ we have: $\varphi(\bar x,\bar b)\in q_{\Delta,\bar p}(\bar x)$, where $\tp(\bar b)\in \bar p$, iff $\varphi(\bar x,\bar b')\in q_{\Delta,\bar p}(\bar x)$ for all $\bar b'$ realizing the same extension in $S_{\bar y}(T^*)$ of $\tp(\bar b) \in S_{\bar y}(T)$ as $\bar b$. Since $\Delta$ and $\bar p$ are finite and each $p_0\in\bar p$ has only finitely many extensions in $S_{\bar y}(T^*)$, $q_{\Delta,\bar p}(\bar x)$ is completely determined by finitely many conditions. Therefore, $\Inv^*_{\bar c,\Delta}(\bar p)$ is finite, and so $\Im(\eta_{\Delta,\bar p})$ is finite.

The ``In particular'' part follows from (D),
as (D)$\implies$(A).
\end{proof}

\begin{remark}\label{remark prop B}
The assumption of Proposition \ref{proposition B} that for each finite $\bar y$ there are only finitely many extensions in $S_{\bar y}(T^*)$ of each type from $S_{\bar y}(T)$ holds when $L^*$ is relational, $L^*\smallsetminus L$ is finite, and $T^*$ has q.e.\qed
\end{remark}

\begin{corollary}\label{added corollary of proposition B}
The assumptions of Proposition \ref{proposition A} imply the assumptions of Proposition \ref{proposition B} which in turn imply that $T$ has sep.\ fin.\ \edeerdeg.
\end{corollary}

\begin{proof}
The first part follows from Remark \ref{remark prop B}. To see the rest, note that since Proposition \ref{proposition B} yields (D), 
we can use Theorem \ref{theorem characterization of sep fin dpeerdeg}($\Leftarrow$).
\end{proof}

Thus, the assumptions of Proposition \ref{proposition B} yield a criterion for having sep.\ fin.\ \edeerdeg. In fact, this is a counterpart of the criterion proved by Zucker in \cite[Theorem 8.14]{Z}, which we now explain (but for the definitions the reader is referred to \cite{Z} and \cite{The2}).
Zucker shows that a Fra\"{i}ss\'{e} structure has sep.\ fin.\ Ramsey degree iff it has a Fra\"{i}ss\'{e}, precompact (relational) expansion whose age has the embedding Ramsey property and the expansion property relative to the age of the original structure. 
Combining this with \cite[Theorem 6]{The2} (and noting that the Fra\"{i}ss\'{e} subclass of $\Age(\pmb{F}^*)$ obtained there is a reasonable expansion of $\Age(\pmb{F})$), we get that  a Fra\"{i}ss\'{e} structure has sep.\ fin.\ Ramsey degree iff it has a Fra\"{i}ss\'{e}, precompact (relational) expansion whose age has the embedding Ramsey property. 
Now, if both the original and the expanded structure are $\aleph_0$-saturated (so have q.e.), then being precompact exactly means that for any finitely many variables $\bar y$ there are only finitely many extensions in $S_{\bar y}(T^*)$ of each type from $S_{\bar y}(T)$. Moreover, by Theorem \ref{theorem characterization of dpeerp}, the assumption of Proposition \ref{proposition B} saying that there is $\eta^*\in\EL(S_{\bar c}(\C^*))$ such that $\Im(\eta^*)\subseteq\Inv_{\bar c}(\C^*)$ is equivalent to $T^*$ having \edeerp\ (which is the appropriate counterpart of the embedding Ramsey property).

However, Example \ref{example criterion b in not necessary} shows that, in contrast with Zucker's result, in our case the assumptions of Proposition \ref{proposition B} are only sufficient to have sep. finite \edeerdeg\ for $T$, but they are not necessary. It could be interesting to find an ``iff'' criterion of this kind; we have not tried to do that.

\section{Applications and examples}\label{section applications}

In this section, we use our analysis to prove [pro]finiteness (or triviality) of the Ellis groups in some classes of theories, as well as give examples illustrating various phenomena and showing the lack of implications between some key conditions considered in this paper.

\begin{example}\label{example stable} We prove that stable theories have sep.\ fin.\ \edeerdeg. Denote by $\NF_{\bar c}(\C)$ the space of all types from $S_{\bar c}(\C)$ which do not fork over $\emptyset$. 
Using only one consequence of stability, namely that forking equals dividing over $\emptyset$, the proof of \cite[Proposition 7.11]{KNS} yields:
\begin{claim*} There exists an idempotent $u$ in a minimal left ideal $\M$ of $\EL(S_{\bar c}(\C))$ with $\Im(u)\subseteq\NF_{\bar c}(\C)$.
\end{claim*}
Choose $u\in\J(\M)$ given by the claim. Fix a finite set of $L$-formulae $\Delta$ and finite tuple of types $\bar p$. Let $u_{\Delta,\bar p} := \hat{f}(u) \in \EL(S_{\bar c, \Delta}(\bar p))$, where $\hat{f}: \EL(S_{\bar c}(\C))\to \EL(S_{\bar c,\Delta}(\bar p))$ is induced by the restriction map $f: S_{\bar c}(\C) \to S_{\bar c, \Delta}(\bar p)$ via Fact \ref{fact induced epimorphism}(i). 
Denote by $\NF_{\bar c,\Delta}(\bar p)$ the restrictions of all types from $\NF_{\bar c}(\C)$ to $S_{\bar c,\Delta}(\bar p)$. By Corollary \ref{corollary of fact 2.4}, we conclude that $\Im(u_{\Delta,\bar p})\subseteq\NF_{\bar c,\Delta}(\bar p)$.
On the other hand, by 
finiteness of the $\Delta$-multiplicity of $\tp(\bar c)$, $\NF_{\bar c,\Delta}(\bar p)$ is finite. Hence, $\Im(u_{\Delta,\bar p})$ is finite.
Thus, Condition (D) 
is fulfilled, and hence $T$ has sep.\ fin.\ \edeerdeg\ by Theorem \ref{theorem characterization of sep fin dpeerdeg}.

If we in addition assume that $\tp(\bar c)$ is stationary (equivalently, all types in $S(\emptyset)$ are stationary), then  $T$ has \edeerp. Indeed, under this assumption, $\NF_{\bar c}(\C)$ is a singleton consisting of an invariant type, so $u$ given by the claim has image contained in the invariant types, and hence $T$ has \edeerp\ by Theorem \ref{theorem characterization of dpeerp}. Conversely, if $\tp(\bar c)$ is not stationary, then there is no global invariant extension of this type, i.e.\ $T$ is not extremely amenable, so $T$ does not have \deerp. 

This implies that in stable theories the properties \edeerp\ and \deerp\ are equivalent (which is also immediate from definitions, using definability of types).
%
\end{example}

The next concrete example shows that a stable theory does not need to have sep.\ fin.\ \eerdeg; in  particular, sep.\ fin.\ \edeerdeg\ does not imply sep.\ fin.\ \eerdeg.

\begin{example}\label{example ACF0}
Take $T:=\mathrm{ACF_0}$. Consider on $\C^\times$ the relations $\sim_n$ given by $a\sim_n a'$ iff $a'=\xi_n^ka$ for some $k$, where $\xi_n$ denotes an $n$-th primitive root of unity. The $\sim_n$'s are equivalence relations on $\C^\times$, and let $X_n$ be a set of representatives of the $\sim_n$-classes; then $\C^\times=\bigsqcup_{k<n}\xi_n^k X_n$. For any transcendental $a$, take $\bar b=(a,\xi_na,\dots,\xi_n^{n-1}a)$ and consider the coloring $c:{\C\choose a}\to n$ given by: $c(a')=k$ iff $a'\in\xi_n^kX_n$. For any $\bar b'\in{\C\choose\bar b}$ we have that $\bar b'=(a',\omega_na',\dots,\omega_n^{n-1}a')$ for some transcendental $a'$ and primitive $n$-th root of unity $\omega_n$. Note that $\#c[{\bar b'\choose a}]=n$. Thus, a transcendental element cannot have finite \eerdeg, so $T$ does not have sep.\ fin.\ \eerdeg, while $T$ is stable and as such has sep.\ fin.\ \edeerdeg.

If we modify $T$ by naming all constants from the algebraic closure of $\mathbb{Q}$, then all types in $S(\emptyset)$ are stationary, so the resulting theory has \edeerp, but the above argument shows that it does not have sep.\ fin.\ \eerdeg.
\end{example}

The next example shows that the criterion for having sep.\ fin.\ \edeerdeg\ given in Proposition \ref{proposition B} (see also Corollary \ref{added corollary of proposition B}) is not a necessary condition.

\begin{example}\label{example criterion b in not necessary}
Let $T$ be a stable theory with a finitary type $p \in S_{\bar y}(T)$ of infinite multiplicity (i.e.\ with infinitely many global non-forking extensions). (For example, one can take $T:=(\mathbb{Z},+)$ and $p(y):=\tp(1)$.) Then it has sep.\ fin.\ \edeerdeg, but we will show that there is no expansion $T^*$ of $T$ satisfying the assumptions of Proposition \ref{proposition B}. Suppose for a contradiction that $T^*$ is such an expansion. The assumption that $\Im(\eta^*)\subseteq\Inv_{\bar c}(\C^*)$ implies that $T^*$ is extremely amenable, so $\acl^{eq,*}(\emptyset) = \dcl^{eq,*}(\emptyset)$ (both computed in $\C^*$). Since $p$ has infinite multiplicity, it has infinitely many extensions to complete types over $\acl^{eq}(\emptyset)$ computed in $\C$, so also over $\acl^{eq,*}(\emptyset) \supseteq \acl^{eq}(\emptyset)$. Therefore, we conclude that $p$ has infinitely many extensions in $S_{\bar y}(T^*)$, a contradiction.
\end{example}

In the following example, we list some Fra\"iss\'e classes which are known to have the embedding Ramsey property and whose Fra\"iss\'e limits are $\aleph_0$-categorical, and hence $\aleph_0$-saturated. By Remark \ref{remark fraisse eerp}, the theories of these Fra\"iss\'e limits have \eerp, so \edeerp\ as well, and thus their Ellis groups are trivial by Corollary \ref{corollary dpeerp implies trivial ellis}.

By a hypergraph we mean a structure in a finite relational language such that each basic relation $R$ is irreflexive (i.e.\ $(R(a_0,\dots,a_{n-1})$ implies that the $a_i$'s are pairwise distinct) and symmetric (i.e.\ invariant under permutations of coordinates).

\begin{example}\phantomsection\label{example rp}
\begin{enumerate}[label=(\alph*), align=right, leftmargin=*]
\item The class of all finite (linearly) ordered hypergraphs omitting a fixed class of finite irreducible hypergraphs in a finite relational language containing $\leqslant$. See \cite[Theorem A]{NR}. In particular, the theories of the ordered random graph and the ordered $K_n$-free random graph have \eerp. 

\item The class of all finite sets with $n$ linear orderings (for a fixed $n$) in the language $L=\{\leqslant_1,\dots,\leqslant_n\}$. See \cite[Theorem 4]{S}.

\item The class of all finite structures in the language $L=\{\sqsubseteq,\leqslant\}$ for which $\sqsubseteq$ is a partial ordering and $\leqslant$ is a linear ordering extending $\sqsubseteq$. See \cite{S0}.

\item The class of all finite structures in the language $L=\{\sqsubseteq,\leqslant,\preccurlyeq\}$ for which $\sqsubseteq$ is a partial ordering, $\leqslant$ is a linear ordering, and $\preccurlyeq$ is a linear ordering extending $\sqsubseteq$. See \cite[Theorem 1]{S}.

\item The class of all naturally ordered finite vector spaces over a fixed finite field $F$ in the language $L=\{+,m_a,\leqslant\}_{a\in F}$, where $+$ is the addition, $m_a$ is the unary multiplication by $a\in F$, and $\leqslant$ is the anti-lexicographical linear ordering induced by an ordering of a basis. See  \cite{KPT} together with \cite[Corollary 2]{GLR}.

\item The class of all naturally ordered finite Boolean algebras in the language of Boolean algebras expanded by $\leqslant$, where $\leqslant$ is the anti-lexicographical linear order induced by an ordering of atoms. See \cite[Proposition 6.13]{KPT} and \cite{GR}.

\item The class of all finite linearly ordered structures in the infinite language consisting of relational symbols $R_n$, $n>0$, where $R_n$ is $n$-ary, such that each $R_n$ is irreflexive and symmetric. It is easy to see that this is a Fra\"{i}ss\'{e} class whose limit is $\aleph_0$-categorical. The fact that it has the embedding Ramsey property follows from the fact that the restrictions to the finite sublanguages have it by (a). This holds more generally in a situation when for each $n$ there are only finitely many relational symbols of arity $n$.



\end{enumerate}
\end{example}


In the examples of Fra\"{i}ss\'{e} structures with the embedding Ramsey property which are not $\aleph_0$-saturated, the situation is not so obvious, as we cannot use Remark  \ref{remark fraisse eerp} to deduce \eerp\ and so \edeerp. 
We leave as a problem to find an example of a Fra\"{i}ss\'{e} structure with the embedding Ramsey property whose theory does not have \eerp\ or even \edeerp, or show that there is none.

We now list some  Fra\"iss\'e classes which are known to have sep.\ fin.\ embedding Ramsey degree and whose Fra\"iss\'e limits are $\aleph_0$-categorical, and hence $\aleph_0$-saturated. By Remark \ref{remark fraisse fin sep eerdeg}, the theories of these Fra\"iss\'e limits have sep.\ fin.\ \eerdeg, so sep.\ fin.\ \edeerdeg, hence their Ellis groups are profinite by Corollary \ref{corollary sep fin dpeerdeg implies 0-dim ellis}.

As was recalled after Corollary \ref{added corollary of proposition B}, by \cite{Z} and \cite{The2}, to see that a Fra\"iss\'e class $\mathcal{K}$ has sep.\ fin.\ embedding Ramsey degree, one needs to find a Fra\"{i}ss\'{e} class which is a reasonable, precompact, relational expansion of $\mathcal{K}$ with the embedding Ramsey property.


\begin{example}\phantomsection\label{exampre rdeg}
\begin{enumerate}[label=(\alph*), align=right, leftmargin=*]
\item The class of all finite hypergraphs omitting a fixed class of finite irreducible hypergraphs in a finite relational language. Example \ref{example rp}(a) gives the desired expansion of this class.

\item The class of all finite vector spaces over a fixed finite field. Example \ref{example rp}(e) gives the desired expansion of this class.

\item The class of all finite Boolean algebras in the language of all Boolean algebras. Example \ref{example rp}(f) gives the desired expansion of this class.



\item The age of any homogeneous directed graph. See \cite{JLTW}.

\item The class of all finite structures in the infinite language consisting of relational symbols $R_n$, $n>0$, where $R_n$ is $n$-ary, such that each $R_n$ is irreflexive and symmetric. Example \ref{example rp}(g) gives the desired expansion of this class. This holds more generally in a situation when for each $n$ there are only finitely many relational symbols of arity $n$.

\end{enumerate}
\end{example}

Sometimes we can say more. For example, in Example \ref{exampre rdeg}(b) the Fra\"{i}ss\'{e} limit is just an infinite dimensional vector space over a finite field, so it is strongly minimal with the property that $\NF_{\bar c}(\C) = \Inv_{\bar c}(\C)$, and hence the Ellis group is trivial by the claim in Example \ref{example stable} and the argument from the proof of Corollary \ref{corollary dpeerp implies trivial ellis}.
The situation in Example \ref{exampre rdeg}(a) is more interesting.
Namely, the Ellis group there is finite. This follows by Corollary \ref{corollary A}. Indeed, let $L$ be a finite relational language and $L^*=L\cup\{\leqslant\}$. Let $K^*$ be the Fra\"iss\'e limit of the class of all linearly ordered finite $L$-hypergraphs (possibly omitting a fixed class of finite irreducible hypergraphs). Then $K:=K^*_{\upharpoonright L}$ is the Fra\"iss\'e limit of all finite $L$-hypergraphs (omitting this fixed class of finite irreducible hypergraphs). We may find a monster model $\C^*$ of $T^*:=\mathrm{Th}(K^*)$ such that $\C:=\C^*_{\upharpoonright L}$ is a monster model of $T:= \mathrm{Th}(K)= T^*_{\upharpoonright L}$. By Example \ref{example rp}(a), $T^*$ has \edeerp, so we can find $\eta^*\in\EL(S_{\bar c}(\C^*))$ such that $\Im(\eta^*)\subseteq \Inv_{\bar c}(\C^*)$ by Theorem \ref{theorem characterization of dpeerp}. Moreover, both theories $T$ and $T^*$ have q.e. Thus, by Corollary \ref{corollary A}, the Ellis group of $T$ is finite. The same holds for the homogeneous directed graphs from Example \ref{exampre rdeg}(d).

\begin{example}\label{example random graph}
If $T$ is the theory of the random $n$-hypergraph (so $L$ in Example \ref{exampre rdeg}(a) consists of only one $n$-ary relational symbol $R$), then one can say even more: $T$ has \edeerp, so the Ellis group is trivial. To see this, consider $L^*:=L\cup\{\leqslant\}$. Let $T$ be the theory of the random $n$-hypergraph (the theory of the Fra\"iss\'e limit of the class of all finite $L$-hypergraphs), and $T^*$ the theory of the ordered random $n$-hypergraph (the theory of the Fra\"iss\'e limit of the class of all ordered finite $n$-hypergraphs); then $T=T^*_{\upharpoonright L}$. Choose a monster model $\C^*\models T^*$ such that $\C:=\C^*_{\upharpoonright L}$ is a monster model of $T$. 

By Example \ref{example rp}(a), $T^*$ has \edeerp, so, by Theorem \ref{theorem characterization of dpeerp}, we can find $\eta_0^*\in\EL(S_{\bar c}(\C^*))$ such that $\Im(\eta_0^*)\subseteq\Inv_{\bar c}(\C^*)$. By Lemma \ref{lemma ellis semi epi bigger lang invariant}(ii), there is $\eta_0\in\EL(S_{\bar c}(\C))$ such that $\Im(\eta_0)\subseteq\Inv_{\bar c}^*(\C)$, where $\Inv_{\bar c}^*(\C)$ is the set of all $\Aut(\C^*)$-invariant types in $S_{\bar c}(\C)$. 
We finish by Theorem 3.14 after noting that each $\Aut(\C^*)$-invariant type from $S_{\bar c}(\C)$ is also $\Aut(\C)$-invariant. To see this observe that, as $T^*$ has q.e.\ and $R$ is $n$-ary and irreflexive, for any $\bar a$ and $\bar b$ with $|\bar a| = |\bar b| <n$, $\bar a\equiv^{L^*}\bar b'$ for some permutation $\bar b'$ of $\bar b$. So, by symmetry of $R$, for any $\Aut(\C^*)$-invariant type $p(\bar x) \in S_{\bar c}(\C)$, $\bar z \subseteq \bar x$, and $\bar a \subseteq \C$ with $|\bar z|+ |\bar a|=n$ and $|\bar a|<n$, we have $R(\bar z,\bar a)\in p(\bar x)$ iff $R(\bar z,\bar b)\in p(\bar x)$. So, $p(\bar x)$ is $\Aut(\C)$-invariant by q.e.\ of $T$.
\end{example}

A similar argument to the one in the previous example is viable if $L$ contains finitely many relational symbols each of which is of one of the two consecutive arities $n$ and $n+1$. We leave to the reader to check the details. However, the above conclusion does not hold for all random hypergraphs in finite languages. In the following example, we state that the theory of the random $(2,4)$-hypergraph has a non-trivial Ellis group and describe various consequences of this result. All computations around this example are contained in Appendix \ref{section example}.

\begin{example}\label{example R2 R4} 
Let $L=\{R_2,R_4\}$, where $R_2$ is a binary and $R_4$ is a quaternary relation. The Fra\"iss\'e limit $K$ of the class of all finite $L$-hypergraphs is the random $(2,4)$-hypergraph. We will prove that the Ellis group of the theory $T:=\mathrm{Th}(K)$ is $\mathbb Z/2\mathbb Z$. Hence, $T$ does not have \edeerp\ by Corollary \ref{corollary dpeerp implies trivial ellis}, although it does have sep.\ fin.\ \eerdeg\ by Example \ref{exampre rdeg}(a).  Furthermore, this theory does have \deerp\ by Theorem \ref{theorem characterization of deerp}, as the class above has free amalgamation and so $T$ is extremely amenable (see the discussion after Corollary 2.16 in \cite{HKP}). So the theory $T$ shows that in general sep.\ fin.\ \edeerdeg\ (or even \eerdeg) does not imply \edeerp, and that \deerp\ does not imply \edeerp.

Since $T$ is extremely amenable, $\Gal_L(T)=\Gal_{KP}(T)$ is trivial by \cite[Proposition 4.2]{HKP}, whereas the Ellis group of $T$ is non-trivial. 
On the other hand, \cite[Theorem 0.7]{KNS} says that a certain natural epimorphism from the Ellis group of an amenable theory with NIP to its KP-Galois group is always an isomorphism (this is a variant of Newelski's conjecture for groups of automorphisms). Thus, our theory $T$ shows that in this theorem the assumption that the theory has NIP cannot be dropped (even assuming extreme amenability of the theory in question). Such an example was not known so far.
\end{example}

Now, we will describe a variant of the above example which yields a Fra\"{i}ss\'{e} structure whose theory $T$ satisfies the assumptions of Remark \ref{remark prop B} and Proposition \ref{proposition B}, and so has sep.\ fin.\ \edeerdeg, and the Ellis group of $T$ is infinite (and clearly profinite). Also, $T$ is extremely amenable, so $\Gal_{KP}(T)$ is trivial. It will be a many-sorted example. But this is not a problem, because, as we said before, the whole theory developed in this paper works in a many-sorted context after minor adjustments. In the case of Proposition  \ref{proposition B}, it is easy to see that the proof goes through (only in Corollary \ref{corollary c to infinite z}, which is involved in this proof, one has to consider the tuple $\bar z$ in which there are infinitely many variables associated with each sort).

\begin{example}\label{example R2 R4 Ps}
Consider the language $L$ consisting of infinitely many sorts $S_n$, $n<\omega$, and relational symbols $R_2^n$ and $R_4^n$, $n<\omega$, where each $R_2^n$ is binary, each $R_4^n$ quaternary, and they are both associated with the sort $S_n$. Let $L^*:=L\cup\{\leqslant_n: n<\omega\}$, where each $\leqslant_n$ is a binary symbol associated with $S_n$. Let $K$ be the $L^*$-structure which is the disjoint union of copies of the ordered random $(2,4)$-hypergraph; then $K:=K^*_{\upharpoonright L}$ is the disjoint union of copies of the random $(2,4)$-hypergraph. Put $T:= \Th(K)$ and $T^*:=\Th(K^*)$. Choose a monster model $\C^*$ of $T^*$ such that $\C:=\C^*_{\upharpoonright L}$ is a monster model of $T$.

Note that the assumptions of the obvious many-sorted version of Remark \ref{remark prop B} are satisfied, so we check the remaining assumption of Proposition  \ref{proposition B} that there is  $\eta^*\in\EL(S_{\bar c}(\C^*))$ such that $\Im(\eta^*)\subseteq \Inv_{\bar c}(\C^*)$. Let $\C_n^*:=S_n(\C^*)$ be the $\{R_2^n,R_4^n,\leqslant_n\}$-structure induced from $\C^*$, for $n<\omega$. Then $\Aut(\C^*)= \prod_{n<\omega}\Aut(\C_n^*)$ and $S_{\bar c}(\C^*)= \prod_{n<\omega}S_{\bar c_n}(\C_n^*)$ after the 
clear identifications, where $\bar c_n\subseteq\bar c$ is an enumeration of $\C_n^*$. By \cite[Lemma 6.44]{RzPhD}, $\EL(S_{\bar c}(\C^*))$ is naturally isomorphic (as a semigroup and as an $\Aut(\C^*)$-flow) with $\prod_{n<\omega} \EL(S_{\bar c_n}(\C_n^*))$. But in each $\EL(S_{\bar c_n}(\C_n^*))$ we have $\eta_n^*$ whose image is contained in the $\Aut(\C_n^*)$-invariant types (by Example \ref{example rp}(a) and Theorem \ref{theorem characterization of dpeerp}), so the corresponding  $\eta^*\in \EL(S_{\bar c}(\C^*))$ has image contained in $\Inv_{\bar c}(\C^*)$.

Now, we will compute the Ellis group of $T$. Let $\C_n:=S_n(\C)$ be the $\{R_2^n,R_4^n\}$-structure induced from $\C$, for $n<\omega$. Then $\Aut(\C)= \prod_{n<\omega}\Aut(\C_n)$ and $S_{\bar c}(\C)= \prod_{n<\omega}S_{\bar c_n}(\C_n)$ after the clear identifications (where $\bar c_n\subseteq\bar c$ is an enumeration of $\C_n$). By \cite[Lemma 6.44]{RzPhD}, $u\M \cong \prod_{n<\omega}u_n\M_n$ (also in the $\tau$-topologies), where $\M$ and $\M_n$ are minimal left ideals of $\EL(S_{\bar c}(\C))$ and $\EL(S_{\bar c_n}(\C_n))$, respectively, and $u \in \J(\M)$, $u_n \in \J(\M_n)$. Since each $u_n\M_n\cong\mathbb Z/2\mathbb Z$ by Example \ref{example R2 R4}, we conclude that  $u\M \cong (\mathbb Z/2\mathbb Z)^\omega$ as a topological group; in particular, $u\M$ is infinite.

Extreme amenability of $T$ can be seen directly: each finitary type $p(\bar x)\in S(\emptyset)$ has a global invariant extension determined by the formulae: $y \ne a$, $\neg R_2^n(y,b)$, and $\neg R_4^n(\bar y',\bar b')$, for all $n < \omega$, and $a, b, \bar b' \subseteq \C$ (with $|\bar b'| \geqslant 1$), $y \in \bar x$, $\bar y' \subseteq \bar x$ (with $|\bar y'|\geq 1$) from the sorts for which these formulae make sense.
\end{example}

We now give an example of a theory with sep.\ fin.\ \edeerdeg\ (even \edeerp) such that for some finite set of formulae $\Delta$ there is no $\eta\in \EL(S_{\bar c,\Delta}(\C))$ with finite image. This shows that we really need to consider spaces $S_{\bar c,\Delta}(\bar p)$, rather than the classical spaces of $\Delta$-types $S_{\bar c,\Delta}(\C)$.

\begin{example}\label{example R_2 Ps}
Let $L=\{R_2,P_n\}_{n<\omega}$, where the $P_n$'s are unary and $R_2$ is binary. Put $P_\omega:=\bigwedge_{n<\omega}\lnot P_n$. Consider the class of all finite $L$-structures $A$ such that:
\begin{itemize}
\item $P_n^A$ are mutually disjoint for $n<\omega$, and
\item $R_2^A$ is irreflexive and symmetric.
\end{itemize}
This class is Fra\"iss\'e. Its Fra\"iss\'e limit $K$ consists of infinitely many disjoint parts $P_n(K)$, for $n\leqslant\omega$, each of which is isomorphic to the random graph, but also there is the random interaction between them. Also, $T:=\mathrm{Th}(K)$ has q.e., hence it is binary, $K$ is $\aleph_0$-saturated, although it is not $\aleph_0$-categorical. 

To see that $T$ has \edeerp, by Proposition \ref{proposition C}, we shall prove that for each singleton $a$, finite $\bar b$ containing $a$, $n<\omega$, and a coloring $c:{\C\choose a}\to 2^n$ there is $\bar b'\in {\C\choose\bar b}$ such that $\bar b'\choose a$ is monochromatic with respect to $c$. For each $n\leqslant\omega$, $P_n(\C)$ is the set of realizations of a complete $1$-type over $\emptyset$. Fix $a\in\C$; then $a\in P_n(\C)$ for some $n\leqslant\omega$. Take a finite $\bar b \ni a$ and without loss assume $\bar b\subseteq P_n(\C)$. Take $r<\omega$ and $c:{\C\choose a}\to r$ (so $c:P_n(\C)\to r$). By saturation, there is a copy $G\subseteq P_n(\C)$ of the random graph. The restriction of $c$ to $G$ corresponds to a finite partition of $G$, so there is a monochromatic isomorphic copy $G'$ of $G$ (see \cite[Proposition 3.3]{Cam}). We finish by taking a copy $\bar b'\subseteq G'$ of $\bar b$, as by q.e., $\bar b'\in{\C\choose\bar b}$.

Consider now $\Delta:=\{R\}$. We check that $\EL(S_{\bar c,\Delta}(\C))$ does not contain an element with finite image. Let $\bar x$ correspond to $\bar c$ and let $x_0\in\bar x$ be any fixed single variable. For every $S\subseteq\omega+1$ denote by $\pi_S(x_0)$ the set:
$$\{R(x_0,a)\mid a\in P_n(\C)\mbox{ for some }n\in S\}\cup \{\lnot R(x_0,a)\mid a\in P_n(\C)\mbox{ for some }n\notin S\}.$$
By randomness, $\pi_S(x_0)$ is consistent with $\tp(\bar c)$, so we can find $p_S(\bar x)\in S_{\bar c,\Delta}(\C)$ extending it. Each $\pi_S(x_0)$ is $\Aut(\C)$-invariant, as the $P_n(\C)$'s are sets of realizations of complete $1$-types over $\emptyset$. Thus, for each $\eta\in\EL(S_{\bar c,\Delta}(\C))$ we have $\pi_S(x_0)\subseteq\eta(p_S)$, so $\eta(p_S)$ are pairwise distinct for $S\subseteq\omega+1$. Therefore, $\Im(\eta)$ is infinite.
\end{example}

Next, we give an example of a theory $T$ showing that (A') does not imply (B) in Theorem \ref{THMsecond}. Hence, (D) fails for $T$, so $T$ does not have sep.\ fin.\ \edeerdeg\ by Theorem \ref{theorem characterization of sep fin dpeerdeg}. Also, $T$ is amenable and so has \deecrp\ by Theorem \ref{theorem deecrp}.

\begin{example}\label{example equivalences with two classes}
Consider the relational language $L=\{R,E_n\}_{n<\omega}$, where each symbol binary, and consider an $L$-structure $M$ such that:
\begin{itemize}
\item $R$ is irreflexive and symmetric;
\item each $E_n$ is an equivalence relation with exactly two classes;
\item equivalences $\{E_n\}_{n<\omega}$ are independent in the sense that for each choice of $E_n$-class $C_n$, the family $\{C_n\}_{n<\omega}$ has the finite intersection property;
\item for pairwise distinct $a_i,b_j\in M$, where $i<l,j<m$, and any $n<\omega$, the set $\bigcap_{i<l}R(M,a_i)\cap \bigcap_{j<m}\lnot R(M,b_j)$ intersects each $\bigcap_{k\leqslant n}E_k$-class.
\end{itemize}
Note that in particular we have that $(M,R)$ is a model of the theory of the random graph. Let $T:=\mathrm{Th}(M)$. By standard arguments, $T$ has q.e. Hence, there is only one type in $S_1(T)$. Let $\C\models T$ be a monster model.

We will first show that Condition (C) (equivalently (B)) fails for $T$. More precisely, we will show that the Ellis group of the $\Aut(\C)$-flow $S_{\bar c,\Delta}(p)$ is infinite, where $\Delta:=\{R(x_0,y)\}$ and $p(y)\in S_y(T)$ is the unique $1$-type over $\emptyset$ (so, $S_{\bar c,\Delta}(p)= S_{\bar c,\Delta}(\C)$). Here, $\bar x$ is reserved for $\bar c$ and $x_0\in\bar x$ is any fixed variable. Consider the following global partial types:
$$\pi_n(\bar x):= \{(R(x_0,a)\leftrightarrow R(x_0,b))\leftrightarrow E_n(a,b)\mid a,b\in\C\}.$$
By randomness, each $\pi_n(\bar x)\cup\tp(\bar c)$ is consistent. So $X_n:=[\pi_n(\bar x)]$ is an $\Aut(\C)$-subflow of $S_{\bar c}(\C)$. By independence of the relations $E_n$, the $X_n$'s are pairwise disjoint. 
Let $\Phi: S_{\bar c}(\C)\to S_{\bar c,\Delta}(\C)$ be the $\Aut(\C)$-flow epimorphism given by restriction to $\Delta$-types. Note that $\Phi^{-1}[\Phi[X_n]]=X_n$. 
Therefore, $Y_n:=\Phi[X_n]$ are pairwise disjoint $\Aut(\C)$-subflows of $S_{\bar c,\Delta}(\C)$. 
Hence, for any $\eta\in\EL(S_{\bar c,\Delta}(\C))$, $\eta[Y_n] \subseteq Y_n$, and so $\Im(\eta)$ is infinite. Thus, Condition (D) does not hold, but we want to show that (C) fails.

Let $u\in\EL(S_{\bar c,\Delta}(\C))$ be an idempotent in a minimal left ideal $\M$, and let $q_n(\bar x)\in Y_n$ be in $\Im(u)$; so $u(q_n)=q_n$. For each $n<\omega$ there is $\sigma_n\in\Aut(\C)$ such that $\sigma_n$ fixes each $E_k$-class for $k<n$ and swaps the $E_n$-classes. Fix $a\in\C$ and $k<n$. For $\sigma\in\Aut(\C)$, note that $R(x_0,\sigma_n^{-1}(\sigma^{-1}(a)))\in q_k$ iff $R(x_0,\sigma^{-1}(a))\in q_k$ iff $R(x_0,\sigma_k^{-1}(\sigma^{-1}(a)))\notin q_k$. 
Thus $R(x_0,a)\in\sigma(\sigma_n(q_k))$ iff $R(x_0,a)\notin\sigma(\sigma_k(q_k))$ for every $\sigma\in\Aut(\C)$, so $R(x_0,a)\in u(\sigma_n(q_k))$ iff $R(x_0,a)\notin u(\sigma_k(q_k))$ holds as well. This means that $u\sigma_nu(q_k)= u\sigma_n(q_k)\neq u\sigma_k(q_k)=u\sigma_ku(q_k)$, so $u\sigma_nu\neq u\sigma_ku$. Since all $u\sigma_nu\in u\M$, $u\M$ is infinite and Condition (C) does not hold.

We now prove amenability of $T$. We have to find an invariant, Borel probability measure on $S_p(\C)$, for every $p(\bar y)\in S_{\bar y}(T)$ where $\bar y=(y_0,\dots,y_{m-1})$. As the reduct $T'$ of $T$ to the sublanguage $L':=\{E_n\}_{n < \omega}$ is stable, it is amenable by \cite[Corollary 2.21]{HKP}. We can choose $\C$ so that $\C':=\C_{\upharpoonright L'}$ is a monster model of $T'$. Let $p' \in S_{\bar y}(T')$ be the restriction of $p$ to the language $L'$. By amenability of $T'$, there is an $\Aut(\C')$-invariant,  Borel probability measure $\nu$ on $S_{p'}(\C')$. Consider $\Phi:S_{p}(\C)\to S_{p'}(\C')$ given by restriction. Let $q_R(\bar y):=\{R(y_i,a)\mid i<m,a\in\C\}$; by randomness and q.e., $q_R(\bar y)$ is a partial type consistent with $p(\bar y)$. It is $\Aut(\C)$-invariant, so $X:=[q_R(\bar y)]$ is an $\Aut(\C)$-subflow of $S_p(\C)$.  
By randomness and q.e., $\Phi_{\upharpoonright X}:X\to S_{p'}(\C')$ is an isomorphism. Define $\mu$ on Borel subsets of $S_p(\C)$ by $\mu(U):= \nu(\Phi[X\cap U])$. It is clear that $\mu$ is an $\Aut(\C)$-invariant, Borel probability measure on $S_p(\C)$.

The rest of the analysis of Example \ref{example equivalences with two classes} is devoted to the proof that Condition (A') holds for $T$, i.e.\ our goal is to show that the canonical Hausdorff quotient of the Ellis group of $T$ is profinite. Since the theory is binary, by Corollary \ref{corollary c to finite z m-ary}, we can work with $(\Aut(\C), S_1(\C))$ in place of $(\Aut(\C),S_{\bar c}(\C))$. 

Put $E_{\infty}:=\bigwedge_{i\in \omega} E_i$ and $H:=2^{\omega}$. For each $i\in \omega$, fix an enumeration $C_{i,0},C_{i,1}$ of the $E_i$-classes, and an enumeration $(C_{\epsilon})_{\epsilon\in H}$ of the $E_{\infty}$-classes. Let $X:=S_1(\C)$ and put 
$$X':=\bigcap_{a,b\in \C, E_{\infty}(a,b)} [R(x,a)\leftrightarrow R(x,b)].$$ 
It is clear that $X'$ is an $\Aut(\C)$-subflow of $X$.
By q.e., each $p\in X'$ is implied by the union of the following partial types for unique $\epsilon \in H$ and $\delta \in 2^{H}$:

\begin{itemize}
	\item $p_{\epsilon}(x):=\{x\in C_{i,\epsilon(i)} \mid i\in \omega\}$; and
	\item $q_{\delta}(x):=\{R(x,a) \mid a\in C_{\epsilon'},\delta(\epsilon')=1\}\cup\{\neg R(x,a) \mid a\in C_{\epsilon'},\delta(\epsilon')=0\}$.
\end{itemize}
Conversely, for each $\epsilon \in H$ and $\delta \in 2^{H}$ the union of the above partial types implies a type in $X'$. So $X'$ is topologically identified with the space $H\times 2^{H}$. 
Further, for $\Stab(X'):=\{\sigma\in \Aut(\C) \mid \sigma_{\upharpoonright X'}=\id_{X'}\}$, we see that $\Aut(\C)/\Stab(X')\cong H$ and the flow $(\Aut(\C)/\Stab(X'), X')$ can be identified with the flow $(H,H\times 2^{H})$ equipped with the following action: For $\sigma \in H$ and $(\epsilon,\delta)\in H\times 2^{H}$, $$\sigma(\epsilon,\delta):=(\sigma+\epsilon, \sigma\delta),$$ 
where $\sigma\delta(\epsilon'):=\delta(\epsilon'-\sigma)=\delta(\epsilon'+\sigma)$ for $\epsilon'\in H$.

\begin{claim*}\label{rem:eta_image_contained_in_X'}
There is $\eta\in \EL(X)$ whose image is contained in $X'$.
\end{claim*}
\begin{proofclaim}
For each $\epsilon\in H$, put $\Delta_{\epsilon}(x):=\{R(x,a)\leftrightarrow R(x,b):a,b\in C_{\epsilon}\}.$
First, we will show that for every $\epsilon \in H$ there is $\eta_{\epsilon}\in \EL(X)$ such that:
\begin{itemize}
	\item $\eta_{\epsilon}$ is the limit of a net of $\sigma\in \Aut(\C)$ fixing each $C_{\epsilon'}$, $\epsilon' \in H$, setwise (equivalently, this is a net of Shelah strong automorphism);
	\item $\Im(\eta_{\epsilon})\subseteq [\Delta_{\epsilon}]$.
\end{itemize}
For this, choose any representatives $a_{\epsilon'}$ of the classes $C_{\epsilon'}$, $ \epsilon' \in H$, and consider any $p_0,\ldots,p_{k-1}\in X$ and $a_0,\ldots,a_{n-1}\in C_{\epsilon}$. Since $(C_{\epsilon},R_{\upharpoonright_{C_{\epsilon}}})$ is a (monster) model of the theory of the random graph, 
by \cite[Proposition 3.3]{Cam} and q.e., 
there are $a_0',\ldots,a_n'\in C_{\epsilon}$ such that $\bar a:= a_0\ldots a_{n-1}\equiv a_0' \ldots a_{n-1}'=: \bar a'$ and for each $l<k$: 
$$p_l\models \bigwedge_{j,j'<n} R(x,a_j')\leftrightarrow R(x,a_{j'}').$$
By randomness, q.e., and saturation, we can find $a_{\epsilon'}'\in C_{\epsilon'}$ for all $\epsilon' \in H $ so that 
$$(a_j\mid j<n)^{\frown}(a_{\epsilon'}\mid \epsilon' \in H)\equiv (a_j' \mid j<n)^{\frown}(a_{\epsilon'}'\mid \epsilon' \in H).$$ 
So, there is $\sigma_{\bar p,\bar a}\in \Aut(\C)$ mapping the latter sequence above to the former one. Now,  it is enough to take $\eta_\epsilon$ to be an accumulation point of the net $(\sigma_{\bar p,\bar a})_{\bar p, \bar a}$.

Since the $\eta_\epsilon$'s are approximated by Shelah strong automorphisms,  for each $\epsilon,\epsilon' \in H$,  $\eta_{\epsilon}[[\Delta_{\epsilon'}]]\subseteq [\Delta_{\epsilon'}]$ holds.
For $\bar \epsilon =\{\epsilon_i\}_{i<n} \subseteq H$ put $\eta_{\bar \epsilon}:=\eta_{\epsilon_{n-1}}\circ \cdots \circ \eta_{\epsilon_0}$, where $\epsilon_0<\dots <\epsilon_{n-1}$ (say with respect to the lexicographic order).  Then, an accumulation point of the net $(\eta_{\bar \epsilon})_{\bar \epsilon}$ satisfies $\Im(\eta)\subseteq \bigcap [\Delta_{\epsilon}]=X'$. 
\end{proofclaim}

By the claim and Corollary \ref{corollary isomorphism of ellis groups having good eta}, the Ellis groups of the flows $(\Aut(\C),X)$ and $(\Aut(\C),X')$ are topologically isomorphic. By the discussion before the claim, the latter Ellis group is topologically isomorphic with the Ellis group of the flow $(H, H \times 2^H)$ which in turn is topologically isomorphic with the Ellis group of the flow $(H, \beta H)$ by Proposition \ref{prop:ellisgroup_associated_bernoullishift}(ii). As a conclusion, we get that the canonical Hausdorff quotient of the Ellis group of $T$ is topologically isomorphic with the canonical Hausdorff quotient of the Ellis group of $\beta H$, i.e.\ with the generalized Bohr compactification of $H$. Since $H$ is abelian (and so strongly amenable), Fact \ref{fact:bohr_compactification_tau_topology} implies that this generalized Bohr compactification coincides with the Bohr compactification of $H$, which is profinite by Fact \ref{fact:profinite_bohr_compactification}, as $H$ is abelian of finite exponent. So Condition (A') has been proved.
\end{example}

The next example is a modification of the previous one. It shows that (A'') does not imply (A') in Theorem \ref{THMsecond}. Moreover, the obtained theory is supersimple of SU-rank 1, so we see that even for supersimple theories the Ellis group need not be profinite, while $\Gal_{KP}(T)$ must be profinite by \cite{BPW}. Our theory is also amenable.

\begin{example}\label{example equivalences n classes} The idea is to modify the previous example by considering independent equivalence relations with growing finite number of classes in order to get at the end that the canonical Hausdorff quotient of the Ellis group of the resulting theory is topologically isomorphic with the Bohr compactification of the product $\prod_{n \geqslant 2} \mathbb{Z}/n\mathbb{Z}$ which is not profinite by Fact \ref{fact:profinite_bohr_compactification}, as $\prod_{n \geqslant 2}\mathbb{Z}/n\mathbb{Z}$ is not of finite exponent.

Consider the relational language $L'=\{R,E_n\}_{2\leqslant n<\omega}$, where each symbol binary, and consider an $L'$-structure $M'$ satisfying the same axioms as $M$ in Example \ref{example equivalences with two classes}, except each $E_n$ has exactly $n$ classes.
Expand $M'$ to the structure $M$ in the language $L:=L' \cup\{O_{n,k}\}_{\substack{2\leqslant n<\omega\\1\leqslant k<n}}$, where each $O_{n,k}$ is a binary relational symbol interpreted as follows: for each $2\leqslant n<\omega$ enumerate the $E_n$-classes as $C_{n,0},\dots,C_{n,n-1}$ and take the cyclic permutation $\rho_n:=(0,1,\dots,n-1) \in \Sym(n)$; set $O_{n,k}:= \bigcup_{i<n}C_{n,i}\times C_{n,\rho_n^k(i)}$. 
Let $T:=\mathrm{Th}(M)$. A standard back-and-forth argument shows that $T$ has q.e.; in particular, there is only one type in $S_1(T)$. Let $\C\models T$ be a monster model.

Put $\bar C_n:=(C_{n,0},\ldots,C_{n,n-1})$ and define $O_n(x_0,\dots,x_n)$ as $\bigwedge_{1 \leqslant k<n} O_{n,k}(x_0,x_k)$. Observe that $O_n(a_0,\dots,a_{n-1})$ holds iff $(a_0/E_n,\dots,a_{n-1}/E_n)$ belongs to the orbit of $\bar C_n$ under the action of $\mathbb{Z}/n\mathbb{Z}$ on $(M/E_n)^n$ defined as follows: for $(b_0,\ldots,b_{n-1})\in (M/E_n)^n$ and $l\in \mathbb{Z}/n\mathbb{Z}$, $l(b_0,\ldots,b_{n-1}):=(b_{\rho_n^l(0)},\ldots,b_{\rho_n^l(n-1)}).$

By q.e.\ and randomness, it easily follows that $\tp(a/A)$ does not fork over $\emptyset$ for any $a$ and $A$ with $a \notin A$. Hence, $T$ is supersimple of SU-rank 1. (The same is true in Example \ref{example equivalences with two classes}, but supersimplicity is more important here.) So, by \cite{BPW}, $\Gal_{KP}(T)$ is profinite, i.e.\ (A'') holds. Amenability of $T$ can be proved as in Example \ref{example equivalences with two classes}.

To show that (A') fails, one should follow the lines of the argument in Example \ref{example equivalences with two classes} with $H:=\prod_{n \geqslant 2} \mathbb{Z}/n\mathbb{Z}$. The fact the $O_n$'s are $L$-formulae is used to see that $(\Aut(\C)/\Stab(X'), X')$ can be identified with the flow $(H,H\times 2^{H})$. At the end, we get that the Ellis group of $T$ is topologically isomorphic with the Bohr compactification of $H$, which is not profinite by Fact \ref{fact:profinite_bohr_compactification}, i.e.\ Condition (A') fails.
\end{example}

By \cite[Proposition 4.2]{HKP}, if a theory $T$ is extremely amenable, then $\Gal_L(T)=\Gal_{KP}(T)$ is trivial. The next example (whose details are left to the reader) shows that amenability of $T$ does not even imply that $\Gal_{KP}(T)$ is profinite.

\begin{example}\label{example that amenability does not imply KP is profinite}
Let $N=(M,X,\cdot)$ be a two-sorted structure, where:
\begin{itemize}
	\item $M$ is a real closed field in the language $L_{or}(\mathbb{R})$ of ordered rings with constant symbols for all $r\in \mathbb{R}$;
	\item $\cdot : S^1\times X\rightarrow X$ is a strictly 1-transitive action of the circle group $S^1$ on $X$.
\end{itemize}
$N$ is clearly interpretable in $M$, hence $T:=\Th(N)$ has NIP. By \cite{GN}, we easily get that $\Gal_{KP}(T) \cong S^1$, so  $\Gal_{KP}(T)$ is not profinite. By \cite[Corollary 2.21]{HKP}, we know that a NIP theory is amenable iff $\emptyset$ is an extension base for forking. We leave as a non-difficult exercise to check that in $T$ every set is an extension base.
\end{example}

We have determined the relationships between most of the introduced properties. Let us discuss a few remaining questions.

\begin{question}
Is there an example for which (C) holds but (D) does not?
\end{question}

\begin{question}
\begin{enumerate}[label=(\roman*), align=left, leftmargin=*,labelsep=-4pt]
\item Is there an example for which (A) holds but (C) does not?
\item  Is there an example for which (A') holds but (A) does not?
\end{enumerate}
\end{question}

Example \ref{example equivalences with two classes} shows that (A') does not imply (B). So this example either witnesses that (A') does not imply (A), or that (A) does not imply (B); but we do not know which of these two lack of implications is witnessed by Example \ref{example equivalences with two classes}. By Remark \ref{remark uM and um/H zero dimensionality}(a), it witnesses the lack of the implication (A') $\implies$ (A) if and only if  the Ellis group of the theory in this example is not Hausdorff.

By Example \ref{example stable}, we know that having sep.\ fin.\ \edeerdeg\ does not imply \deerp; this is witnessed by any stable theory with a non-stationary type in $S_1(\emptyset)$. 
Is the converse true, i.e.\ does \deerp\ (equiv.\ extreme amenablity) imply sep.\ fin.\ \edeerdeg? Probably not. 
Example \ref{example equivalences with two classes} shows that \deecrp\ (equiv.\ amenability) does not imply sep.\ fin.\ \edeerdeg. In fact, Example \ref{example that amenability does not imply KP is profinite} shows that amenability does not even imply that $\Gal_{KP}$ is profinite (i.e.\ (A'')), whereas extreme amenability implies that $\Gal_{KP}$ is trivial by \cite[Proposition 4.2]{HKP}.

By Example \ref{exampre rdeg}(e), we know that the Ellis group there is profinite. Is it infinite? Can one compute it precisely, as we did for the theory from Example \ref{example R2 R4}? The same problem for all random hypergraphs, although here we know that the Ellis groups are finite by the paragraph after Example  \ref{exampre rdeg}.

\appendix
\section{Example \ref{example R2 R4}}\label{section example}

We now calculate the Ellis group from Example \ref{example R2 R4}. Recall that we consider the language $L:=\{R_2,R_4\}$, where $R_2$ is a binary and $R_4$ is a quaternary relational symbol. We consider the class of all finite $(2,4)$-hypergraphs, i.e.\ the class of all finite structures in $L$ for which $R_2$ and $R_4$ are irreflexive and symmetric. This is a Fra\"iss\'e class, and its Fra\"iss\'e limit $K$ is called the random $(2,4)$-hypergraph. The theory $T:=\mathrm{Th}(K)$ is $\aleph_0$-categorical, and so $\aleph_0$-saturated with q.e. 

We also consider the following expansion of $T$. Let $L^*=L\cup\{<\}$. Consider the class of all finite linearly ordered $(2,4)$-hypergraphs. This is a Fra\"iss\'e class, and its Fra\"iss\'e limit $K^*$ is called the ordered random $(2,4)$-hypergraph. The theory $T^*:=\mathrm{Th}(K^*)$ is $\aleph_0$-categorical, so $\aleph_0$-saturated with q.e. Then  $K^*_{\upharpoonright L} \cong K$, so  $T^*_{\upharpoonright L}=T$, and we may assume that $K^*_{\upharpoonright L}=K$.

By a straightforward back-and-forth argument one can easily see that there exists $\rho\in\Aut(K)$ such that $\rho$ reverses the order on $K^*$: $\rho[<]=>$.
Take $L^*_\rho:= L^*\cup\{\rho\}$ and the obvious expansion $K^*_\rho$ of $K^*$. Take a monster $\C^*_\rho$ of $\mathrm{Th}(K^*_\rho)$ such that $\C^*:=\C^*_{\rho\upharpoonright L^*}$ and $\C:=\C^*_{\rho\upharpoonright L}$ are monster models of $T^*$ and $T$, respectively. The interpretation of $\rho$ in $\C^*_\rho$, which will be also denoted by $\rho$, is an automorphism of $\C$ reversing $<$. Further on, we fix $\C^*$, $\C$, and $\rho$.

Since $T^*$ has \edeerp\ (even \eerp\ by Example \ref{example rp}(a)), by Theorem \ref{theorem characterization of dpeerp}, we can find $u^*\in\EL(S_{\bar c}(\C^*))$ with $\Im(u^*)\subseteq\Inv_{\bar c}(\C^*)$.
By Lemma \ref{lemma ellis semi epi bigger lang invariant}(iv), for a single variable $z$, we can find $u\in\EL(S_z(\C))$ such that $\Im(u)\subseteq\Inv^*_z(\C)$, where $\Inv^*_z(\C)$ is the set of all $\Aut(\C^*)$-invariant types in $S_z(\C)$. Moreover, we may assume that $u\in\J(\M)$ for a minimal left ideal $\M$ of $\EL(S_z(\C))$.

The elements of $\Inv^*_z(\C)$ are not hard to describe. Let $p(z)\in\Inv^*_z(\C)$. Since there is only one type in $S_z(T^*)$, $p(z)$ either contains $R_2(z,a)$ for all $a\in \C$ or $\lnot R_2(z,a)$ for all $a\in\C$. Note that $S_{\pi(\bar y)}(T^*)$, where $\bar y=(y_0,y_1,y_2)$ and $\pi(\bar y)=\{y_0\neq y_1\neq y_2\neq y_0\}$, is completely determined by restriction to $\{R_2,<\}$. Let us write $[\C]^3=O_0\sqcup O_1\sqcup O_2\sqcup O_3$, where $O_i$ is the set of all $\{a,b,c\}\in[\C]^3$ with exactly $i$-many $R_2$-edges on the set $\{a,b,c\}$. Note that by symmetry of $R_4$, either $R_4(z,a,b,c)\in p(z)$ for all $\{a,b,c\}\in O_0$, or $\lnot R_4(z,a,b,c)\in p(z)$ for all $\{a,b,c\}\in O_0$. The same holds for $O_3$. The sets $O_1$ and $O_2$ are more interesting. 
Write $O_1= O_1^{-}\sqcup O_1^{\circ}\sqcup O_1^{+}$, where for $\{a,b,c\}\in O_1$ with $b$ being $R_2$-unconnected with $a$ and $c$ we put:
$$\{a,b,c\}\in\left\{\begin{array}{cl}
O_1^- & \mbox{ if $b$ is minimal among }\{a,b,c\}\\
O_1^\circ & \mbox{ if $b$ is the middle one among }\{a,b,c\}\\
O_1^+ & \mbox{ if $b$ is maximal among }\{a,b,c\}\\
\end{array}\right..$$
Similarly, we write $O_2= O_2^-\sqcup O_2^\circ\sqcup O_2^+$, where the division is determined by the element $R_2$-connected to both other elements. By symmetry of $R_4$, we have either $R_4(z,a,b,c)\in p(z)$ for all $\{a,b,c\}\in O_1^-$, or $\lnot R_4(z,a,b,c)\in p(z)$ for all $\{a,b,c\}\in O_1^-$. The same holds for $O_1^\circ,O_1^+,O_2^-,O_2^\circ$ and $O_2^+$. By q.e., $p(z)$ is completely determined by the previous information. Moreover, by randomness, each described possibility occurs. Thus, we see that $\Inv^*_z(\C)$ has $2^9$ elements.

\begin{lemma}\label{lemma image of u are all inv}  $u(p)=p$ for all $p\in\Inv_z^*(\C)$. In particular, $\Im(u)=\Inv^*_z(\C)$.
\end{lemma}
\begin{proof}
It is enough to prove the first part. If $R_2(z,a)^\epsilon\in p$ for $\epsilon\in 2$ and $a\in\C$, then $R_2(z,a)^\epsilon \in\sigma(p)$ for every $\sigma\in\Aut(\C)$, so $R_2(z,a)^\epsilon \in u(p)$ holds as well. Similarly, if $R_4(z,a,b,c)^\epsilon\in p$ for $\epsilon\in 2$ and $\{a,b,c\}\in O_0$, then $R_4(z,a,b,c)^\epsilon\in \sigma(p)$ for every $\sigma\in\Aut(\C)$, so 
$R_4(z,a,b,c)^\epsilon\in u(p)$. The same holds for $O_3$.

Note that if we put $q=u(p)$, then, by idempotency, $u(q)= u^2(p)=u(p)=q$.

Let us focus on $O_1$. We say that $p\in\Inv^*_z(\C)$ is of type $(\epsilon_-,\epsilon_\circ,\epsilon_+)$, where $\epsilon_-,\epsilon_\circ,\epsilon_+\in 2$, if $R_4^{\epsilon_\star}(z,a,b,c)\in p$ for all $\{a,b,c\}\in O_1^\star$, for each $\star\in\{-,\circ,+\}$.

\begin{claim*} $p$ and $q=u(p)$  have the same type.
\end{claim*}
\begin{proofclaim}
If $p$ is of type $(0,0,0)$ or $(1,1,1)$, then $q$ is of the same type, as all automorphisms in these cases preserve the type of $p$, and hence $u$ preserves it, too.
Take $a,b,c,d$ such that $a<b<c<d$, $R_2(a,d)$, $R_2(b,c)$, and there are no other $R_2$-edges between $a,b,c,d$. Let $q$ be of type $(\epsilon_-,\epsilon_\circ,\epsilon_+)$. Then the formula $\phi(z,a,b,c,d)$:=
$$R_4^{\epsilon_-}(z,a,b,c)\land R_4^{\epsilon_\circ}(z,b,a,d)\land R_4^{\epsilon_\circ}(z,c,a,d)\land R_4^{\epsilon_+}(z,d,b,c)\; \in \; q=u(p)=u(q),$$
We can approximate $u$ by $\sigma\in\Aut(\C)$ such that $\phi(z,a,b,c,d)\in \sigma(p),\sigma(q)$, i.e.\ $\phi(z,a',b',c',d')\in p,q,$ where $\sigma(a',b',c',d')=(a,b,c,d)$. 


To see that $p$ and $q$ are of the same type, we do the following case analysis. E.g.\ if $q$ is of type $(1,0,0)$, then $p$ is of one of the types $(1,0,0)$, $(1,1,0)$ or $(1,0,1)$, as $R_4(z,a',b',c')\in q$ implies $a'<b',c'$, so $R_4(z,a',b',c')\in p$. If $p$ is of type $(1,1,0)$, then $\lnot R_4(z, b',a',d')\land\lnot R_4(z,d',b',c')\in p$ implies $b'>a',d'$ and $d'>b',c'$ which is not possible. Similarly, one eliminates that $p$ is of type $(1,0,1)$. 

A similar analysis works when $q$ is of type $(0,0,1)$ and $(0,1,0)$, and also shows that $p$ is not of type $(1,-,-)$ nor $(-,-,1)$ when $q$ is of type $(0,0,0)$. Exceptionally, to eliminate the case when $q$ is of type $(0,0,0)$ and $p$ is of type $(0,1,0)$ a different trick is required.

Choose elements $a_0,\dots,a_4$ such that $R_2(a_i,a_{i+1})$ for all $i<5$ ($+$ is modulo 5), and there are no other $R_2$-edges between them. Then $\bigwedge_{i<5}\lnot R_4(z,a_i,a_{i+2},a_{i+3})\in q=u(p)=u(q)$, as $q$ is of type $(0,0,0)$. As above, we can approximate $u$ by $\sigma$ and find a copy $(a_0',a_1',a_2',a_3',a_4')$ of $(a_0,a_1,a_2,a_3,a_4)$ such that both $p$ and $q$ contain $\bigwedge_{i<5}\lnot R_4(z,a_i',a_{i+2}',a_{i+3}')$. Since $p$ is of type $(0,1,0)$, we get that $a_i'<a_{i+2}',a_{i+3}'$ or $a_i'>a_{i+2}',a_{i+3}'$, for all $i<5$. But it is easy to see that this is impossible (just looking at $<$).

The remaining cases are completely dual by interchanging $0$ and $1$, and $R_4$ and $\lnot R_4$ in the previous cases.
\end{proofclaim}

The above analysis for $O_1$ applies to $O_2$ by interchanging all $R_2$-edges and $R_2$-non-edges.
The lemma is proved.
\end{proof}

We can now easily see that $u\M$ is not trivial. Take $p,q\in\Inv_z^*(\C)$ such that $p(z)$ implies that $z$ is not $R_2$-connected to anything and only $R_4$-connected to $O_1^-$, and $q(z)$ implies that $z$ is not $R_2$-connected to anything and only $R_4$-connected to $O_1^+$. Note that $u\rho u\in u\M$ and that $\rho[O_1^-]=O_1^+$ and vice versa. Thus, by Lemma \ref{lemma image of u are all inv}, $u\rho u(p)= u\rho(p)=u(q)=q$, so $u\rho u\neq u$ as $u(p)=p$. We will see that $u\M=\{u,u\rho u\}$, but this will require more work, involving applications of contents.

\begin{lemma}\label{lemma uM two element} $u\M=\{u,u\rho u\}$, so $u\M\cong\mathbb Z/2\mathbb Z$.
\end{lemma}
\begin{proof}
Since $\Im(u) \subseteq \Inv_z^*(\C)$, it is enough to prove that for any $\eta\in u\M$: either  $\eta(p)=u(p)$ for all $p\in \Inv_z^*(\C)$, or $\eta(p)=u\rho u(p)$ for all $p\in\Inv_z^*(\C)$. We define the notion of $O_1$-type and $O_2$-type of $p \in \Inv_z^*(\C)$ as in the proof of Lemma \ref{lemma image of u are all inv}. Fix $\eta\in u\M$; $p$ will always range over $\Inv_z^*(\C)$.

\begin{claim*} If the $O_1$-type of $p$ is $(\epsilon_-,\epsilon_\circ,\epsilon_+)$, then the $O_1$-type of $\eta(p)$  is $(\epsilon_-,\epsilon_\circ,\epsilon_+)$ or $(\epsilon_+,\epsilon_\circ,\epsilon_-)$. 
The same holds for $O_2$-types.
\end{claim*}
\begin{proofclaim}
Let us first deal with $O_1$-types.
If $\epsilon_-=\epsilon_\circ=\epsilon_+$, then each automorphism preserves the $O_1$-type of $p$, hence $\eta$ preserves it, too, and we are done. 

Take elements $a_0,a_1,a_2,a_3$ such that $R_2(a_0,a_2)$ and $R_2(a_1,a_3)$, and there are no other $R_2$-edges between them. Put $q(y_0,y_1,y_2,y_3):=\tp^L (a_0,a_1,a_2,a_3)$. For a realization $\bar b$ of $q$ we will say that it is of type:
\begin{enumerate}[label=\Alph*:, align=right, leftmargin=*]
\item if $\min(\bar b)$ is $R_2$-connected to $\max(\bar b)$ ($\min$ and $\max$ are taken in $\C^*$);
\item if $\min(\bar b)$ and $\max(\bar b)$ are not $R_2$-connected and $\min(\bar b)^\star<\max(\bar b)^\star$, where $b_i^\star=b_{i+2}$, so it is the element to which $b_i$ is $R_2$-connected;
\item if $\min(\bar b)$ and $\max(\bar b)$ are not $R_2$-connected and $\max(\bar b)^\star<\min(\bar b)^\star$.
\end{enumerate}

For a type $p$, let $\delta_0,\delta_1,\delta_2,\delta_3\in \{0,1\}$ be the unique numbers such that the formula $\bigwedge_{i<4}R_4^{\delta_i}(z,b_i,b_{i+1},b_{i+2})$ belongs to $p$. Note that they depend on $p$ and $\bar b\models q$. Denote this formula by $\phi_{p,\bar b}(z,\bar b)$. In the following table, we calculate $\sum_{i<4}\delta_i$ depending on the $O_1$-type of $p$ and the type of ordering on $\bar b$:

\begin{center}
\begin{tabular}{|c||c|c|c|c|c|c|c|c|}\hline
& (0,0,0) & (1,0,0) & (0,1,0) & (0,0,1) & (1,1,0) & (1,0,1) & (0,1,1) & (1,1,1) \\\hline\hline
A/C & 0 & 1 & 2 & 1 & 3 & 2 & 3 & 4\\\hline
B & 0 & 2 & 0 & 2 & 2 & 4 & 2 & 4\\\hline
\end{tabular}
\end{center}

Recall that by Fact \ref{fact contents ellis}, $\ct(\eta(p))\subseteq \ct(p)$. By our choice, $(\phi_{p,\bar b}(z,\bar y),q(\bar y))\in \ct(p)$ for every $\bar b\models q$. Consider the following cases.
 
Case 1. $p$ is of $O_1$-type $(1,0,0)$ or $(0,0,1)$. If $\eta(p)$ is of type $(0,0,0)$, $(0,1,0)$, $(1,0,1)$ or $(1,1,1)$, then choose $\bar b\models q$  of type B. Since $(\phi_{\eta(p),\bar b}(z,\bar y),q(\bar y))\in\ct(\eta(p))$, this pair belongs to $\ct(p)$ as well. But, by the table, this is not possible, since $\phi_{\eta(p),\bar b}$ has either $0$ or $4$ positive occurrences of $R_4$, whereas this does not happen in $\varphi_{p,\bar b'}$ for any $\bar b'\models q$ if $p$ is of type $(1,0,0)$ or $(0,0,1)$. Similarly, if $\eta(p)$ is of type $(1,1,0)$ or $(0,1,1)$, by choosing $\bar b\models q$ of type A, we have $(\phi_{\eta(p),\bar b}(z,\bar y),q(\bar y))\in\ct(\eta(p))\subseteq\ct(p)$, but since $\phi_{\eta(p),\bar b}$ has $3$ positive occurrences of $R_4$, we again cannot find $\bar b'\models q$ such that $\phi_{\eta(p),\bar b}(z,\bar b')\in p$.
So, $\eta(p)$ is either of $O_1$-type $(1,0,0)$ or $(0,0,1)$.

Case 2. $p$ is of $O_1$-type $(0,1,0)$. 
By similar considerations to those in Case 1, we obtain that the only possibilities for the $O_1$-type of $\eta(p)$ are $(0,0,0)$ and $(0,1,0)$.
The case when $\eta(p)$ is of $O_1$-type $(0,0,0)$ requires a different trick, but this can be done in the same way as in the proof of Lemma \ref{lemma image of u are all inv}. So $\eta(p)$ is of $O_1$-type $(0,1,0)$.
 
The remaining cases are dual. Interchanging $R_2$ edges and $R_2$-non-edges, we obtain the claim for $O_2$-types.
\end{proofclaim}

Note that if $p$ is of $O_1$-type $(\epsilon_-,\epsilon_\circ,\epsilon_+)$, then $\rho(p)$ is of $O_1$-type $(\epsilon_+,\epsilon_\circ,\epsilon_-)$, and similarly for $O_2$-types. Therefore, the previous claim says that $\eta(p)$ has the same $O_1$-type [$O_2$-type] as $p$ or as $\rho(p)$. 

\begin{claim*} Either for every $p$ the $O_1$-types of $p$ and $\eta(p)$ are equal, or for every $p$ the $O_1$-types of $\rho(p)$ and $\eta(p)$ are equal.
The same holds for $O_2$-types.
\end{claim*}
\begin{proofclaim}
Consider the case of $O_1$-types (the case of $O_2$-types follows by interchanging all $R_2$-edges and $R_2$-non-edges).
Suppose not. Then we have types $p,p'$ with $O_1$-types $(\epsilon_-,\epsilon_\circ,\epsilon_+)$ and $(\epsilon_-',\epsilon_\circ',\epsilon_+')$ such that $\eta(p)$ and $\eta(p')$ are of $O_1$-types $(\epsilon_-,\epsilon_\circ,\epsilon_+)$ and $(\epsilon_+',\epsilon_\circ',\epsilon_-')$, respectively, where  $\epsilon_-\neq\epsilon_+$ and $\epsilon_-'\neq\epsilon_+'$. We have two cases.

Case 1. $(\epsilon_-,\epsilon_+)= (\epsilon_-',\epsilon_+')$. 
Consider 
$a_i$, $i<4$, such that $R_2(a_0,a_2)$, $R_2(a_1,a_3)$, and there are no other $R_2$-edges between them, and $a_0,a_2<a_1,a_3$. Let $q(\bar y)=\tp^L(a_0,a_1,a_2,a_3)$. Choose $\delta_0,\delta_1,\delta_2,\delta_3$ such that $\bigwedge_{i<4}R_4^{\delta_i}(z,a_i,a_{i+1},a_{i+2})\in \eta(p)$. 
Then $\bigwedge_{i<4}R_4^{1-\delta_i}(z,a_i,a_{i+1},a_{i+2})\in\eta(p')$, so:
$$\left(\bigwedge_{i<4}R_4^{\delta_i}(z,y_i,y_{i+1},y_{i+2}), \bigwedge_{i<4}R_4^{1-\delta_i}(z,y_i,y_{i+1},y_{i+2}),q(\bar y)\right)\in\ct(\eta(p),\eta(p')).$$
By Fact \ref{fact contents ellis}, this triple belongs to $\ct(p,p')$, so $\bigwedge_{i<4}R_4^{\delta_i}(z,b_i,b_{i+1},b_{i+2})\in p$ and $\bigwedge_{i<4}R_4^{1-\delta_i}(z,b_i,b_{i+1},b_{i+2})\in p'$ for some $\bar b\models q$. Choose $i$ such that $b_{i+1}=\min(\bar b)$. 
Then $R_4^{\delta_i}(z,b_i,b_{i+1},b_{i+2})\in p$ implies that $\epsilon_-=\delta_i$, and $R_4^{1-\delta_i}(z,b_i,b_{i+1},b_{i+2})\in p'$ implies that $\epsilon_-'=1-\delta_i$. Therefore, $\epsilon_-\neq \epsilon_-'$; a contradiction.

Case 2. $(\epsilon_-,\epsilon_+)\neq (\epsilon_-',\epsilon_+')$. Then $(\epsilon_-,\epsilon_+)= (\epsilon_+',\epsilon_-')$, so we reduce this to Case 1 by considering $\eta(p)$ and $\eta(p')$ instead of $p$ and $p'$, and $\eta^{-1}$ (computed in $u\M$)  instead of $\eta$ (note that $\eta^{-1}(\eta(p))= u(p) =p$ and $\eta^{-1}(\eta(p'))=u(p')= p'$ by Lemma \ref{lemma image of u are all inv}).
\end{proofclaim}



We finally prove:

\begin{claim*} Either for every $p$ the $O_1$-types of $p$ and $\eta(p)$ are equal and the $O_2$-types of $p$ and $\eta(p)$ are equal, or for every $p$ the $O_1$-types of $\rho(p)$ and $\eta(p)$ are equal and the $O_2$-types of $\rho(p)$ and $\eta(p)$ are equal.
\end{claim*}
\begin{proofclaim}
Let $p$ have both the $O_1$-type and the $O_2$-type equal to $(1,0,0)$. If the claim fails, then, by the previous two claims, the $O_1$-types of $p$ and $\eta(p)$ are equal and the $O_2$-types of $\rho(p)$ and $\eta(p)$ are equal, or the $O_1$-types of $\rho(p)$ and $\eta(p)$ are equal and the $O_2$-types of $p$ and $\eta(p)$ are equal.

So, assume first that $\eta(p)$ has $O_1$-type $(1,0,0)$ but $O_2$-type $(0,0,1)$. Consider $a_0,a_1,a_2,a_3$ such that $R_2(a_0,a_1)$, $R_2(a_0,a_3)$ and there are no other $R_2$-edges between them, and $a_0>a_1>a_2>a_3$; set $q(\bar y)=\tp^L(a_0,a_1,a_2,a_3)$. Then $R_4(z,a_2,a_0,a_1)\land\lnot R_4(z,a_2,a_0,a_3)\land R_4(z,a_0,a_1,a_3)\in\eta(p)$. But, by Fact \ref{fact contents ellis}, $\ct(\eta(p))\subseteq\ct(p)$. Hence, we can find $\bar b\models q$ such that $R_4(z,b_2,b_0,b_1)\land\lnot R_4(z,b_2,b_0,b_3)\land R_4(z,b_0,b_1,b_3)\in p$. Since the $O_1$-type and the $O_2$-type of $p$ are both $(1,0,0)$, this implies $b_2<b_0,b_1$, \ $b_0<b_1,b_3$, but $b_2$ is not less than both $b_0$ and $b_3$. Clearly, this is not possible.

If $\eta(p)$ has $O_1$-type $(0,0,1)$ but $O_2$-type $(1,0,0)$, the proof is dual by reversing the order on $\{a_0,a_1,a_2,a_3\}$.
\end{proofclaim}

We are ready to finish the proof of the lemma. If $p$ contains $R_2^\epsilon(z,a)$ for some $\epsilon\in 2$ and all $a\in\C$, then $\sigma(p)$ contains it, too, and so does $\eta(p)$. Similarly, if $p$ contains $R_4^\epsilon(z,a,b,c)$ for some $\epsilon\in 2$ and all $\{a,b,c\}\in O_0$ [resp. $\in O_3$], then $\eta(p)$ contains it, too. Thus, the restrictions of $\eta(p)$, $p$, and $\rho(p)$ to these formulae coincide for every $p \in \Inv_z^*(\C)$. 
By the previous claim, either for every $p \in \Inv_z^*(\C)$ the restrictions of $\eta(p)$ and $p$ to the formulae $R_4^\epsilon(z,a,b,c)$ for $\epsilon\in 2$ and $\{a,b,c\}\in O_1 \cup O_2$ coincide, or for every $p\in \Inv_z^*(\C)$ the restrictions of $\eta(p)$ and $\rho(p)$ to these formulae coincide. Therefore, either for every $p\in \Inv_z^*(\C)$ we have $\eta(p)=p=u(p)$, or for every $p\in \Inv_z^*(\C)$ we have $\eta(p)=\rho(p)=u\rho u(p)$. But this means that  either $\eta=u$, or $\eta=u\rho u$.
\end{proof}

\begin{proposition} The Ellis group of $(\Aut(\C),S_{\bar c}(\C))$ is $\mathbb Z/2\mathbb Z$.
\end{proposition}
\begin{proof}

Take $u^*\in\EL(S_{\bar c}(\C^*))$  such that $\Im(u^*)\subseteq\Inv_{\bar c}(\C^*)$, as $T^*$ has \edeerp.
By Corollary \ref{lemma ellis semi epi bigger lang invariant}(ii), there is $u' \in \EL(S_{\bar c}(\C))$ with $\Im(u') \subseteq\Inv_{\bar c}^*(\C)$. Furthermore, we may assume that $u'$ is an idempotent in a minimal left ideal $\M'$ of $\EL(S_{\bar c}(\C))$. By Lemma \ref{lemma: first reduction}  (having in mind the natural identification of $S_{\bar d}(\C)$ with $S_{\bar c}(\C)$), we have the flow and semigroup epimorphism $\Phi:\EL(S_{\bar c}(\C))\to \EL(S_z(\C))$ (where $z$ is a single variable) given by:
$$\Phi(\eta)(p(z)) =\hat\eta(p(z)):= \eta(q(\bar x))_{\upharpoonright x'}[x'/z],$$
where $p(z)\in S_z(\C)$, $x'\in\bar x$, and $q(\bar x)\in S_{\bar c}(\C)$ are such that $q(\bar x)_{\upharpoonright x'}[x'/z]=p(z)$. By Fact \ref{fact epi ellis}, $u:=\Phi(u')$ is an idempotent in the minimal left ideal $\M:=\Phi[\M']$ of $\EL(S_z(\C))$, and $\Phi_{\upharpoonright u'\M'}:u'\M'\to u\M$ is  a group epimorphism. By the formula above, $\Im(u)\subseteq\Inv_{z}^*(\C)$, so by Lemma \ref{lemma image of u are all inv} and Lemma \ref{lemma uM two element}, we have that $\Im(u)=\Inv_z^*(\C)$ and $u\M=\{u,u\rho u\}$ has two elements. So it remains to show that $\ker(\Phi_{\upharpoonright u'\M'})$ is trivial.

Let $\eta\in u'\M'$ be such that $\hat{\eta}=u$. It is enough to prove that $\eta(q)=q$ for all $q\in\Im(\eta)=\Im(u')$ (recall that all such $q$'s are $\Aut(\C^*)$-invariant). 

If $R_2^\epsilon(x_i,a)\in\eta(q)$, then $R_2^\epsilon(x_i,\sigma(a))\in q$ for some $\sigma\in\Aut(\C)$, but then $R_2^\epsilon(x_i,a)\in q$ by $\Aut(\C^*)$-invariance of $q$ (as there is only one type in $S_1(T^*)$).
Similarly, if $R_4^\epsilon(x_i,x_j,x_k,a)\in\eta(q)$, then $R_4^\epsilon(x_i,x_j,x_k,a)\in q$ by invariance. If $R_4^\epsilon(x_i,x_j,a,b)\in\eta(q)$, by symmetry of $R_4$ and invariance of $q$, the conclusion is the same: $R_4^\epsilon(x_i,x_j,a,b)\in q$.
Let us consider $R_4^\epsilon(x_i,a,b,c)\in\eta(q)$. Let $p(z)= q(\bar x)_{\upharpoonright x_i}[x_i/z]$; note that $p(z)\in\Inv_z^*(\C)$, so $p(z)\in\Im(u)$. Then $R_4^\epsilon(z,a,b,c)\in\hat\eta(p)= u(p)= p= q_{\upharpoonright x_i}[x_i/z]$. Thus $R_4^\epsilon(x_i,a,b,c)\in q$.

By q.e., we conclude that  $\eta(q)\subseteq q$, so $\eta(q)=q$, and we are done.
\end{proof}

\section*{Acknowledgments}
The first author would like to thank Pierre Simon for inspiring discussions on random hypergraphs. The authors are also grateful to Tomasz Rzepecki for some useful suggestions which helped us to complete our analysis of Example \ref{example R2 R4} and for suggesting Example \ref{example R2 R4 Ps}.

\bibliographystyle{plain}

\end{document}